\documentclass[reqno]{amsart}
\usepackage[scale=0.75, centering, headheight=14pt]{geometry}
\usepackage[latin1]{inputenc}
\usepackage[T1]{fontenc}
\usepackage{lmodern}
\usepackage[english]{babel}
\usepackage{stmaryrd}
\usepackage{tikz}
\usepackage{bbm}
\usepackage{amsmath,amssymb,amsfonts,amsthm}
\usepackage{mathtools,accents}
\usepackage{mathrsfs}
\usepackage{xfrac}
\usepackage{array} 
\usepackage{aliascnt}
\usepackage{booktabs} 
\usepackage{array} 

\usepackage{verbatim} 
\usepackage{subfig} 
\usepackage{mathrsfs}
\usepackage{mathrsfs, dsfont}
\usepackage{amssymb}
\usepackage{amsthm}
\usepackage{amsmath,amsfonts,amssymb,esint}
\usepackage{graphics,color}
\usepackage{enumerate}
\usepackage{mathtools,centernot}

\usepackage{microtype}
\usepackage{paralist} 
\usepackage{cases}
\usepackage[initials,alphabetic]{amsrefs}
\allowdisplaybreaks

\usepackage{braket}
\usepackage{bm}

\usepackage[citecolor=blue,colorlinks]{hyperref}
\addto\extrasenglish{}

\usepackage{enumerate}
\usepackage{xcolor}

\usepackage{aliascnt}

\makeatletter
\def\newaliasedtheorem#1[#2]#3{
  \newaliascnt{#1@alt}{#2}
  \newtheorem{#1}[#1@alt]{#3}
  \expandafter\newcommand\csname #1@altname\endcsname{#3}
}
\makeatother

\theoremstyle{plain}
\newtheorem{theorem}{Theorem}[section]
\newaliasedtheorem{lemma}[theorem]{Lemma}
\newaliasedtheorem{prop}[theorem]{Proposition}
\newaliasedtheorem{claim}[theorem]{Claim}
\newaliasedtheorem{corollary}[theorem]{Corollary}

\newtheorem{maintheorem}{Theorem}

\theoremstyle{definition}
\newaliasedtheorem{definition}[theorem]{Definition}
\newaliasedtheorem{example}[theorem]{Example}

\theoremstyle{remark}
\newaliasedtheorem{remark}[theorem]{Remark}
\newaliasedtheorem{ex}[theorem]{Example}

\numberwithin{equation}{section}

\def\eps{\sigma}

\def\R{\mathbb R}
\def\N{{\mathbb N}}
\def\Z{{\mathbb Z}}
\def\T{{\mathbb T}}

\DeclareMathOperator{\supp}{supp}

\DeclareMathOperator{\loc}{loc}
\DeclareMathOperator{\swap}{swap}
\DeclareMathOperator{\initial}{in}

\DeclareMathOperator{\dist}{d}

\newcommand{\XX}{{\mbox{\boldmath$X$}}}
\newcommand{\WW}{{\mbox{\boldmath$W$}}\mkern-4mu}

\title[Anomalous Dissipation and Lack of Selection 
in the Obukhov-Corrsin Theory]{Anomalous Dissipation and Lack of Selection \\ 
in the Obukhov-Corrsin Theory of Scalar Turbulence}
\author[M. Colombo, G. Crippa, M. Sorella]{Maria Colombo, Gianluca Crippa \and Massimo Sorella}

\address{Maria Colombo  
\hfill\break EPFL B, Station 8, CH-1015 Lausanne, Switzerland}
\email{maria.colombo@epfl.ch}

\address{Gianluca Crippa
\hfill\break Departement Mathematik und Informatik, Universit\"at Basel, Spiegelgasse 1, CH-4051 Basel, Switzerland}
\email{gianluca.crippa@unibas.ch}

\address{Massimo Sorella
\hfill\break EPFL B, Station 8, CH-1015 Lausanne, Switzerland}
\email{massimo.sorella@epfl.ch}

\begin{document}

\begin{abstract}
The Obukhov-Corrsin theory of scalar turbulence~\cite{Obukhov,Corrsin} advances quantitative predictions on passive-scalar advection in a turbulent regime and can be regarded as the analogue for passive scalars of Kolmogorov's K41 theory of fully developed turbulence~\cite{kolmo}. The scaling analysis of Obukhov and Corrsin from~1949-1951 identifies a critical regularity threshold for the advection-diffusion equation and predicts anomalous dissipation in the limit of vanishing diffusivity in the supercritical regime. In this paper we provide a fully rigorous mathematical validation of this prediction by constructing a velocity field and an initial datum {such that the unique bounded solution of the advection-diffusion equation} is bounded uniformly-in-diffusivity { within any fixed} supercritical Obukhov-Corrsin regularity regime { while also exhibiting} anomalous dissipation. Our approach relies on a fine quantitative analysis of the interaction between the spatial scale of the solution and the scale of the Brownian motion which represents diffusion in a stochastic Lagrangian setting. This provides a direct Lagrangian approach to anomalous dissipation which is fundamental in order to get detailed insight on the behavior of the solution. Exploiting further this approach, we also show that for a velocity field in $C^\alpha$ of space and time (for an arbitrary $0 \leq \alpha < 1$) neither vanishing diffusivity nor regularization by convolution provide a selection criterion for bounded solutions of the advection equation. This is motivated by the fundamental open problem of the selection of solutions of the Euler equations as vanishing-viscosity limit of solutions of the Navier-Stokes equations and provides a complete negative answer in the case of passive advection.
\end{abstract}

\maketitle

\section{Introduction}

The advection of a passive scalar $\vartheta = \vartheta(t,x) \in \R$ by a divergence-free velocity field $u = u(t,x) \in \R^n$ on the $n$-dimensional torus $\T^n \cong \R^n / \Z^n $ is described by the equation
\begin{equation}\label{e:intro:advection}
\partial_t \vartheta + u \cdot \nabla \vartheta = 0 \,.
\end{equation}
The basic physical example is the advection of the temperature, which is assumed to have no influence on the given background flow and therefore to be passively rearranged by it. This is in contrast with the case of active scalars, as for instance the vorticity of a two-dimensional fluid which is directly coupled to the velocity by the curl relation. For regular solutions the incompressibility of the flow guarantees the conservation in time of all rearrangement-invariant norms of solutions of~\eqref{e:intro:advection}, in particular of the spatial $L^2$ norm: for all times $t > 0$ we have $\| \vartheta(t,\cdot) \|_{L^2(\T^n)} = \| \vartheta(0,\cdot) \|_{L^2(\T^n)}$.

\medskip

The Obukhov-Corrsin theory of scalar turbulence~\cite{Obukhov,Corrsin} advances quantitative predictions on the energy spectrum and the structure functions for the passive scalar in a regime of turbulent advection and can be regarded as the analogue for passive scalars of Kolmogorov's K41~\cite{kolmo} theory of fully developed turbulence for the velocity field. In the idealized regime of small viscosity, within K41 the velocity field is predicted to satisfy Kolmogorov's two-thirds law, which reflects into Onsager's~\cite{onsager} critical regularity of order $\sfrac{1}{3}$ for the velocity field (see~\cite{frisch,eyinksurvey} for the general context). Based on dimensional analysis, in 1949-1951 Obukhov and Corrsin independently predicted the same scaling in the idealized regime of small diffusivity~$\kappa>0$ for solutions of the advection-diffusion equation with a turbulent advecting flow
\begin{equation}\label{e:intro:advdiff}
\partial_t \vartheta_\kappa + u \cdot \nabla \vartheta_\kappa = \kappa \Delta \vartheta_\kappa \,.
\end{equation}
More precisely, the $q$-th order structure function of the passive scalar $S_q^\vartheta(\ell) := \langle | \delta_\ell \vartheta|^q \rangle$ is predicted to exhibit as a function of the spatial increment $\ell$ a scaling $\sim \ell^{\sfrac{q}{3}}$ in the so-called inertial-convective range, for large Reynolds number turbulence and for a diffusivity at least of the order of magnitude of the viscosity. Ideally, this can be interpreted as a regularity of order $\sfrac{1}{3}$ for the passive scalar uniformly for small but strictly positive diffusivity (see also~\cite{sree1,sree2}).
{ The observations in \cite{E96} and the numerical simulations in \cite{IS18} suggest the occurrence of intermittency in passive-scalar turbulence, analogously to the case of the velocity field in the Navier-Stokes equations: intermittency entails corrections to the structure functions (which therefore may depend non-linearly on the integrability exponent $q$) and has been validated to a large extend both experimentally and numerically \cite{IS20}.}

\medskip

{ A major tenet of turbulence theory is the persistence of dissipation in the limit of vanishing viscosity (for the velocity) \cite{kolmo} or vanishing diffusivity (for the passive scalar) \cite{sree1,sree2,S00,D05}.}
 Testing~\eqref{e:intro:advdiff} against the solution $\vartheta_\kappa$ leads to a dissipation term and a nonlinear term, namely
\begin{equation}\label{e:intro:testing}
\kappa \int_0^T \| \nabla \vartheta_\kappa(s,\cdot) \|^2_{L^2(\T^n)} \, ds
\quad\text{ and }\quad
\int_0^T \! \! \int_{\T^n} \left(u \cdot \nabla \vartheta_\kappa \right) \vartheta_\kappa \, dx\, ds \sim \int_0^T \! \!  \int_{\T^n} \nabla^{\alpha} u \left( \nabla^{\sfrac{(1-\alpha)}{2}} \vartheta_\kappa \right)^2 \, dx\, ds
\end{equation}
(we will discuss in the next paragraph the role of the parameter $0\leq\alpha\leq 1$ in the formal rewriting of the nonlinear term) and to the energy balance
\begin{equation}\label{e:intro:balance}
\int_{\T^n} |\vartheta_\kappa (t,\cdot)|^2 \, dx + 2 \kappa \int_0^t \| \nabla \vartheta_\kappa(s,\cdot) \|^2_{L^2(\T^n)} \, ds 
= \int_{\T^n} |\vartheta_\kappa(0,\cdot)|^2 \, dx 
\qquad \text{for $t\in[0,T]$.}
\end{equation}
Anomalous dissipation for the passive scalar amounts to 
\begin{equation}\label{e:intro:anomalous}
\limsup_{\kappa\to 0} \; \kappa \int_0^T \| \nabla \vartheta_\kappa(s,\cdot) \|^2_{L^2(\T^n)} \, ds > 0\,,
\end{equation}
which (by lower semicontinuity of the $L^2$ norm under weak convergence, and up to subsequences) implies convergence of $\vartheta_\kappa$ to a solution $\vartheta$ of~\eqref{e:intro:advection} for which $\| \vartheta(T,\cdot) \|_{L^2(\T^n)} < \| \vartheta(0,\cdot) \|_{L^2(\T^n)}$. The presence of anomalous dissipation in particular requires a blow up as $\kappa\to 0$ of the first-order derivative of the passive scalar. The argument by Obukhov and Corrsin leading to the scaling $\sim \ell^{\sfrac{q}{3}}$ for the $q$-th order structure function predicts a sharp, uniform-in-diffusivity bound on the regularity of order $\sfrac{1}{3}$ for the passive scalar.

\medskip

The formal rewriting of the nonlinear term as in~\eqref{e:intro:testing} reveals the generalized regime of critical regularity for the problem. Assuming that the divergence-free velocity field $u$ belongs to $L^p([0,T];C^\alpha(\T^n))$ for $1 \leq p \leq \infty$ and $0\leq\alpha\leq 1$, we define implicitly $p^\circ \! \in [2,+\infty]$ and $\alpha^\circ \!\in [0,\sfrac{1}{2}]$ by the so-called Yaglom's~\cite{yaglom} relation
\begin{equation}\label{e:intro:yaglom}
\frac{1}{p} + \frac{2}{p^\circ} = 1 
\qquad \text{ and } \qquad
\alpha + 2 \alpha^\circ = 1 \,.
\end{equation}

{
Obukhov-Corrsin criticality corresponds to the passive scalar $\vartheta_\kappa$ belonging to $L^{p^\circ}([0,T];C^{\alpha^\circ}(\T^n))$ uniformly with respect to the diffusivity~$\kappa>0$. The scaling analysis by Obukhov and Corrsin  reveals very strong similarities to \cite{kolmo}, for both the scaling of structure functions and the energy spectrum. Therefore, Obukhov and Corrsin conjectured similar behaviors for the passive scalar as for the velocity for the Navier-Stokes. With the function-space language we just introduced, this means in analytical terms that for $u\in L^p([0,T];C^\alpha(\T^n))$ and $\vartheta \in L^{p^\circ}([0,T];C^{\beta}(\T^n))$, Obukhov and Corrsin predict:
\begin{itemize}
\item[(1)] In the subcritical Obukhov-Corrsin regime $\beta > \alpha^\circ$ the advection equation~\eqref{e:intro:advection} has a unique solution, the conservation $\| \vartheta(t,\cdot) \|_{L^2(\T^n)} = \| \vartheta(0,\cdot) \|_{L^2(\T^n)}$ holds for {every} 
  $t \in[0,T]$, and there is no anomalous dissipation. 
 \item[(2)] In the supercritical Obukhov-Corrsin regime $\beta < \alpha^\circ$ 
 \begin{itemize}
\item[(i)] the equation ~\eqref{e:intro:advection} with a fixed velocity field in $L^p([0,T];C^\alpha(\T^n))$ has infinitely many solutions in $L^{p^\circ} ([0,T];C^\beta(\T^n))$ with the same initial datum, and these solutions dissipate in time the spatial $L^2$ norm;
\item[(ii)] there is anomalous dissipation for solutions $\vartheta_\kappa$ of \eqref{e:intro:advdiff}  which are equibounded in $L^{p^\circ}([0,T];C^\beta(\T^n))$ uniformly in $\kappa >0$.
\end{itemize}
\end{itemize} 
}

\medskip

The above statements have a clear connection with the corresponding ones for the Euler equations. The Onsager criticality threshold $L^3([0,T];C^{\sfrac{1}{3}}(\T^n))$ for the velocity field corresponds to the Obukhov-Corrsin criticality threshold by formally considering $u=\vartheta$. Onsager's conjecture~\cite{onsager} 
 asserts conservation of the kinetic energy for solutions with spatial $C^{\sfrac{1}{3}+}$ regularity, and existence of nonunique and energy-dissipative
 solutions with spatial $C^{\sfrac{1}{3}-}$ regularity. The positive part of Onsager's conjecture corresponds to the subcritical case in (1) above and has been answered affirmatively (in a number of slightly different functional settings) in~\cite{eyinkonsag,ConstantinETiti,fried}. The negative part of Onsager's conjecture corresponds to item (i) in (2) above and has recently been established as the culmination of an amazing mathematical tour de force~\cite{DLL09,DLL13,BDLIS,IsettOnsager,OnsagerAdmissible} by relying on techniques of convex integration. 

Item (ii) in (2) goes beyond Onsager's conjecture and corresponds to a deterministic version of the so-called Kolmogorov's 0-th law of turbulence~\cite{kolmo}, which postulates the universality in a statistical sense of anomalous dissipation (the analogue of~\eqref{e:intro:anomalous} for the velocity field) for the Euler equations in the vanishing-viscosity limit of the Navier-Stokes equations in the regime of fully developed turbulence. The mathematical understanding of these predictions remains a great challenge. On one side, even a rigorous formulation of the statistical problem is missing. On the other side, { we lack even a single deterministic example of anomalous dissipation}: generating solutions of the Navier-Stokes equations with bounds uniformly in viscosity and displaying anomalous dissipation appears to be a major  challenge. In this direction, in~\cite{BuckmasterVicolAnnals} convex integration has been employed to construct weak solutions of the Navier-Stokes equations and it has been shown that energy-dissipative 
$C^\beta$ solutions of the Euler equations can be realized as limit of irregular Navier-Stokes solutions in $L^\infty ([0,T];  L^2(\T^3))$. It remains open whether the same can be done with Leray-Hopf or smooth solutions
. We will further comment on this after the statement of Theorem~\ref{t_mainadvection}. 
At the same time as the present manuscript, an example of anomalous dissipation for the 
three-dimensional Euler equations with (viscosity-dependent) forcing has been constructed in~\cite{camilloelia}. Such an example has a ``two-and-a-half-dimensional'' structure: the first two components of the vector field solve the forced two-dimensional Euler equations and  exhibit the same quasi-self-similar behaviour in~\cite[Sections~6--8]{ACM-JAMS}, while 
the third component displays anomalous dissipation and is a bounded scalar passively advected as in \eqref{e:intro:advection}. 
Notice that in~\cite[Theorem 3.1]{camilloelia} no uniform-in-viscosity estimates better than boundedness are available for the passive scalar, in particular no information is available on a uniform H\"older modulus of continuity as predicted in the Onsager and Obukhov-Corrsin theories.

\medskip 

The positive result in the subcritical Obukhov-Corrsin regime (recall item~(1) above) can be proven along the lines of~\cite{ConstantinETiti} (regularizing the equation and showing the convergence of the commutator which results from a decomposition analogue to the formal rewriting of the nonlinear term in~\eqref{e:intro:testing}) and has been done in~\cite{emil_notes,Gautam}. In~\cite{procaccia2} the criticality relations~\eqref{e:intro:yaglom} have been interpreted as a consequence of the fractal geometry of the level sets of the passive scalar. 
{ Recently, in \cite{Gautam} the authors address the endpoint case $\alpha <1$ and $\beta =0$ (interpreted as $L^\infty$ bounds for the passive scalar) and provide criteria for anomalous dissipation. They provide an explicit example (based on a previous construction in \cite{P94}) of a velocity field which exhibits anomalous dissipation for every initial datum sufficiently close to an eigenfunction of the Laplacian. More in general they also construct, for every initial datum with suitable regularity, a velocity field for which anomalous dissipation occurs; however, the velocity field depends on the chosen initial datum. To the best of our knowledge, no further analytical results prior to our paper are available in the full supercritical Obukhov-Corrsin regime, in particular, no examples are known in which the passive scalar enjoys some fractional regularity uniformly in diffusivity.}

 We stress  that there is a major gap in difficulty between showing boundedness for the passive scalar, and some fractional regularity uniformly in diffusivity (as in \eqref{bound_main_OC}  below). Indeed, the advection equation~\eqref{e:intro:advection} and the advection-diffusion equation~\eqref{e:intro:advdiff} are easily seen to propagate boundedness of the initial datum 
 since the velocity field is divergence-free. On the other hand, the advection equation~\eqref{e:intro:advection} is known not to propagate any fractional regularity of the initial datum~\cite{ACM-loss}, not even for velocity fields with Sobolev regularity of order one (and therefore in the DiPerna-Lions class~\cite{DPL89}). The fractional regularity of the passive scalar is a major difficulty, especially when requiring uniform-in-diffusivity bounds. This has been left fully open in~\cite{Gautam} (see in particular~\cite[Question~5.1]{Gautam}). 

\medskip

In this paper we provide several rigorous results within the supercritical Obukhov-Corrsin regime. Our first result constructs, for any chosen regularity in the supercritical regime, a velocity field and an initial datum exhibiting anomalous dissipation, under the required regularity bounds on the passive scalar uniformly in diffusivity:

\begin{maintheorem}[Anomalous dissipation in the Obukhov-Corrsin theory]\label{t_main_OC} 
Let $p \in [2,\infty]$, $p^\circ \in [2,4]$, $\alpha\in [0,1]$, and $\beta \in [0,\sfrac{1}{2}]$ be such that 
$$
\frac{1}{p} + \frac{2}{p^\circ} = 1
\qquad\text{ and }\qquad
 \alpha + 2\beta <1 \,.
$$
Then there exists a divergence-free velocity field $u \in L^{p} ([0,1]; C^\alpha(\T^2))$ and an initial datum $\vartheta_{\initial} \in C^\infty(\T^2)$ with $ \int_{\T^2} \vartheta_{\initial}=0$ such that the solutions $\vartheta_\kappa$ of the advection-diffusion equation \eqref{e:intro:advdiff} with initial datum $\vartheta_{\initial}$ satisfy the uniform-in-diffusivity bound 
\begin{equation}\label{bound_main_OC} 
\sup_{\kappa \in [0,1]} \| \vartheta_\kappa\|_{L^{p^\circ} ([0,1]; C^\beta(\T^2))} < \infty
\end{equation}
and exhibit anomalous dissipation
\begin{equation}\label{diss_main_OC} 
\limsup_{\kappa \to 0} \, \kappa \int_0^1 \int_{\T^2} | \nabla \vartheta_\kappa|^2 \, dx\,dt >0 \,.
\end{equation}
\end{maintheorem}
In Theorem~\ref{t_main_OC}, the solutions $\vartheta_\kappa$ are unique and bounded in $L^\infty((0,1)\times \T^2)$ by $\|\vartheta_{\initial} \|_{L^\infty}$. 
{
As it was the case for \cite{Gautam}, in our example all of the anomalous dissipation is concentrated at time $t=1$, in the sense that for any $\eps>0$
$$
\limsup_{\kappa \to 0} \, \kappa \int_0^{1- \eps} \int_{\T^2} | \nabla \vartheta_\kappa|^2 \, dx\,dt =0 \,.
$$
 This can be interpreted as mimicking the development in time of a turbulent cascade.
Recently, the paper \cite{AV23}, which appeared on arXiv about ten months after the present manuscript, presents a new construction of  vector fields  in $C^\alpha$ for $\alpha < \sfrac{1}{3}$ exhibiting anomalous dissipation for {any} non-constant initial data in $H^1$.  The anomalous dissipation in \cite{AV23} happens continuously in time, in the sense that an analogue of \eqref{diss_main_OC} holds with the time interval $[0,1]$ replaced by any subinterval
, in agreement with the predictions of homogeneous isotropic turbulence, which postulates (statistical) stationarity and therefore the absence of a ``preferred'' time in turbulent phenomena. The passive scalars are also announced to be H\"older regular uniformly in dissipation. The arguments in \cite{AV23} are elaborate and build on homogenization theory. 

We will give a quick, informal overview of our proof at the end of the introduction. We remark that we do not rely on the criteria for anomalous dissipation in~\cite{Gautam}.}
 In fact, such criteria are based on inverse interpolation inequalities and therefore only allow to obtain in the vanishing-diffusivity limit 
%
%
solutions of the advection equation~\eqref{e:intro:advection} which dissipate the $L^2$ norm. In our second theorem we show the existence of a velocity field for which, in the vanishing-diffusivity limit, two subsequences of solutions of~\eqref{e:intro:advdiff} exist, one converging to a solution of~\eqref{e:intro:advection} which conserves the~$L^2$ norm and one converging to a solution which dissipates the $L^2$ norm. 

\begin{maintheorem}[Lack of selection by vanishing diffusivity]\label{t_mainadvection}
For every  $\alpha \in [0,1)$ there exists  a divergence-free velocity field $u \in C^\alpha ([0, 2] \times \T^2)  $ and an initial datum $\vartheta_{\initial} \in C^\infty(\T^2)$ with $ \int_{\T^2} \vartheta_{\initial}=0$ such that the following holds.
The sequence $\vartheta_\kappa$ of solutions of the advection-diffusion equation \eqref{e:intro:advdiff} with velocity field $u$ and initial datum $\vartheta_{\initial}$ has at least two distinct limit points as $\kappa \to 0$ in the weak$^*$ topology, which are two distinct solutions of the advection equation~\eqref{e:intro:advection}. 
Moreover, one limit solution conserves the $L^2$ norm, namely~$\| \vartheta( t, \cdot ) \|_{L^2} = \| \vartheta^{\initial}\|_{L^2} $ for a.e. $t \in[0,2]$, and the other one exhibits strict dissipation of 
the $L^2$ norm, namely~$\| \vartheta(  t,\cdot ) \|_{L^2} \leq \| \vartheta^{\initial}\|_{L^2} / 2$ for any $t \geq 1$.
\end{maintheorem}

This theorem is the first instance in the literature showing  the impossibility to select a unique solution of the advection equation \eqref{e:intro:advection}  by vanishing diffusivity.
 It identifies, in the limit, nonunique bounded solutions for the advection equation~\eqref{e:intro:advection} with a $C^\alpha$ velocity field, which were known to exist~\cite{hamilt1,hamilt2,ACM-JAMS,Gautam}. In contrast, a unique solution of~\eqref{e:intro:advection} is selected as vanishing-diffusivity limit for velocity fields within the DiPerna-Lions theory~\cite{DPL89} but for solutions lacking the required integrability for the theory to apply (see~\cite{PBGC}). In turn, nonuniqueness of solutions in this latter context has been shown in a series of papers~\cite{Modena1,Modena2,Modena3,chesluo1,chesluo2}
. Therefore, Theorem \ref{t_mainadvection} provides a full picture for the selection principle for bounded solutions of the advection equation~\eqref{e:intro:advection} with a $C^\alpha$ velocity field in the following sense:
\begin{itemize}
\item[(1)] If the velocity field is in $C^{1}$ then there exists a unique bounded solution of the advection equation and it is selected by vanishing diffusivity.
\item[(2)] For any $\alpha <1$ there exists a velocity field  $ u \in C^\alpha$ such that  there are at least two distinct bounded solutions of the advection equation selected by vanishing diffusivity.
\end{itemize}

Theorem~\ref{t_mainadvection} is motivated by the fundamental open problem of the selection of solutions of the Euler equations as vanishing-viscosity limit of solutions of the Navier-Stokes equations. The results in~\cite{BuckmasterVicolAnnals} show that no selection is possible as limit of weak solutions in $L^\infty ([0,T];  L^2(\T^n))$ of the Navier-Stokes equations. The next step would be addressing the question of selection as limit of Leray-Hopf solutions (which anyway are not expected to be unique, see~\cite{jiasverakselfsim,jiasverakillposed,LerayHopf}) or smooth solutions. The results in~\cite{BuckmasterVicolAnnals} can be seen as a mathematical indication of a negative answer, and our Theorem~\ref{t_mainadvection} supports a negative answer as well, providing a proof in the case of passive scalars.


\medskip

The selection of a unique weak solution  is a key question for many different PDEs with multiple weak solutions. For instance, the entropy conditions or, equivalently, the vanishing diffusivity select a unique solution for scalar conservation laws, which in general possess infinitely many weak solutions. { The existence of a divergence-free velocity field in $C^\alpha$, for arbitrary $\alpha <1$, disproving the possibility of selection under vanishing viscosity for any non-constant initial data is an open problem. Partial results in this direction appeared subsequently to our work in \cite{AV23}, for a velocity field in $C^\alpha$, for $\alpha < \sfrac{1}{3}$, and later in \cite{HT23}, for a velocity field in $L^\infty$. In \cite{HT23} the authors also construct an example of  a bounded, divergence-free velocity field such that the vanishing diffusivity limit selects a unique solution with non decreasing energy profile.}
 Besides vanishing diffusivity, another conceivable selection criterion for the advection equation~\eqref{e:intro:advection} is based on a regularization of the velocity field, that is, by considering limit points of solutions of~\eqref{e:intro:advection} in which the velocity field $u$ is replaced by some regularization $u_\sigma \to u$. This question has been addressed and answered in the negative in~\cite{ciampa,Giri-DL}, yet relying on extremely ad-hoc regularizations of the velocity field, explicitly tuned on the singularities of the velocity field. The freedom in the choice of the sequence $u_\sigma$ is exploited in an essential way in~\cite{ciampa,Giri-DL} in order to generate several distinct ways to bypass the singularity which are maintained in the limit $\sigma\to 0$. In particular, the chosen regularization is not based on convolution with a smooth kernel. In our third result we show that selection cannot be obtained by regularization of the velocity field by convolution with a smooth kernel in space-time (notice that we do not make any assumptions on the profile of the kernel $\varphi$). 

\begin{maintheorem}[Lack of selection by convolution]\label{t_main_convol} 
Let $C>0$ and $\alpha \in [0,1)$ and let $\varphi \in C^\infty_c((-1,1) \times B(0,1))$ be a convolution kernel in space-time
with $\| \varphi \|_{C^1} \leq C$. 
There exists  a divergence-free velocity field $u \in C^\alpha ([0, 2] \times \T^2)  $ and an initial datum $\vartheta_{\initial} \in C^\infty(\T^2)$ with $ \int_{\T^2} \vartheta_{\initial}=0$, both depending only on $\alpha$ and $C$ and not on $\varphi$, such that the  following holds.
The sequence $\vartheta_\sigma$ of solutions of the advection equation~\eqref{e:intro:advection} with velocity field $u \ast \varphi_\sigma$  and initial datum $\vartheta_{\initial}$ has at least two distinct limit points as $\sigma \to 0$ in the weak$^*$ topology, which are two distinct solutions of the advection equation~\eqref{e:intro:advection}. 
\end{maintheorem}

We now present an informal discussion of the main ideas and tools in our approach, with a focus on the novelties and making several connections with the previous literature. In Section~\ref{section:vector_geometric} we give a more detailed and quantitative, but still heuristic, account of our proof, before proceeding in the rest of the paper to the full proofs. 

\medskip

Roughly speaking, the basic mechanism for the proof of our three theorems is the same and relies on the construction of a solution which gets mixed (that is, weakly converges to its spatial average) when the time approaches the singular time $t=1$. This is reminiscent of Depauw's example~\cite{Depauw} (see also~\cite{aize} for an earlier related construction), in which mixing is generated by alternating discontinuous dyadic shear flows which rearrange the solution into finer and finer dyadic chessboards, and of the more regular self-similar examples in~\cite{ACM-JAMS,YZ}. However, in order to achieve our goals (regularity of velocity field and passive scalar, anomalous dissipation, convergence to multiple limit points), we need several major twists in the construction. 

The first main novelty is replacing the dyadic decay $2^{-q}$ by a decay along a well-chosen, superexponential sequence~$a_q \downarrow 0$ at times $1 - T_q \uparrow 1$; correspondingly, the velocity field is localized at frequency $a_{q+1}^{-1}$ in the time interval $[1-T_q, 1-T_{q+1}]$. In all the three proofs, we interpret the convolution or the presence of diffusion (represented by the Brownian motion in the stochastic differential equation) as a ``filter on high frequencies'' acting at a well-chosen intermediate scale between  the well-separated scales $a_{q+1} \ll a_q$ and leaving low frequencies essentially unchanged.

In Theorem~\ref{t_main_convol}, the convolution blocks the mixing process once the passive scalar reaches scale $a_q$. After the singular time $t=1$, we define the velocity field by reflection (which would un-mix the passive scalar along the corresponding scales), but we also add a velocity field which at every transition of scale has the effect to ``swap'' a parity marker on the chessboards. Depending on the parity of $q$, the passive scalar undergoes an even or an odd number of swaps, which produces two distinct solutions in the limit. 
\medskip

Dealing with diffusion is more delicate and requires a second major twist with respect to the previous approaches. We develop an original and fairly general Lagrangian approach (in a deterministic or stochastic sense) to anomalous dissipation. Solutions of the advection equation~\eqref{e:intro:advection} are transported by the flow of the velocity field and the Feynman-Kac formula represents solutions of the advection-diffusion equation~\eqref{e:intro:advdiff} via the flow of the associated stochastic differential equation (see~\eqref{feynman_kac_real}). Our Lagrangian approach allows to keep track of the regularity of the solutions as well as to make explicit the anomalous dissipation mechanism in a fully quantitative way in our situation, by means of a precise control of the frequency of the solution up to the precise scale where anomalous dissipation happens. 
 The anomalous dissipation relies on the following observation: if the diffusion is active for a time~$\tau$, then it acts on a spatial scale~$\sqrt{\kappa\tau}$, and additionally its effect is enhanced by the high frequency of the passive scalar. The last observation is at the core of the so-called enhanced-dissipation phenomenon, which received a considerable attention in the last years, from nonquantitative spectral characterizations~\cite{constantinannals}, to quantitative approaches~\cite{BCZ,DCZ,CZ2020} also in terms of the mixing properties of the velocity field~\cite{CZDE,pulsed}, to the randomly-forced case~\cite{bedrendiss}. In Theorems~\ref{t_main_OC} and~\ref{t_mainadvection} the quantified relation between the time scale, the frequency of the solution and the diffusivity parameter 
 allows, respectively avoids, time intervals on which the diffusion substantially mixes the solution. This switches on and off the dissipation effect 
 along different subsequences, allowing to identify in the limit a solution which conserves the $L^2$ norm and a solution which dissipates the $L^2$ norm. Our approach represents a major novelty and needs to be compared with the approach in~\cite{Gautam}, which relies on criteria for anomalous dissipation based on reverse interpolation inequalities for the passive scalar. It seems not possible to exploit the strategy of~\cite{Gautam} and obtain multiple limit points in the vanishing-diffusivity limit. Our approach identifies a Lagrangian mechanism and provides its quantification, establishing a fully general novel method which can be employed in problems exhibiting similar features.

\medskip

The issue of the regularity has been neglected in the above informal description. Finding more regular versions of Depauw's~\cite{Depauw} velocity field is a notoriously hard task, which has been addressed in a dyadic, self-similar setting in~\cite{ACM-JAMS,YZ}. However, the velocity fields in~\cite{ACM-JAMS,YZ} do not appear to be compatible with our choice of a superexponential sequence of scales $a_q$. Moreover, it is unclear which instruments could be employed to obtain the fractional regularity of the associated solutions. Instead, we directly employ a finer-scales version of Depauw's velocity field and directly regularize it at a scale finer than its own scale, slightly perturbing the scaling and directly controlling the error terms. The resulting velocity field is again made by alternating shear flows concentrated at well separated frequencies, in stark contrast with the velocity fields in~\cite{ACM-JAMS,YZ} which are genuinely two-dimensional. 
This last fact is in turn crucial in order to establish the uniform-in-diffusivity regularity of the passive scalar in~\eqref{bound_main_OC}. We stress that such uniform regularity does not follow from any functional argument (the advection equation~\eqref{e:intro:advection} is known not to propagate any fractional regularity of the initial datum~\cite{ACM-loss}), noticing also that the expected regularity of the passive scalar predicted by the Obukhov-Corrsin theory (namely, the admissible $\beta$ in \eqref{bound_main_OC}) worsens when the velocity field becomes more regular. Instead we prove a novel regularity estimate for the stochastic flow which relies in an essential way on the alternating shear-flow structure of the velocity field. The estimate is uniform in the stochastic parameter and therefore thanks to the Feynman-Kac formula can be translated into a regularity estimate for the passive scalar. 

\medskip

\noindent{\bf Plan of the paper.} In Section~\ref{section:vector_geometric} we present a quantitative heuristics of our arguments aimed at providing the reader with a more detailed, yet nontechnical understanding of our paper. Section~\ref{section:preliminaries} is devoted to a few preliminaries on deterministic and stochastic flows. In Section~\ref{section:parameters} we fix the parameters needed for our arguments and we construct the velocity field. We then move to the proof of Theorem~\ref{t_main_convol} in Section~\ref{section:prooftheorem}. The convergence in the vanishing-diffusivity limit to a solution which conserves the $L^2$ norm (which is a part of Theorem~\ref{t_mainadvection}) is shown in Section~\ref{section:conservativesolution}. In Section~\ref{section:nonconservative} we show the convergence in the vanishing-diffusivity limit to a solution which dissipates the $L^2$ norm, thus concluding the proof of Theorem~\ref{t_mainadvection} and showing the anomalous dissipation claimed in Theorem~\ref{t_main_OC}. In Section~\ref{section:regularity} we show the uniform-in-diffusivity regularity of the passive scalar, thus concluding the proof of Theorem~\ref{t_main_OC}. 

\medskip

\noindent{\bf Acknowledgments.} MC and MS were supported by the SNSF Grant 182565 and by the Swiss State Secretariat for Education, Research and lnnovation (SERI) under contract number M822.00034.
This work has been started during a visit of GC at EPFL as a Visiting Professor. The support and the warm hospitality of EPFL are gratefully acknowledged. GC is partially supported by the ERC Starting Grant 676675 FLIRT and by the SNSF Grant 212573 FLUTURA. MS thanks Lucio Galeati for the interesting discussion about the advection-diffusion equation during his visit at EPFL.

%
%

%


%
%

\section{Strategy of the proof and heuristics} \label{section:vector_geometric}

In this section we provide an heuristic description of our construction and of the proofs of the three theorems stated in the introduction. Even though the presentation in this section oversimplifies several technical aspects and includes only heuristic estimates, it allows to introduce all the ideas and tools in the proofs, and at the same time it retains a quantitative enough character which motivates the choice of the many parameters in the construction. This section is intended to provide a guide for the reader through the actual proofs in the remaining of the paper.

\subsection{Geometric construction of the velocity field}\label{ss:vector}

Our construction features a divergence-free velocity field $u=u(t,x)$ defined for $(t,x) \in [0,2] \times \T^2$. The velocity field is smooth outside $\{1\} \times \T^2$ and at each  time $t \in [0,2] \setminus \{1\}$ is concentrated in frequency, where the frequencies blow up as the time approaches $1$.

Let us fix a sequence of frequencies $\{ \lambda_q\}_{q \in \N}$ which are well separated, namely with at least superexponential growth, and  a sequence of times $\{1-T_q\}_{q \in \N}$ with $T_q \downarrow 0$ to be chosen later in terms of the frequencies.
We first define the velocity field for $t \in [0,1]$. On each time interval $[1-T_{q}, 1-T_{q+1}]$ the velocity field is constructed as follows.
\begin{enumerate}
\item[(1)]\label{heuristics_1} The velocity field is composed by two shear flows (one horizontal and one vertical) with frequency concentrated at $\lambda_{q+1}$ rearranging the chessboard of side $\lambda^{-1}_{q}$ into the chessboard of side $\lambda^{-1}_{q+1}$ (see Figure~\ref{f:twoshears}). 

\begin{figure}[h]
\begin{minipage}{4cm}
\centering
\includegraphics[width=4cm]{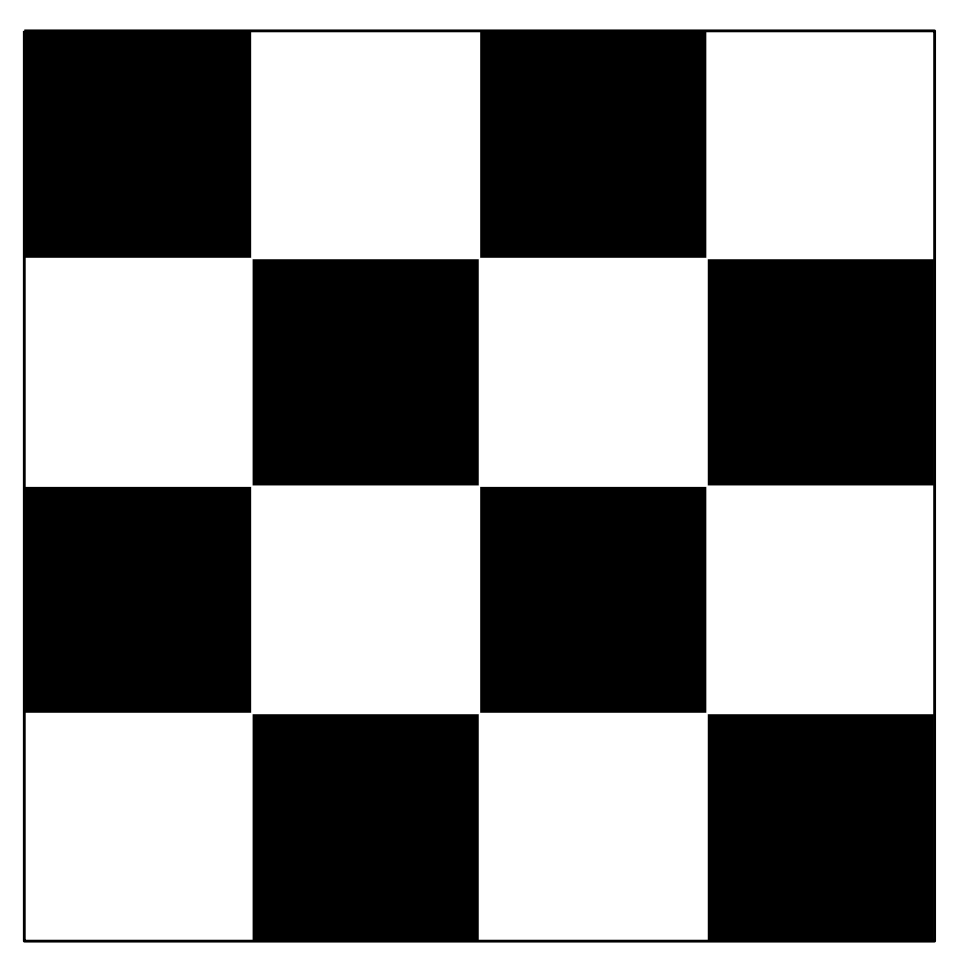}
\end{minipage}
\begin{minipage}{4cm}
\centering
\includegraphics[width=4cm]{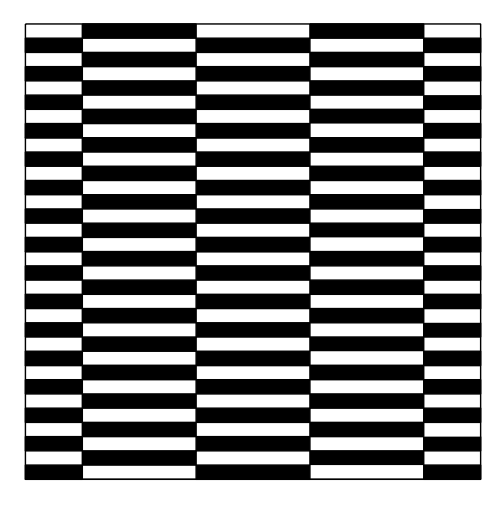}
\end{minipage}
\begin{minipage}{4cm}
\centering
\includegraphics[width=4cm]{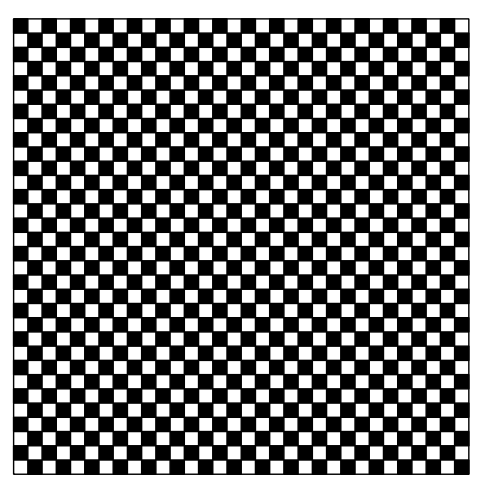}
\end{minipage}
\caption{The action of the two shear flows rearranging the chessboard of side $\lambda^{-1}_{q}$ into the chessboard of side $\lambda^{-1}_{q+1}$.}\label{f:twoshears}
\end{figure}

\item[(2)]\label{c:interval_long} For some fixed $m \in \N$, on the first part of the time interval $[1-T_{q}, 1-T_{q+1}]$ and only for $q \in m\N$ (by which we mean that $q$ is a multiple of $m$) we add a long time subinterval on which the velocity field is zero. The choices of $m \in \N$ and of the length of the subinterval are tuned with certain choices of the dissipation and the dominant frequency of the solution of the advection-diffusion equation at that time.
\end{enumerate}

In the remaining time interval $[1,2]$ (after the singularity at time $t=1$) the velocity field is defined by reflection, and for Theorem~\ref{t_main_convol} only we add a swap velocity field as follows. 
\begin{enumerate}
\item[(3)]\label{c:interval_reflect} The reflected velocity field is defined by
\begin{equation}
\label{eqn:campodopo}
u(t,x) = - u(2-t,x) \qquad \text{for $t \in [1,2]$ and $x \in \T^2$.}
\end{equation}
\item[(4)]\label{heurstics_3} For Theorem~\ref{t_main_convol} we add a swap velocity field $u_{\swap}$ which, in each time interval $[1+T_{q+1}, 1+T_{q}]$, is active in a subinterval in which the velocity field in \eqref{eqn:campodopo} is zero. The task of the swap velocity field on the time interval $[1+T_{q+1}, 1+T_{q}]$ is to swap the parity of the chessboard of side $\lambda^{-1}_q$ (namely, it exchanges the ``black squares'' with the ``white squares'', see Figure~\ref{f:swap}). The swap velocity field acts after the velocity field $u$ in~(3) has already reconstructed (at a certain time $1+T_{q+1} + \tau < 1+ T_q$) the chessboard of side $a_q$  from the the chessboard of side~$a_{q+1}$ and it has frequency concentrated at~$\lambda_{q+1}$.

\begin{figure}[h]
\begin{minipage}{3cm}
\centering
\includegraphics[width=3cm]{immagine1}
\end{minipage}
\hspace{1.5cm}
\begin{minipage}{3cm}
\centering
\includegraphics[width=3cm]{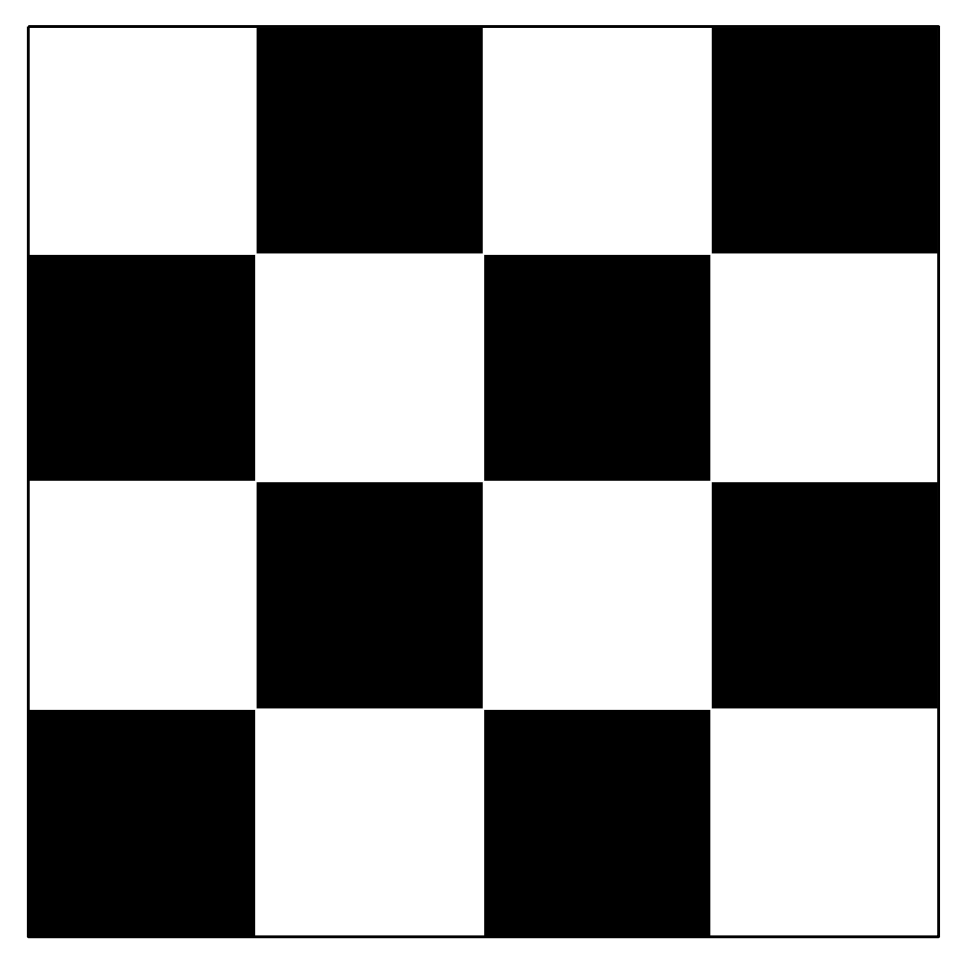}
\end{minipage}
\caption{The action of the swap velocity field}\label{f:swap}
\end{figure}

\end{enumerate}

In any time interval $[1-T_{q}, 1-T_{q+1}]$ or $[1+T_{q+1}, 1+T_{q}]$, the constructed velocity field either vanishes or has frequency concentrated at $\lambda_{q+1}$ (it is actually almost $\lambda_{q+1}$-periodic in a suitable sense). This fundamental property allows to localize the effect of the convolution and the effect of the diffusion to certain specific time intervals, following the general principle that in our regimes \emph{both convolution and diffusion act as a filter on high frequencies, leaving low frequencies relatively unchanged.}

To make the heuristics visually clear we refer in this section to velocity fields that ``mix a chessboard into a finer one'' with implicit reference to a simple and widely known example by Depauw~\cite{Depauw} (see also \cite{Giri-DL}, where such construction was used to build an example of lack of selection for bounded velocity fields via an ad-hoc regularization not of convolution type). We underline three major differences between the velocity fields in~\cite{Depauw,Giri-DL} and ours. First, it is fundamental for us to separate consecutive frequencies in a sharp way by using a superexponential sequence of frequencies. Second, we introduce time delays in order for 
 the diffusion to trigger the anomalous dissipation and the lack of selection.
Third, in order to guarantee H\"older regularity in our example, all the statements about rearrangements of chessboards must be intended ``up to small errors'' due to the presence of small-scale mollifications that guarantee the smoothness of the velocity field and of the passive scalar away from the singular time $t=1$. Without such mollifications, the heuristics holds only in fractional Sobolev spaces such as $W^{\alpha,1} (\T^2)$, rather than in H\"older spaces $C^\alpha(\T^2)$. 
We observe in passing that, without mollifications and under fractional Sobolev regularity bounds, full dissipation for the passive scalar can be obtained in Theorem~\ref{t_main_OC}, namely replacing the weaker property~\eqref{diss_main_OC} with 
$$
2 \limsup_{\kappa \to 0} \kappa \int_0^1 \int_{\T^2} | \nabla \vartheta_\kappa|^2 \, dx\,dt
=  \| \vartheta(0,\cdot) \|^2_{L^2(\T^2)} \,.
$$

Obtaining the correct regularity of the velocity field and the associated solutions is a central issue in our approach. The same difficulty has been faced in~\cite{ACM-JAMS,YZ}, where the same rate of decay for the mixing scale as in~\cite{Depauw} was obtained for more regular (Sobolev, or even Lipschitz) velocity fields. However, the construction in~\cite{ACM-JAMS,YZ} has been achieved by a completely different argument due to necessity to keep under control the evolution in the entire space without allowing for error terms, as needed when dealing with a statement concerning mixing. In particular, the velocity fields in~\cite{ACM-JAMS,YZ} are genuinely two-dimensional. In our case we retain by mollification the alternating shear-flow structure of the velocity field from~\cite{Depauw}, which is essential in order to quantitatively control the regularity of the solution in presence of diffusion. However, we have to face the issue to control the evolution and quantify the dissipation of the passive scalar out of the small set where the mollification takes place.

\subsection{Geometric construction of the initial datum}
We consider as initial datum the function $\vartheta_0$ which equals $1$ on the even chessboard of side $2 \lambda_0$ (and therefore, has periodicity $\lambda_0^{-1}$) and $-1$ on the odd chessboard. As before, the actual proof involves a mollification of this function, but at the present heuristic level we ignore this issue.

\subsection{Qualitative behaviour of solutions}\label{ss:qualitsol}
Up to time $t=1$, the solution of the advection equation~\eqref{e:intro:advection} with velocity field $u$ is unique thanks to its (local-in-time) regularity: at each time $1-T_q$, the solution $\vartheta$ approximately equals the chessboard of side~$\lambda_q^{-1}$. Therefore, $\vartheta$ can be extended by weak continuity at the singular time~$t=1$ as $\vartheta(1,\cdot) \equiv 0$.

After the singular time $t=1$ we can consider three possible continuations of this solution. First, we consider the solution which fully dissipates the $L^2$ norm, namely
$$
\vartheta^{\rm mix}(t, \cdot) \equiv 0 \qquad \text{ for $t\in [1,2]$.}
$$
Second, we observe that if point (4) in the construction in Section~\ref{ss:vector} is omitted, then a solution which conserves the $L^2$ norm is  given by reflection, namely
$$
\vartheta^{\rm cons}(t, \cdot) = \vartheta(2-t, \cdot) \qquad \text{ for $t\in [1,2]$.}
$$
For the velocity field introduced in Section~\ref{ss:vector} with the addition of the swap velocity field in~(4), we can find two distinct backward solutions on $[1,2]$, starting at $t=2$, which are compatible with $\vartheta$ at time~$1$, namely $\vartheta^{\rm odd}$ generated by evolving backward the final datum $\vartheta^{\rm odd} (2, \cdot) = \vartheta_0$, and $\vartheta^{\rm even}$ generated by evolving backward the final datum $\vartheta^{even} (2, \cdot) =- \vartheta_0$. We observe that $\vartheta^{\rm odd}$ and $\vartheta^{\rm even}$ are characterized at times $1+T_q$ for  $q\in \N$ and equal alternatively the odd or even chessboard of side $\lambda_q^{-1}$.

\subsection{Lack of selection by convolution (Theorem~\ref{t_main_convol})}

We first present the mechanism behind the lack of selection in Theorem~\ref{t_main_convol} as it is the simplest one.
Fix $\sigma>0$ such that for some~$q\in \N$
\begin{equation}\label{e:sigmaparity}
\lambda_{q+1}^{-1} \ll \sigma \ll \lambda_q^{-1} 
\end{equation}
(more precise quantitative bounds will be required in the proof), and consider the convolution $u \star \varphi_{\sigma}$ of the velocity field $u$ (including the swap velocity field in~(4)) with a kernel $\varphi_\sigma$. The convolution does not change much the low-frequency part of the velocity field, namely its restriction to the time intervals~$[0, 1-T_q] \cup [1+T_q, 1]$ thanks to the condition $\sigma \ll \lambda_q^{-1}$. In contrast, the convolution almost completely cancels the fast oscillations of the velocity field in $[ 1-T_q,1+T_q]$, thanks to the condition~$\sigma \gg \lambda_{q+1}^{-1}$. As a consequence, the flow of $u \star \varphi_{\sigma}$ is close in the $L^1$ norm to the flow of the smooth velocity field $u \mathbbm{1}_{[1-T_q, 1+ T_q]^c}$. Hence, the evolution of $\vartheta_{\initial}$ under the velocity field $u \star \varphi_{\sigma}$ undergoes an even or odd amount of parity swaps (under the action of the swap velocity field $u_{\swap}$) depending on the parity of the integer $q$ for which~\eqref{e:sigmaparity} holds, and therefore it possesses two qualitatively distinct behaviours along two subsequences of $\sigma\to 0$. As $q \to \infty$, these two distinct behaviours are maintained and give rise to the two limit solutions~$\vartheta^{\rm even}$ and $\vartheta^{\rm odd}$ of Section~\ref{ss:qualitsol}.

\subsection{Qualitative behaviour of solutions in the presence of diffusion}

From now on we consider the velocity field $u$ introduced in Section~\ref{ss:vector} ignoring the part (4) of the construction
. Given the unique stochastic flow $\XX^\kappa : \Omega \times [0,2] \times \T^2 \to \R^2$ solution of the stochastic differential equation (SDE) 
\begin{align}\label{d:SDE}
\begin{cases}
d \XX_{t}^\kappa =   u(t, \XX_{t}^\kappa) dt + \sqrt{2 \kappa}  {d} \WW_{t} 
\\
\XX_{0}^\kappa = x_0
\end{cases}
\end{align}
for a fixed probability space $(\Omega, \mathcal{F}, \mathbb{P})$, the unique bounded solution of the advection-diffusion equation~\eqref{e:intro:advdiff} is represented by the Feynman-Kac formula
$$
 \int_{\T^2} f  \vartheta_\kappa \, dx=  \mathbb{E} \int_{\T^2} f(\XX_t^{\kappa}) \vartheta_0 \, dx \qquad \mbox{for any $f \in L^\infty(\T^2)$.}
$$
The effect of the diffusivity in the advection-diffusion equation~\eqref{e:intro:advdiff}, or equivalently at Lagrangian level the effect of the Brownian motion in the SDE~\eqref{d:SDE}, can be seen as a regularizing mechanism. It shares similarities, as well as relevant differences, with the regularization by convolution of the velocity field directly at the level of the advection equation~\eqref{e:intro:advection}. The regularization by convolution acts at the level of the velocity field, whereas in~\eqref{d:SDE} the ``regularization by noise'' acts at the level of the trajectories by adding the random perturbation provided by the Brownian motion. Remarkably, the two regularizations act on different scales. For the regularization by noise the typical scale of regularization is proportional to the length of the time interval on which it acts. On a time interval of length $\tau$ where the velocity field vanishes, the regularization by convolution does not modify the flow (and therefore, the solution), while the regularization by noise acts on the solution as a convolution with a Gaussian with variance proportional to $\kappa \tau$. The stochastic Lagrangian interpretation will be particularly useful for our analysis due to the fundamental estimate for the Brownian motion
%
%
\begin{equation} \label{stima:brown}
\mathbb{P} \left ( \omega \in \Omega : \sup_{t \in [0, T]}  \sqrt{2 \kappa} | \WW_t | \leq c \right ) \geq 1-  2 e^{ \sfrac{-c^2}{(2 \kappa T)}} \,.
\end{equation}
%
We finally observe that the effect of the dissipation is stronger when acting on highly-oscillatory solutions. This observation lies at the heart of the so-called enhanced dissipation phenomenon and will be employed in an essential way in our proofs.

\subsection{Convergence via diffusion to a  solution which dissipates the $L^2$ norm (Theorems~\ref{t_main_OC} and~\ref{t_mainadvection})} \label{ss:eur:dissipative}
Our strategy relies on the choice of a sequence of diffusivity parameters $\{\tilde \kappa_{q}\}_q$ such that the solution $\vartheta_{\tilde \kappa_{ q}}$ of the advection-diffusion equation~\eqref{e:intro:advdiff} enjoys the following properties:
\begin{itemize}
\item
{\bf Goal~1:} It can be well approximated by the solution of the advection equation until time $1-T_{q}$, in particular at time $1-T_{ q }$ it is close to an (almost) $\lambda_{ q }^{-1}$-periodic function;
\item
{\bf Goal~2:} It dissipates half of its $L^2$ norm in the time interval $[1-T_{ q }, 1-T_{{ q }+1}]$ only by the effect of diffusion. 
\end{itemize}
Since the velocity field has frequency $\lambda_q$ in the time interval $[1-T_{q-1}, 1-T_q]$ and the solution of the advection equation resembles the  chessboard of side $\lambda_{ q  }^{-1}$ at time $1-T_{ q  }$, in order to accomplish Goal~1 we need to require (at least) that the stochastic  flow deviates on average from the flow of $u$ less than the typical side of the chessboard $\lambda_q^{-1}$ in the time interval $[1-T_{q-1}, 1-T_q]$, namely
\begin{equation}
\label{eqn:close-flows}
| \XX^{\tilde \kappa_q}_{1-T_{ q }, 1-T_{q-1}} - \XX_{1-T_{ q }, 1-T_{q-1}}| \ll \lambda_q^{-1}.
\end{equation}
In fact, a technical point in the proof involves a precise control of the set where this estimate holds, together with suitable estimates on the analogue quantities $| \XX^{\tilde \kappa_q}_{1-T_{ k }, 1-T_{k-1}} - \XX_{1-T_{ k}, 1-T_{k-1}}|$ for $k<q$. Estimate~\eqref{eqn:close-flows} entails, by \eqref{stima:brown} with $c=a_{ q }$, $\kappa = \tilde \kappa_q$ and $T= T_{ q -1}-T_{{ q }}$,  the following constraint on $\tilde \kappa_q$
 \begin{equation}\label{eqn:close-flows-bis}
 \tilde \kappa_q  \lambda_{q} ^2 ( T_{q-1} - T_q) \ll 1 .
 \end{equation}

To achieve Goal~2 we recall the following. 
If $\vartheta_{\kappa, \lambda} \in L^\infty ((0,T); L^2 (\T^2))$ is a solution of the heat equation $\partial_t \vartheta_{\kappa, \lambda} - \kappa \Delta \vartheta_{\kappa, \lambda} =0 $ with initial datum $\vartheta_{\initial , \lambda}  (x) = \vartheta_{\initial } (\lambda x) \in L^2  (\T^2)$ (for some $\lambda \in \N$)  for a  given  function~$\vartheta_{\initial }\in L^2  (\T^2)$ with zero average, then by the scaling and decay properties of the heat equation
%
$$ \| \vartheta_{\kappa, \lambda}(t, \cdot) \|_{L^2(\T^2)}^2 = \| \vartheta_{\kappa,1} (\lambda^2 t, \cdot ) \|_{L^2(\T^2)}^2    \leq e^{- \kappa \lambda^2 t} \| \vartheta_{\initial} \|_{L^2(\T^2)}^2 = e^{- \kappa \lambda^2 t}  \| \vartheta_{\initial, \lambda} \|_{L^2(\T^2)}^2 \ .  $$

 Hence, in the time interval $[1-T_q, 1-T_{q+1}]$ 
with initial datum the (almost) $\lambda_{ q }^{-1}$-periodic function $\vartheta_{\tilde \kappa_q} (1-T_{ q }, \cdot )$, Goal~2 entails the constraint 
$$
 \label{eqn:eur-1}
  \tilde \kappa_q \lambda_{ q }^2 ( T_{q} - T_{q+1}) \gg 1.
$$
Since the construction requires $ T_q - T_{q-1} \to 0$ as $q \to \infty$, we will require $ T_q - T_{q+1} \gg  T_{q-1} - T_q$ only for a lacunary sequence (more precisely, for every multiple of a fixed integer $m$). This is necessary in order to make the two previous constraints compatible at least on a subsequence of $q \to \infty$.

\subsection{Convergence via diffusion to a solution which conserves the $L^2$ norm (Theorem~\ref{t_mainadvection})}\label{ss:eur-dissip}
For every $ q$, we want to choose $\kappa_{ q}$ in such a way that the solution $\vartheta_{\kappa_{ q}}$ of the advection-diffusion equation is close to the solution $\vartheta_q$ of the advection equation with the smooth velocity field  $u_q =  u \mathbbm{1}_{[1-T_q, 1+ T_q]^c} $. We notice that  $\vartheta_q$  has some convenient explicit features: 
first of all, $\vartheta_q (1-T_q , \cdot )=\vartheta_q (1+T_q, \cdot )$ agrees (up to small errors) with the chessboard of side $\lambda_q^{-1}$, and moreover by symmetry $\vartheta_q(2,\cdot) = \vartheta_q(0,\cdot)$.
In order to guarantee the closeness sketched above we need 
two controls of different nature:
\begin{itemize}
\item
{\bf Goal~1':} Similarly to Goal~1 in Section~\ref{ss:eur-dissip}, the solution $\vartheta_{ \kappa_{ q}}$ of the advection-diffusion equation is well approximated by the solution of the advection equation until time $1-T_{ q }$ and after time~$1+T_{ q }$; in particular, it is close to an (almost) $\lambda_{ q }^{-1}$-periodic function at time $1-T_{ q }$;
\item
{\bf Goal~2':} In full contrast with Goal~2 in Section~\ref{ss:eur-dissip}, the solution $\vartheta_{ \kappa_{ q}}$ does not dissipate energy, and actually it remains essentially unchanged in the interval $[1-T_{ q }, 1+T_{{ q }}]$. 
\end{itemize}

Goal~1' can be achieved similarly to Goal~1 in Section~\ref{ss:eur-dissip} and requires (at least) the constraint~\eqref{eqn:close-flows-bis} on the diffusivity $\kappa_{q}$.

In order to accomplish Goal~2' we need two steps. First, the so-called It\^o-Tanaka trick (first used in the context of the advection equation with multiplicative noise in \cite{FGP}) allows to exploit suitable cancellations to show that the stochastic 
flow remains almost constant on $[1-T_q, 1 - T_{q+1}]$. 
This relies on the fact that the velocity field on such time interval either vanishes or has frequency at least $\lambda_{q+1}$ and entails the fact that the typical displacement of the Brownian motion 
is much larger than the inverse of the typical frequency of the velocity field in the time interval  $[1-T_q, 1- T_{q+1}]$, namely
\begin{equation}\label{eur:ito-tanaka}
\sqrt{\kappa_q  (T_q - T_{q+1})} \gg \lambda_{q+1}^{-1} .
\end{equation}
Second, we need to require that the solution $\vartheta_{\kappa_q}$ dissipates only a small portion of $L^2$ norm in the time interval $[1-T_q, 1+T_q]$, which entails the constraint 
$$
 \kappa_q \lambda_{ q }^2 T_q \ll 1.
$$

\subsection{Regularity of the velocity field and uniform-in-diffusivity bounds on the passive scalar}\label{ss:regsol}
The velocity field rearranges in the time interval $[1- T_{q-1}, 1- T_q]$  the chessboard of side $\lambda_{q-1}^{-1}$ into the chessboard of side  $\lambda_q^{-1}$, and therefore the associated trajectories have a displacement of order $ \lambda_{q-1}^{-1}$ with frequency concentrated at $\lambda_q$. This implies by interpolation
%
$$ \| u(t, \cdot) \|_{C^\alpha( \T^2)} \sim{} \|u (t, \cdot ) \|_{L^\infty (\T^2)} \lambda_q^\alpha \sim \frac{\lambda_{q-1}^{-1} \lambda_q^\alpha}{ (T_{q-1} - T_q) } 
\qquad \text{ and } \qquad
 \| \vartheta(t, \cdot) \|_{C^\beta( \T^2)} \sim \|\vartheta (t, \cdot ) \|_{L^\infty (\T^2)} \lambda_q^\beta \sim \lambda_q^\beta 
$$
for $t\in[1-T_{q-1}, 1-T_q]$. As a first consequence, we observe that $u \in C^0([0,1] \times \T^2)$ as soon as~$\gamma \sim p^\circ \beta<1$; the latter condition holds in particular close to the  Onsager criticality threshold $\alpha = \alpha^\circ =1/3$ and~$p= p^{\circ} =3$. Moreover we estimate
\begin{equation} \label{heuristics_regularity_u}
 \| u \|^p_{L^{p} ([0,1]; C^\alpha(\T^2))} \sim \sum_{q} \lambda_{q-1}^{-p} \lambda_q^{p\alpha} ( T_{q-1} - T_q)^{1- p}
\end{equation} 
%
and
\begin{equation} \label{heuristics_regularity_theta}
\|\vartheta\|^{p^\circ}_{L^{p^\circ} ([0,1]; C^\beta(\T^2))} \sim   \sum_{q} \lambda_q^{p^{\circ}\beta} ( T_{q-1} - T_q) .
\end{equation}
Notice that  the passive scalar in our example exhibits some sort of mixing for $t \uparrow 1$, which is incompatible with a control of $\vartheta$ in $L^{\infty} ([0,1]; C^\beta(\T^2))$ for $\beta>0$, which would correspond to the case $p=1$.

The above heuristics can be made rigorous for solutions of the advection equation~\eqref{e:intro:advection}. However, for the corresponding solutions of the advection-diffusion equation~\eqref{e:intro:advdiff}, we are not aware of any general method to prove regularity estimates building on the previous observations. In the proof of Theorem \ref{t_main_OC} we develop a new approach to this problem which shows regularity for the stochastic flow uniformly in the stochastic parameter and which strongly relies on the alternating shear-flow structure of the velocity field. Notice that, 
%
%
%
for a genuinely two-dimensional velocity field as in~\cite{ACM-JAMS,YZ}, 
even though the same scaling as in~\eqref{heuristics_regularity_theta} is satisfied, showing uniform-in-diffusivity regularity estimates seems an extremely hard task. In~\cite{Gautam} only $L^\infty$ bounds on the passive scalar are shown to hold uniformly-in-diffusivity as a consequence of direct bounds on the advection-diffusion equation. However, due to results in~\cite{ACM-loss}, no fractional-regularity bounds due to functional arguments are expected to hold  for the advection-diffusion equation.

We close this section by observing how a suitable choice of the parameters is compatible with the regularity bounds~\eqref{heuristics_regularity_u} and~\eqref{heuristics_regularity_theta} in the full supercritical Obukhov-Corrsin regularity range. 
Assuming a superexponential growth of the frequencies of the form $\lambda_q = \lambda_{q-1}^{1+\delta}$ (where $\delta>0$ is a small parameter) and that 
%
$T_{q-1} - T_q \sim \lambda_q^{- \gamma}$ for $q \not \in m\N$ (for a suitable paramater $\gamma>0$), we observe that the convergence of both sums in~\eqref{heuristics_regularity_u} and~\eqref{heuristics_regularity_theta} follows from
$$
\begin{cases}
- p + p \alpha (1+\delta)- \gamma + \gamma p <0,
\\
p^{\circ} \beta (1+\delta)  - \gamma <0.
\end{cases}
$$
Letting $\delta \to 0$, this system
admits a solution for $\gamma$ provided $\sfrac{1}{p} + \sfrac{2}{p^\circ} =1$ and $\alpha + 2 \beta <1$, that is, in the full supercritical Obukhov-Corrsin regularity range.
We also observe that the choices~$\kappa_q = \lambda_q^{-2}$ and~$\tilde{\kappa}_q = \lambda_q^{-2 + \frac{\gamma}{1+ \delta}}$ guarantee the constraints listed in Section~\ref{ss:eur:dissipative} and Section~\ref{ss:eur-dissip} (at least, up to small modifications, for instance in~\eqref{eur:ito-tanaka} 
we need to require~$\delta \geq \sfrac{\gamma}{2}$).

\section{Notations and preliminaries}\label{section:preliminaries}

We mostly employ standard general notation. We will work on the 2-dimensional torus $\T^n \cong \R^n / \Z^n \cong [0,1]^n / \sim$. The only specific notation is the one for the $\varepsilon$-restriction of a set $A \subset \T^n $, defined for any $\varepsilon>0$ as
$$
A[\varepsilon] := \{ x \in \T^n:   \dist (x, A^c ) > \varepsilon \}\,.
$$

\medskip

In all the paper we make systematic use of several notions and results concerning deterministic and stochastic flows, and their relation to advection and advection-diffusion equations. We summarize the main definitions and results in the remaining of this section without any attempt to be complete or systematic. The deterministic theory for smooth velocity fields is fully classical, while for the stochastic theory we refer to~\cite{Evans_SDEbook,OB,K84} and also to~\cite{LL19} for the case of nonsmooth velocity fields. 

\medskip

Given a smooth velocity field $u: (0,T) \times \T^n \to \R^n$ we denote by $\XX : (0,T) \times \T^n \to \T^n$ its flow, that is, the solution of the ordinary differential equation (ODE)
\begin{align} \label{d:ODE}
\begin{cases}
\dot{\XX}_{t} = u(t, \XX_{t}) 
\\
\XX_{0}(x_0) = x_0
\end{cases} 
\end{align}
parametrized by the initial datum $x_0 \in \T^n$. We often use the shorthand notation  $ \XX_t  =\XX(t, \cdot)$. The unique solution $\vartheta$ of the advection equation~\eqref{e:intro:advection} with initial datum $\vartheta_{\initial}$ is characterized for any $t\geq0$ by
\begin{align} \label{eq:solutiontoPDE}
\vartheta (t, \cdot) = ( \XX_t )_\sharp (\vartheta_{\initial}),
\end{align} 
where $(\XX_t )_\sharp (\vartheta_{\initial})$ denotes the push-forward measure of $\vartheta_{\initial} \mathcal{L}^d$ through the map $\XX_t$, defined as
$$
(\XX_t )_\sharp (\vartheta_{\initial}) (A) = \int_{\T^n} \mathbbm{1}_{A} (\XX_t (x)) \vartheta_{\initial}(x) dx
\qquad \text{for all Borel sets $A \subset \T^n$.}
$$
We denote by~$\XX_{t,s}$ the flow when the initial condition is assigned at the time $s$, in particular~$\XX_{t}=\XX_{t,0}$.

\medskip

Let $(\Omega, (\mathcal{F}_t)_t, \mathbb{P})$ be a filtered probability space and $\WW_t$ an adapted $\T^n$-valued Brownian motion. We will make an extensive use of the following estimate for the Brownian motion: for every $c,\kappa>0, T> \tilde T \geq0$
\begin{equation}\label{stima:brownian}
\mathbb{P} \left ( \omega \in \Omega : \sup_{t \in [\tilde{T} , T]}  \sqrt{2 \kappa} | \WW_t - \WW_{\tilde T}| \leq c \right ) \geq 1-  2 e^{ \sfrac{-c^2}{2 \kappa (T-\tilde T)}}.
\end{equation}
A stochastic flow is a stochastic process, parametrized by $x_0$, which is a solution of the stochastic differential equation (SDE)
\begin{align} \label{d:sdee}
\begin{cases}
d \XX_{t}^\kappa =   u(t, \XX_{t}^\kappa) dt + \sqrt{2 \kappa}  {d} \WW_{t} 
\\
\XX_{0}^\kappa = x_0.
\end{cases}
\end{align}
For all times $t_1,t_2\geq 0$ the stochastic flow $\XX_t^\kappa$ satisfies the semi-group property $\XX_{t_1+t_2,0}^\kappa = \XX_{t_1+t_2, t_1}^\kappa (\XX_{t_1,0}^\kappa)$. 
If the velocity field is divergence-free, then the stochastic flow is measure-preserving 
\begin{equation}\label{d:sLf:property}
\int_{\T^n}  \mathbbm{1}_{A} (\XX_t^\kappa (x, \omega)) dx = \mathcal{L}^d (A)
\qquad \text{ for all $t \geq 0$ and for all Borel sets $A \subset \T^n$, \ for $\mathbb{P}$-a.e. $\omega \in \Omega$}
\end{equation}

The fundamental It\^o formula asserts that, for every $f \in C^{\infty}((0, T) \times \T^n)$, the stochastic process~$t \mapsto f(\XX^{\kappa}_t)$ is adapted to the filtration $\mathcal{F}_t$ and solves the SDE
\begin{equation}\label{eq:ito}
df(t,\XX^\kappa_t) = \partial_t f (t, \XX^\kappa_t) \ dt + \nabla f (t, \XX^\kappa_t) \cdot  d \XX^\kappa_t +  \kappa \Delta f (t, \XX^\kappa_t)  \ dt \,.
\end{equation}
Moreover, 
the so-called It\^o isometry (see \cite[Corollary 3.1.7]{OB}) gives for any $t\in [0,T]$
\begin{equation}\label{quadratic_variation}
\mathbb{E} \left [  \left| \int_0^t \nabla f(s, \XX^\kappa_s) \cdot  d\WW_s \right|^2 \right ]  = \mathbb{E} \left [ \int_0^t | \nabla f (s, \XX^\kappa_s)  |^2 ds \right ] \,.
\end{equation}  
 
\medskip

Compared to the case of deterministic flows, for stochastic flows it is slightly more technical to define a backward flow. We observe that, if $\WW_t$ is a Brownian motion, then $t \in [s, \infty) \mapsto  \WW_t - \WW_s $ is a Brownian motion as well.
We can define the backward filtration generated by the Brownian motion, namely for any~$s \leq t$ we consider the  $\sigma$-algebra $\mathcal{F}_{s,t} = \sigma (\WW_r - \WW_t : s \leq r \leq   t)$. The backward SDE starting at time~$t$~is
 \begin{equation}\label{backward_flow}
 \begin{cases}
 d \XX_{s,t}^\kappa =  u(t, \XX_{s,t}^\kappa) dt + \sqrt{2 \kappa}  d \WW_{s} 
\\
\XX_{t,t}^\kappa = x_0\,.
 \end{cases}
 \end{equation}
A continuous stochastic process $ s  \in [0,t] \mapsto \XX_{s,t}^\kappa$ is a solution of the backward SDE \eqref{backward_flow} starting at time $t$ if it is adapted to the backward filtration $(\mathcal{F}_{s,t})_s$ and satisfies
$$\XX^\kappa_{s,t} (x_0, \omega) = x_0 + \int_t^s u(\tau, \XX^\kappa_{\tau,t}(x_0, \omega)) d \tau + \sqrt{2 \kappa} (\WW_s (\omega) - \WW_t (\omega)) \qquad \mbox{for } \mathbb{P} \mbox{-a.e.}~\omega, \, x_0 \in \T^n \,.$$
 We say that the map $\XX^\kappa_{\cdot ,t} : [0,t] \times \T^n \times \Omega \to \T^n$ is a backward stochastic flow.

\medskip

We also observe for later use that, by means of the classical Gr\"onwall inequality, it is possible to estimate the distance between the deterministic and the stochastic flows associated to the same velocity field as
\begin{equation}\label{l:gronwall}
| \XX_{t} (x) - \XX_t^\kappa (y, \omega)| \leq \left ( |x- y| + \sqrt{2 \kappa} \sup_{s \in [0,t]} | \WW_s(\omega)| \right) 
\exp \left(\int_0^t \| \nabla u (s, \cdot) \|_{L^\infty (\T^n)}  ds \right)
\end{equation}
for $\mathbb{P}$-a.e. $\omega$ and for every $x,y\in \T^n$.

\medskip

The Feynman-Kac formula provides a representation of the unique bounded solution of the advection-diffusion equation~\eqref{e:intro:advdiff} via the backward stochastic flow:
\begin{equation}\label{feynman_kac_real}
\vartheta_\kappa (t,x) = \mathbb{E}[ \vartheta_{\initial} (\XX^{\kappa}_{0,t}(x ))] 
\qquad \text{for every $t\geq 0$, for a.e. $ x\in \T^n$.}
\end{equation} 
A formula involving the push-forward via the forward stochastic flow (analogous to~\eqref{eq:solutiontoPDE} for the deterministic flow) holds as well: for any $t
\geq 0$
\begin{equation} \label{feynman_kac}
 \int_{\T^n} f(x)  \vartheta_{\kappa} (t,x) \, dx =  \mathbb{E} \int_{\T^n} f(\XX_t^\kappa (x, \cdot ))   \vartheta_{\initial} (x) \, dx \qquad \text{for any $f \in L^\infty(\T^n)$.} 
\end{equation}

\medskip

All of the above results are quite classical in the case of smooth velocity fields. A systematic theory in the case of Sobolev velocity fields is presented in~\cite{LL19}. However, the velocity field in our construction will not have such a regularity across the singular time $t=1$. For a velocity field $u \in L^p ((0,T); L^q(\T^n))$ under the Lady\v zhenskaya--Prodi--Serrin condition $\sfrac{2}{p} + \sfrac{3}{q} <1$ it is proven in~\cite{KR05} that a unique stochastic flow exists. To the best of our knowledge, it has not been explicitly shown in the literature that the measure-preserving property~\eqref{d:sLf:property} holds in case the velocity field additionally is divergence-free. We sketch a proof of this fact in a case adapted to our context in the following theorem.
{
\begin{theorem} \label{prop:representation}
Fix $\kappa >0$ and $\alpha \in (0,1)$. Let $u \in  C^\alpha ((0,T)   \times \T^n) $ be a divergence-free velocity field
and $\vartheta_{in } \in  L^\infty (\T^n)$. Then, the unique bounded solution $\vartheta_\kappa$ to \eqref{e:intro:advdiff} is represented by formula \eqref{feynman_kac} and the unique stochastic flow satisfies the measure-preserving property~\eqref{d:sLf:property}. 
\end{theorem}

\begin{proof}
The proof is classical for a smooth divergence-free velocity field and the general result follows from the approximation theorem \cite[Theorem 5]{FGP}.
\end{proof}}

\section{Construction and main properties of the velocity field}\label{section:parameters}

In this section we introduce all parameters needed in our constructions. Based on the choice of the parameters, we describe the construction of the velocity field in Theorems~\ref{t_main_OC},~\ref{t_mainadvection}, and~\ref{t_main_convol} and the action of the corresponding flow on the solutions. We also collect several useful properties of the velocity field. 

\subsection{Choice of the parameters}\label{ss:para}
 
%


Let $\sfrac{1}{p} + \sfrac{2}{p^\circ} = 1$ and $\alpha + 2 \beta <1$ be as in Theorem~\ref{t_main_OC}. 
Since $p^\circ \leq 4$ and $\beta < \sfrac{1}{2}$, we have $p^\circ \beta < 2$.
We consider parameters $\epsilon , \delta \in (0,\sfrac14)$ sufficiently small such that 
\begin{subequations} \label{c:beta_eps_condition_all}
\begin{align} \label{c:alpha_beta_eps_kappa}
1 - \frac{2 \beta (1 + 3 \epsilon(1+ \delta) ) (1+ \delta)}{1 -  \delta} - \alpha (1 + \epsilon \delta)(1 + \delta)  - \frac{\delta}{8 }>0\,,
\\
\frac{p^\circ \beta (1 + 3\epsilon (1 + \delta)) (1 + \delta )}{1 -  \delta} + \frac{\delta}{8} < 2 \,, \label{c:gamma_eps}
\\
\epsilon \leq \frac{\delta^3}{ 50}\,.  \label{c:eps_delta}
\end{align} 
\end{subequations}
The first two conditions are satisfied if $\epsilon$ and $\delta$ are small thanks to the assumptions $\alpha + 2 \beta <1$ and $p^\circ \beta < 2$ respectively, while for the third condition it is enough to choose $\epsilon$ depending on $\delta$.

Given $a_0 \in (0,1)$ such that 
 \begin{align} \label{c:d_0}
a_0^{{\epsilon \delta^2} } + a_0^{\sfrac{\epsilon \delta}{8} } \leq \frac{1}{20} \,,
 \end{align}
we define
\begin{align}    \label{d:parameters}
  a_{q+1} = a_q^{1 + \delta}, \qquad
\lambda_q =  \frac{1}{2 a_q}.
\end{align}
To be precise, our construction requires also that  $\sfrac{a_q}{a_{q+1}} $ is a multiple of $4$ for every $q$ to preserve the geometry of chessboards, hence the superexponential sequence should be chosen so that $a_{q}/a_{q+1}$ is an integer multiple of $4$ in the range $ [a_q^{- \delta}-4, a_q^{- \delta}]$. This change affects the proofs only in numerical constants in the estimates and in turn makes the reading more technical, hence we avoid it.

Notice that for any $\ell \geq \epsilon \delta $ we have 
\begin{align*}
 \sum_{k \geq q} a_{k}^{\ell}  = \sum_{k \geq 0} a_{ q}^{(1+ \delta)^k \ell} \leq \sum_{k \geq 0} a_{ q}^{(1+ k \delta) \ell}  \leq  2 a_q^\ell ,
 \end{align*}
 that we will use throughout the proofs.
 


We fix the parameter $\gamma>0$ for the scaling in time by
\begin{align} \label{d:gamma}
\gamma = \frac{p^\circ \beta (1 + 3\epsilon (1 + \delta)) (1 + \delta ) }{1 - \delta} + \frac{ \delta}{8} \,.
\end{align}
We fix $m \in \N $ such that $m-1 \geq \sfrac{16}{\delta^2}$ and define the sequence of times
\begin{align*} \label{d:parameters_Tq}
\begin{cases}
t_q = a_q^{ \gamma} & \mbox{for any $q\in\N$}
\\
\overline{t}_q  = a_q^{\gamma -  \gamma \delta}  &  \mbox{for $q$ such that } q \in m\N 
\\
\overline{t}_q = 0 &  \mbox{for $q$ such that } q \not\in m\N 
 \end{cases}
 \end{align*}
 recalling that by $q \in m \N$ we mean that $q$ is a multiple of $m$. We also
 set
\begin{equation}\label{d:parameters_Tq}
T_q = \sum_{j \geq q}^\infty  \overline{t}_j  + 3 \sum_{j \geq q}^\infty  t_j < 1 \quad \mbox{ for any } q \in \N \,.
\end{equation}

The remaining choices of the diffusivity and convolution parameters are specific to each of our three theorems.
In Theorem~\ref{t_main_OC} we fix the diffusivity parameter
\begin{equation}\label{d:AD_parameters_kappa_q}
\tilde{\kappa}_q = a_q^{2 - \frac{\gamma}{1 + \delta} + 4 \epsilon}
\end{equation}
and observe that this entails  for any $ j \leq q$ 
\begin{subequations}\label{d:AD_parameters_kappa_q_estimates}
\begin{align}
\sqrt{\tilde \kappa_q \, a_{j-1}^{ \gamma }} \leq  a_j^{1 + 2 \epsilon}   \,, \label{d:AD_first}
\\
\sqrt{\tilde \kappa_q \, a_q^{ \gamma -  \gamma \delta}} \geq a_q^{ 1-  \frac{\epsilon}{2}} \,, \label{d:AD_second}
\end{align}
\end{subequations}
where the second condition holds thanks to \eqref{c:eps_delta}.
Notice that a necessary condition for $\tilde \kappa _q \to 0$ is that~$\gamma \sim  p^\circ \beta <2$. This condition is implied by $p \geq 2$, but in fact Theorem~\ref{t_main_OC} holds more in general replacing the assumption~$p \geq 2$ by $p^\circ \beta <2$.
In Theorem~\ref{t_mainadvection} and Theorem~\ref{t_main_convol} we set $\beta=0$ and for the remaining parameters we implement the corresponding choices as described above. The convolution parameter in Theorem~\ref{t_main_convol} is set to be
\begin{equation} \label{parameters:choice_mollification}
\eps_q = a_q^{1 + \gamma}\,.
\end{equation} 
In Theorem~\ref{t_mainadvection} we need to consider two diffusivity parameters. The parameter for the convergence to a solution which dissipates the $L^2$ norm is the same as in~\eqref{d:AD_parameters_kappa_q}, while the parameter for the convergence to a  solution which conserves the $L^2 $ norm is
\begin{equation}  \label{parameters:choice_conservative}
\kappa_q = a_q^{2 + 3 \epsilon}\,.
\end{equation}

\subsection{Construction of the velocity field for $0\leq t\leq1$} \label{section:vectorfield}
We begin by defining chessboard functions and chessboard sets, also considering a notion of parity on the chessboards. 
%
%

%
%
%
%
%

%

 The  building blocks $\mathbb{W}$, $\mathbb{\widetilde{W}}$, $\mathbb{\overline{W}} : \T^2\to\R^2$ for the construction of the velocity field are shear flows defined by
$$
\mathbb{ W}(x_1, x_2) = (W(x_2), 0)\,, \quad
\mathbb{\widetilde W} (x_1, x_2)= \Big(0, \frac{1+W (x_1 )}{2}\Big) \,, \quad
\mathbb{\overline W} (x_1, x_2)= \Big(0,\frac{1-W (x_1 )}{2}\Big)\,,
$$ 
where the function $W: \T \to \R$ is defined as follows
$$
W(z)= 
 \begin{cases}
1 & \text{ if } z \in [0,1/2),
 \\
  -1 & \text{ if } z \in [1/2, 1),
 \end{cases}
$$
(and extended by periodicity).


Recalling the definition of the $\{T_q\}_q$ in~\eqref{d:parameters_Tq}  we define
  \begin{subequations} \label{d:timeintervals}
  \begin{align} 
  \mathcal{I}_{-1} = (0,1-T_0],
  \\
   \mathcal{I}_{q, 0} =   (1-T_q ,1-T_q + \overline{t}_q],
   \\
   \mathcal{I}_{q, i} = (1-T_q + \overline{t}_q + (i-1) t_q,1-T_q  + \overline{t}_q +i t_q],
  \end{align}
  \end{subequations}
  for any $q \in \N$ and $i=1,2,3$
  and we define also ${\mathcal{I}}_q = \cup_{i=0}^3 \mathcal{I}_{q,i}$. 
  Similarly we define the reflected intervals
$$
\mathcal{J}_{-1} = [1+T_0, 2) 
\qquad\text{ and }\qquad
\mathcal{J}_{q, i} = 2 - \mathcal{I}_{q,i}
$$  
for any $q \in \N$ and for $i=0,1,2,3$ and ${\mathcal{J}}_q = \cup_{i=0}^3 \mathcal{J}_{q,i}$. We notice that 
$$
\bigcup_{q=0}^\infty \bigcup_{i=0}^3 \mathcal{I}_{q, i} \cup \mathcal{J}_{q, i} \cup \mathcal{I}_{-1} \cup \mathcal{J}_{-1} = (0,2) \setminus \{ 1 \} \,.
$$

We now define the velocity field on the time intervals $\mathcal{I}_{q, i}$
. First of all, we set
$$
u(t, \cdot) \equiv 0 
\qquad \text{for $t \in \mathcal{I}_{-1}
$}
$$
and
$$
u(t, \cdot) \equiv 0 
\qquad \text{for $t \in \mathcal{I}_{q,0} \cup \mathcal{I}_{q,1}
$ and for any $q \in \N$.}
$$

Fix a convolution kernel $\tilde \psi \in C^\infty_c (-2,2)$ which we assume to be bounded by $1$ and with gradient bounded by $1$. For any $q \in \N$ we define the rescaled kernels
\begin{equation}\label{e:psiscaled}
\psi_q (x_1 , x_2) = \lambda_q^{2 + 2 \epsilon \delta} \tilde \psi( \lambda_q^{1 + \epsilon \delta} x_1) \tilde \psi( \lambda_q^{1 + \epsilon \delta} x_2) \,.
\end{equation}

We further define suitable cut-off functions in time. Let $\eta_{q,2} \in C^{\infty}_c ( \mathcal{I}_{q,2}[\sfrac{a_q^{\gamma}}{6}])$ and $\eta_{q,3} \in C^{\infty}_c ( \mathcal{I}_{q,3}[\sfrac{a_q^{\gamma}}{6}])$ 
be nonnegative functions such that $\int_{\mathcal{I}_{q,2}} \eta_{q,2} = \int_{\mathcal{I}_{q,3}} \eta_{q,3} = \sfrac{a_q^\gamma}{2}$ and there exist constants $C_k >0$ (depending only on $k \in \N$) such that  
 \begin{align} \label{d:convolution_time}
\| \eta_{q,i} \|_{C^k} \leq C_k a_q^{- k \gamma} 
\qquad \text{for $i=2,3$ and $k \in \N$.},
 \end{align}
 where we can fix $C_0 =1$.
%

  \begin{figure}[h]
  \centering
  \includegraphics[width=11cm]{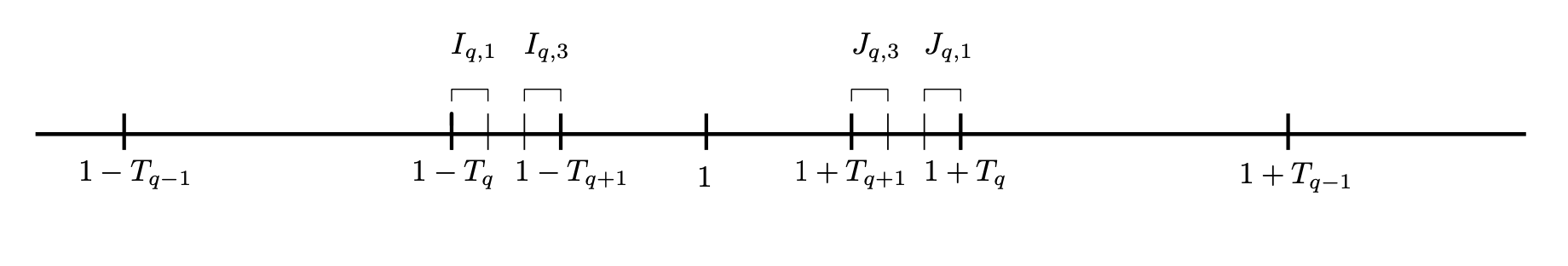}
  \caption{The time intervals $\mathcal{I}_q$ and $\mathcal{J}_q$ for $q \not\in m\N$. 
  }
  \end{figure}
  \begin{figure}[h]
  \centering
  \includegraphics[width=11cm]{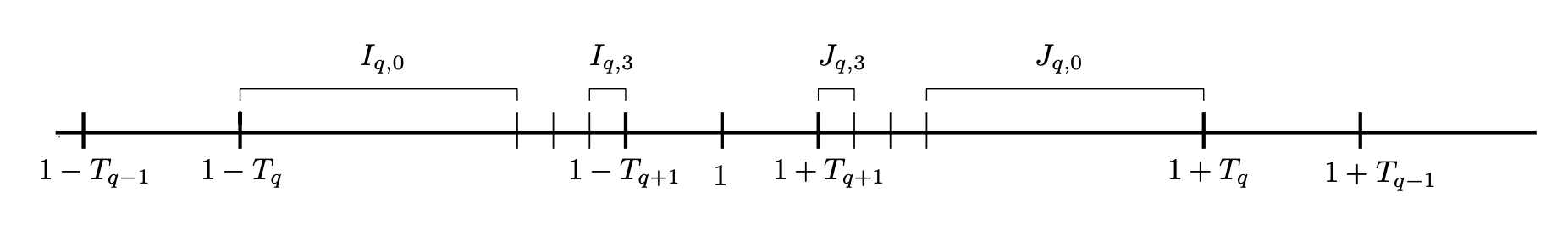}
  \caption{The time intervals $\mathcal{I}_q$ and $\mathcal{J}_q$ for $q \in m\N$. }
  \end{figure}

Let us denote by $\lfloor \xi \rfloor $ the largest integer smaller or equal than the real number $\xi$. Given the discontinuous shear flow
 \begin{align} \label{d:vectorfield_3} 
 w_{q+1,2}(x)=
 \begin{cases}
a_q^{1- \gamma} \mathbb{W}(\lambda_{q+1} x) &  \text{ for $x=(x_1,x_2)$ with $\left \lfloor \sfrac{x_2}{a_q} \right\rfloor $ even,}
\vspace{0.1cm}
 \\
-  {a_q^{1- \gamma}} \mathbb{W} (\lambda_{q+1} x) & \text{ for $x=(x_1,x_2)$ with $\left \lfloor \sfrac{x_2}{a_q} \right \rfloor$ odd,}
  \end{cases}
 \end{align}
we set
$$
u(t,x) = \eta_{q,2} (t) w_{q+1,2} \star \psi_{q+1} (x)
\qquad \text{for $x \in \T^2$ and $t \in \mathcal{I}_{q,2}$.}
$$
Analogously, given the discontinuous shear flow
\begin{align*} w_{q+1,3}(x)=
 \begin{cases}
 2 a_{q+1} a_q^{-\gamma} \mathbb{\widetilde W}(\lambda_{q+1} x) &  \text{ for $x=(x_1,x_2)$ with $\left \lfloor \sfrac{x_1}{a_q} + \sfrac{1}{2} \right \rfloor$ even,}
  \vspace{0.1cm}
 \\
2 a_{q+1} a_q^{-\gamma} \mathbb{\overline{W}} (\lambda_{q+1} x) & \text{ for $x=(x_1,x_2)$ with $\left \lfloor \sfrac{x_1}{a_q} + \sfrac{1}{2} \right \rfloor$ odd,}
  \end{cases}
 \end{align*}
we set 
$$
u(t,x ) = \eta_{q,3} (t) w_{q+1, 3} \star \psi_{q+1}(x)
\qquad \text{for $x \in \T^2$ and $t \in \mathcal{I}_{q,3}$.}
$$
 We highlight that the convolution is taken at a spatial space scale comparable to  $a_{q+1}^{1+ \epsilon \delta} $ which is much smaller than the scale $a_{q+1}$ of the velocity field. Therefore, the structure of the velocity field is preserved up to a very small error.

\begin{figure}[h]
\includegraphics[width=8cm]{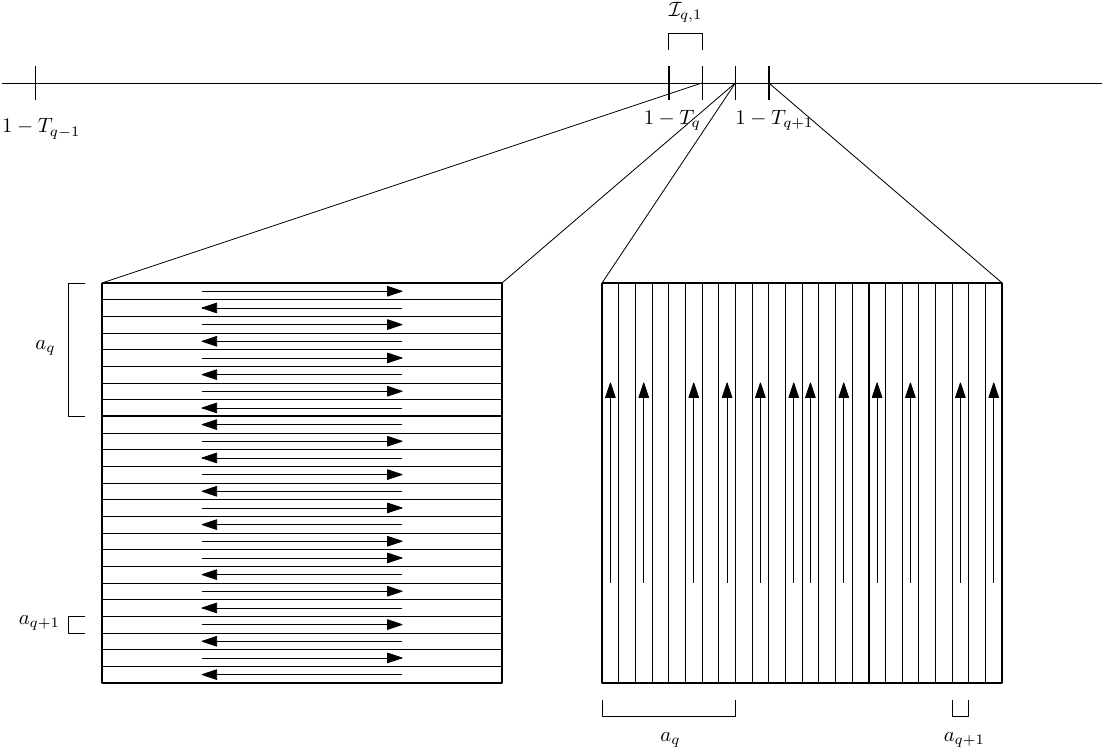}
\caption{The velocity fields $w_{q+1,2}$ and $w_{q+1,3}$  on the time intervals  $\mathcal{I}_{q,2}$ and $\mathcal{I}_{q,3}$ respectively. In the picture we choose $\sfrac{a_q}{a_{q+1}}=8$. }
\end{figure}

\subsection{Construction of the velocity field for $1 \leq t \leq 2$} In Theorem~\ref{t_main_OC} the velocity field is defined only for $0 \leq t \leq 1$.
For Theorem~\ref{t_mainadvection} we simply extend the velocity field  $u$ by reflection, namely  we set
 \begin{equation}\label{vectorfield_VV}
  u(t,x) = - u(2-t,x)
  \qquad \text{for $t \in [1,2]$.}
\end{equation} 
For Theorem~\ref{t_main_convol} in addition to the reflection we add a new velocity field $u_{\swap}$, namely
\begin{equation}\label{vectorfield_convolution} 
 u(t,x ) =  u_{\swap}(t,x) - u(2-t,x)
   \qquad \text{for $t \in [1,2]$.}
\end{equation}
The swap velocity field $u_{\swap}$ is defined as follows. We first define a discontinuous shear flow
$$
w_{q+1, {\swap}}(x)=  2 a_q^{1- \gamma}
 \mathbb{W}(\lambda_{q+1} x) 
$$
and then set 
$$
u_{\swap} (t,x) = \eta_{q,1}(t) w_{q+1, {\swap}} \star \psi_{q+1} (x)
\qquad \text{for $x \in \T^2$ and $t \in \mathcal{J}_{q,1}$,}
$$
where $\eta_{q,1} \in C^\infty_c (\mathcal{J}_{q,1} [\sfrac{a_q^\gamma}{6}])$ is a nonnegative function with $\int_{\mathcal{J}_{q,1}} \eta_{q,1} = \sfrac{a_q^\gamma}{2}$ and $\| \eta_{q,1} \|_{C^k} \leq C_k a_q^{- k \gamma}$ for any~$k \in \N$, with $C_0 =1$.
 \begin{figure}[h]
\includegraphics[width=11cm]{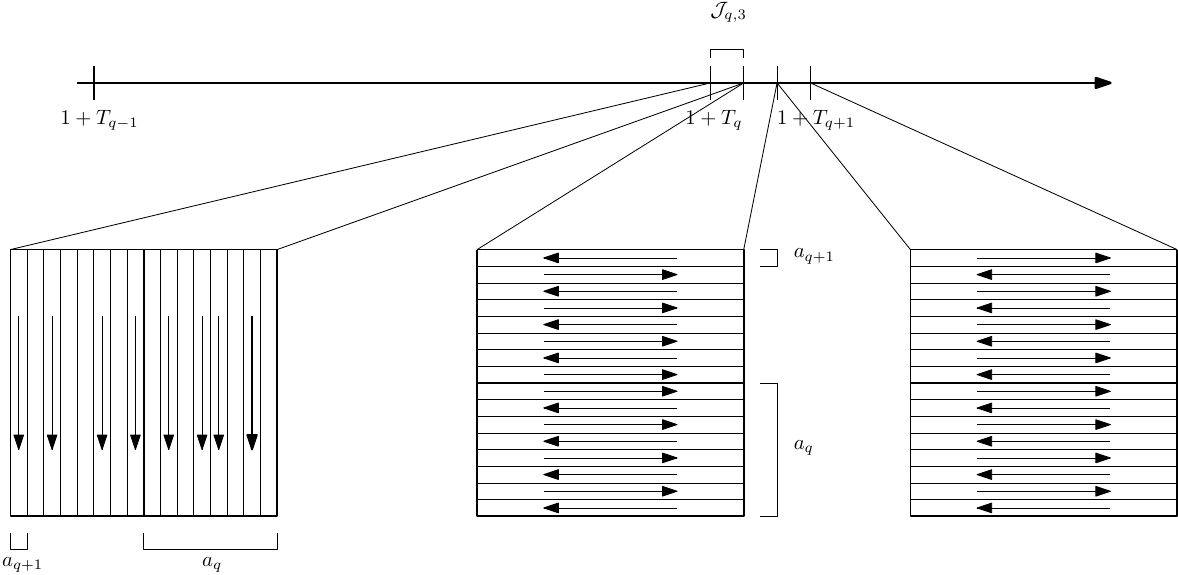}
\caption{
The velocity fields $- w_{q+1,3}$, $- w_{q+1,2}$ and $w_{q+1,{\swap}}$ on the time intervals  $\mathcal{J}_{q,3}$, $\mathcal{J}_{q,2}$, and $ \mathcal{J}_{q,1}$ respectively. In the picture we choose $\sfrac{a_q}{a_{q+1}}=8$.}
\end{figure}


\subsection{Initial datum} \label{section:initial}

 The initial datum is the same for Theorem \ref{t_main_OC}, Theorem \ref{t_mainadvection} and Theorem \ref{t_main_convol}.
\begin{definition}[Chessboards] \label{d_initial datum}
Let $\lambda \in \N$ and $\vartheta_0 : \T^2 \to \R$ be defined as 
\begin{align} \label{d:vartheta_0}\vartheta_0(x_1, x_2) =
\begin{cases}
1 & \text{ if } x_1, x_2 \in [0,1/2)  \text{ or } x_1, x_2 \in [1/2,1)
\\
-1 & \text{ otherwise}
\end{cases}
\end{align} 
and extended by periodicity. We say that the function $\vartheta^{(1)} (x)= \vartheta_0 ( \lambda x )$ is the {\em even chessboard function} of side $(2\lambda)^{-1}$ and the function $\vartheta^{(2)} (x)= - \vartheta_0 ( \lambda x )$ is the {\em odd chessboard function} of side $(2\lambda)^{-1}$.
We further say that the set $ \widetilde{A}_{\lambda} =  \supp\{ 1+ \vartheta^{(1)} \}$ is the {\em even chessboard set} of side $(2 \lambda)^{-1}$ and $\widetilde{B}_{\lambda} = \supp\{ 1+ \vartheta^{(2)} \}$ is the  {\em odd chessboard set} of side $(2 \lambda)^{-1}$.
\end{definition} 
We consider as initial datum a smoothed version of the even chessboard function of side $a_0$: for~$\vartheta_0\in L^\infty (\T^2)$ defined as in~\eqref{d:vartheta_0} we consider 
\begin{equation} \label{initialdatum}
\vartheta_{\initial} =   \vartheta_0 (\lambda_0 \, \cdot) \ast \psi_0,
\end{equation}
where  $\psi_0 (x) = \lambda_0^{2 + 2 \epsilon \delta} \psi(\lambda_0^{1 + \epsilon \delta } x)$ is the convolution kernel in the spatial variable (used also in Section~\ref{section:vectorfield}).

 \subsection{Main properties of the velocity field}

The action of the velocity field has been informally described in Section~\ref{section:vector_geometric}. Up to small (and quantified) errors, on each time interval ${\mathcal{I}}_q$ (before the critical time $t=1$) the scale of the solution is lowered from $a_q$ to $a_{q+1}$, while on each time interval ${\mathcal{J}}_q$ (after the critical time~$t=1$) the scale of the solution is increased from $a_{q+1}$ to $a_q$. The swap velocity field $u_{\swap}$ (only required in Theorem~\ref{t_main_convol}) swaps the parity of the chessboard on each time interval ${\mathcal{J}}_q$. 

In order to quantify the errors due to the regularizations of the velocity field, we define an $a_q^{1 + \epsilon \delta }$-restricted version of the even/odd chessboard sets of side $a_q$ and a ``good set''. Recalling Definition~\ref{d_initial datum}, we define
\begin{equation} \label{d:G_sets}
A_q = \widetilde{A}_{\lambda_q}[5 a_q^{1+ \epsilon \delta}] \,, \qquad B_q = \widetilde{B}_{\lambda_q}[5 a_q^{1+ \epsilon \delta}] \,, \qquad G_q = A_q \cup B_q \,.
\end{equation} 
We observe that  for any $q\in\N$ there hold
\begin{align} \label{stima:G_q}
\mathcal{L}^2 (A_q) = \mathcal{L}^2 (B_q)  \geq \frac{1}{2} - 10 a_q^{\epsilon \delta} 
\qquad \text{ and } \qquad
 \mathcal{L}^2(G_q) \geq 1 - 20 a_q^{\epsilon \delta}.
\end{align}
The above estimates are proved by elementary geometric considerations. For instance, $\mathcal{L}^2 (G^c_{q})$ is estimated by the measure of $a_q^{-1}$ vertical stripes of measure $10 a_q^{1+ \epsilon \delta}$ plus $a_q^{-1}$ horizontal stripes of measure $10 a_q^{1+ \epsilon \delta}$.
Finally, we observe that $\psi_{q}$ is compactly supported in a ball of radius $2 (\lambda_q^{1+ \epsilon \delta})^{-1} \leq 5 a_q^{1+ \epsilon \delta} $, which implies that $w_{q, i} \star \psi_q (x)$ is locally constant in the  set $G_q$.

\begin{remark}[Regularity of the velocity field] \label{remark:vectorfield}
The velocity field $u$ constructed above and extended for times~$1 \leq t \leq 2$ either with formula \eqref{vectorfield_VV} or \eqref{vectorfield_convolution}  is locally smooth away from the singular time $t=1$, namely 
$$
u \in C^\infty_{\loc}(((0,2) \setminus \{ 1\}) \times \T^2)\,,
$$
and for any time $t$ it is a shear flow and therefore divergence-free.

Moreover, the velocity field $u$ enjoys the following estimates: for any $q\in \N$
\begin{subequations}\label{s:vectorfield_I_q5_new}
\begin{align}
& \| u \|_{L^\infty (( {\mathcal{J}}_{q,1}) \times \T^2)}   \leq 2  a_q^{ 1 - \gamma}   \hspace{1.05cm} \quad \quad \quad \quad  \| \nabla u \|_{L^\infty (( {\mathcal{J}}_{q,1}) \times \T^2)} \leq  a_{q}^{- \gamma}  \lambda_q^{-1} \lambda_{q+1}^{1 + \epsilon \delta} \label{s:vectorfield_J_q_1}
\\
&   \| u \|_{L^\infty (( {\mathcal{I}}_{q,2} \cup \mathcal{J}_{q,2}) \times \T^2)}   \leq   a_q^{1 - \gamma}    \quad \quad \hspace{1.25cm}  \| \nabla u \|_{L^\infty (( {\mathcal{I}}_{q,2} \cup \mathcal{J}_{q,2}) \times \T^2)} \leq  a_q^{- \gamma} \lambda_q^{-1} \lambda_{q+1}^{1 + \epsilon \delta}   \label{s:vectorfield_I_q3}
\\
& \| u \|_{L^\infty (( {\mathcal{I}}_{q,3} \cup \mathcal{J}_{q,3}) \times \T^2)}   \leq 2  a_{q+1} a_q^{ - \gamma}   \hspace{1.2cm}  \| \nabla u \|_{L^\infty (( {\mathcal{I}}_{q,3} \cup \mathcal{J}_{q,3}) \times \T^2)} \leq  a_{q}^{- \gamma}  \lambda_{q+1}^{ \epsilon \delta}.
\label{s:vectorfield_I_q5}
\end{align}
\end{subequations}
Hence, by interpolation, for any $q\in \N$ we have 
\begin{equation}
\label{eqn:ucalphaspace}
\| u \|_{L^\infty (( {\mathcal{I}}_{q} \cup \mathcal{J}_{q}) ; C^\alpha( \T^2) )}   \lesssim a_q^{1 - \gamma}  a_{q+1}^{-\alpha (1+\epsilon \delta)}  =  a_q^{1 - \gamma - \alpha (1+\epsilon \delta) (1+\delta)}.
\end{equation}

The velocity field enjoys some additional regularity in time thanks to \eqref{d:convolution_time}, namely for any $k\in \N$ with  $1- \gamma -k \gamma >0$ it holds
\begin{equation}
\label{eqn:time-reg}
\| u \|_{L^\infty(\T^2 ;C^k(0,2))} \lesssim_k  \sup_{q} a_q^{1 - \gamma -  k \gamma} < \infty\,.
\end{equation}

Based on the previous properties, the regularity of the velocity field in Theorems~\ref{t_mainadvection} and~\ref{t_main_convol}  follows, namely given any $\alpha \in (0,1)$ the previous construction performed with the choice $\beta=0$ provides a velocity field in $C^{\alpha}((0,2) \times \T^2)$. Indeed, for $\beta= 0 $ we have $\gamma = \delta/8$ and thanks to \eqref{c:alpha_beta_eps_kappa} the right-hand side in~\eqref{eqn:ucalphaspace} is uniformly bounded in $q$, showing the desired H\"older continuity in space. 
As regards the H\"older continuity in time, interpolating between~\eqref{eqn:time-reg} with $k=0$ and $k=1$ we get
$
\| u \|_{L^\infty(\T^2 ;C^\alpha(0,2))}< \infty\,.
$
%
\end{remark}

\begin{remark}[The action of the flow without  the swap velocity field]\label{remark:locallyconstant:diffusion}
 Let us fix the velocity field constructed in Section \ref{section:vectorfield} and extended for $1 \leq t \leq 2$ by formula \eqref{vectorfield_VV}, a constant $c>0$ and an integer $k \in \N$. Consider the flow $\XX^q$ associated to the smooth velocity field $u_q = u \mathbbm{1}_{[1-T_q, 1+T_q]^c}$. Then, recalling the definition of~$G_k$ in~\eqref{d:G_sets}, the following properties hold:
\begin{equation}\label{c:vectorfield_constant_locally}
\XX^q_{1-T_k} (x) \in G_{k+1}[c] \ \ \ \Longrightarrow \ \ \  u(t, \XX^q_t (x) + v(t)) = u(t, \XX_t^{q}(x)) \qquad \mbox{for any $t \in {\mathcal{I}}_k$}
\end{equation}
and
\begin{equation}
\XX^q_{1+ T_{k+1}} (x) \in G_{k+1}[c] \ \ \ \Longrightarrow \ \ \  u(t, \XX^q_t (x) + v(t)) = u(t, \XX_t^{q}(x)) \qquad \mbox{for any $t \in {\mathcal{J}}_k$} \,,
\end{equation}
for any path $v: [0,2] \to \T^2 $ such that
 $\| v \|_{L^\infty} \leq c$.
%
 Furthermore from the construction of the velocity field we have  (recalling the definition of $A_k$ and $B_k$ as in \eqref{d:G_sets})
\begin{subequations}\label{remark:action}
\begin{align}
\XX_{1- T_k}^q ( x )  \in A_{k} \cap G_{k+1} \ \ \  \Longrightarrow \ \ \  \XX_{1- T_{k+1}}^q ( x )  \in A_{k+1},
\\
\XX_{1- T_k}^q ( x )  \in B_{k} \cap G_{k+1} \ \ \  \Longrightarrow \ \ \  \XX_{1- T_{k+1}}^q ( x )  \in B_{k+1}.
\end{align}
\end{subequations}
and
\begin{subequations}\label{remark:no_swap_action}
\begin{align}
\XX_{1+ T_{k+1}}^q ( x )  \in A_{k+1} \cap G_{k+1} \ \ \  \Longrightarrow \ \ \  \XX_{ 1+T_k}^q ( x )  \in A_{k},
\\
\XX_{1+ T_{k+1}}^q ( x )  \in B_{k+1} \cap G_{k+1} \ \ \  \Longrightarrow \ \ \  \XX_{ 1+T_k}^q ( x )  \in B_{k}.
\end{align}
\end{subequations}
\end{remark}

\begin{remark}[The action of the flow with the swap velocity field] \label{remark:locallyconstant:convolution}
If we consider the velocity field $u$ constructed in Section \ref{section:vectorfield} and extended for $1 \leq t \leq 2$ by formula \eqref{vectorfield_convolution} (by adding the swap velocity field) and denote by $\XX^q$ the flow of the smooth velocity field 
$u_q = u \mathbbm{1}_{[1-T_q, 1+T_q]^c}$, then the same properties as in Remark~\ref{remark:locallyconstant:diffusion} hold for $0 \leq t \leq 1$,  since the velocity field is the same for such times.  {For any $c>0$ and $k \in \N$}, we  additionally have
$$
\XX^q_{1+ T_{k+1}} (x) \in G_{k+1}[c] \ \ \  \Longrightarrow \ \ \ u(t, \XX_t^{q}(x)) = u(t, \XX^q_t (x) + v(t))\,,
$$ 
for any $t \in {\mathcal{J}}_k$ and  for any path $v: [0,2] \to \T^2 $ such that
 $\| v \|_{L^\infty} \leq c$ and  
\begin{subequations} \label{remark:swap_action}
\begin{align}
\XX_{1+ T_{k+1}}^q ( x )  \in A_{k+1} \cap G_{k+1} \ \ \  \Longrightarrow \ \ \  \XX_{ 1+T_k}^q ( x )  \in B_{k},
\\
\XX_{1+ T_{k+1}}^q ( x )  \in B_{k+1} \cap G_{k+1} \ \ \  \Longrightarrow \ \ \  \XX_{ 1+T_k}^q ( x )  \in A_{k} ,
\end{align}
\end{subequations}
differently from~\eqref{remark:no_swap_action} due to the presence of the swap velocity field.


\end{remark}


 \section{Proof of Theorem~\ref{t_main_convol}} \label{section:prooftheorem}
 
 Given $\alpha\in [0,1[$ as in the statement of Theorem~\ref{t_mainadvection}, we choose $\beta=0$ and recall that all relevant parameters have been correspondingly fixed in Section~\ref{ss:para}. In particular, we recall that
$$
a_0^{\sfrac{\epsilon \delta}{8} } \leq \frac{1}{20} \,, \quad
\epsilon \leq \frac{\delta^3}{ 50}  \,, \quad 
\gamma = \frac{\delta}{8} \,, \quad
\sigma_q = a_q^{1 + \gamma} \,, \quad 
a_{q+1} =a_q^{1+ \delta}\,,
$$
and that the length of the time intervals $\mathcal{I}_{q,j}$ is $a_q^{\gamma - \gamma \delta}$ (for $q \in m\N$  and $j=0$), $0$ (for $q \not\in m\N$ and $j=0$) and $a_q^\gamma$ otherwise. 

We consider the initial datum $\vartheta_{\initial}$ defined in~\eqref{initialdatum}  and the velocity field $u \in C^\alpha ((0,2) \times \T^2)$  constructed in Section~\ref{section:vectorfield} for times $0 \leq t \leq 1$ and extended for times~$1 \leq t \leq 2$ by formula \eqref{vectorfield_convolution} (in particular we make use of the swap velocity field). For the regularity of the velocity field  see Remark \ref{remark:vectorfield}. 

For any $\sigma_q$  we let~$\vartheta_{\sigma_q}$ be the unique solution of the advection equation~\eqref{e:intro:advection} with velocity field~$u \star \varphi_{\sigma_q}$ and initial datum~$\vartheta_{\initial}$ as in \eqref{initialdatum}, where $\varphi \in C^\infty_c ((-1,1) \times B(0,1))$ is the convolution kernel in space-time fixed in the statement of the theorem. We observe that~$\| \vartheta_{\sigma_q} \|_{L^\infty((0,2) \times \T^2)}  \leq  \| \vartheta_0 \|_{L^\infty(\T^2)} \leq 1$ for any $q \in \N$. 

In the following lemma we prove an $L^\infty$ bound on the mollified velocity field $u \star \varphi_\sigma$ for $\sigma \in [a_{q+1},a_q]$ and $t\in (1-T_q,1+T_q)$. The choice of $\sigma_q$ as in~\eqref{parameters:choice_mollification} guarantees the uniform smallness of the velocity field in such time interval.


\begin{lemma} \label{l_1} 
Let $u: (0,2) \times \T^2 \to \R^2$ be the velocity field defined above 
and let $\varphi \in C^\infty_c ((-1,1) \times B(0,1))$ be a convolution kernel in space-time. Then, for any $q \in \N$, $a_{q+1} \leq \sigma \leq a_q$ we have
 $$ \| u \star \varphi_\sigma \|_{L^\infty((1-T_q, 1+ T_q ) \times \T^2)} \leq  \frac{ \overline{C} \| \varphi \|_{C^1} a_{q+1} a_q^{1- \gamma}}{\sigma},$$
 where $\overline{C}>0$ is a universal constant.
\end{lemma}

\begin{proof}
Let us fix $q$ and $\sigma$ as in the statement. We first estimate $(u \star \varphi_\sigma) ( t, x)$ for  $x \in \T^2$ and $t \in \mathcal{I}_{q,2}$. Recalling \eqref{d:vectorfield_3} we define  the set
$$ Z =
\{(y_1, y_2) \in B (x, 2 \sigma)   : w_{q+1, 2}(y_2 ) = - w_{q+1,2}( y_2 + a_{q+1}) \} .$$
It holds that $\mathcal{L}^2 (Z^c \cap B(x, 2 \sigma)) \leq 64 a_{q+1} \sigma$ since the set  $Z^c \cap B(x, 2 \sigma)$ is made of at most $16$ horizontal stripes of length $4\sigma$ and height $a_{q+1}$ thanks to the construction of $w_{q+1,2}$. Using~\eqref{s:vectorfield_I_q3} we estimate
\begin{align*}
\left | (u \star \varphi_\sigma) (t, x) \right | & = \frac{1}{2}  \left | \int_{-1}^1 \int_{\T^2} \left [ u(s, y)\varphi_\sigma (t -s, x - y   )  +
 u(s, y + a_{q+1} {\bf e}_2 )\varphi_\sigma (t -s, x - y + a_{q+1}  {\bf e}_2 ) \right ] dy ds  \right |
\\
& \leq  \frac{1}{2} \left |  \int_{-1}^1 \int_{Z} \left [ u(s, y)\varphi_\sigma (t -s, x - y   ) +   u(s, y + a_{q+1} {\bf e}_2)\varphi_\sigma (t -s, x - y + a_{q+1} {\bf e}_2  ) \right ] dy ds \right |
\\
& \quad + \frac{1}{2}  \left | \int_{-1}^1 \int_{Z^c \cap B(x, 2 \sigma)} \big[ u(s, y)\varphi_\sigma (t -s, x - y   ) + u(s, y + a_{q+1} {\bf e}_2)\varphi_\sigma (t -s, x - y + a_{q+1} {\bf e}_2) \big] dy ds \right |
\\
& \leq \frac{ \| u \|_{L^\infty (\mathcal I_{q,2} \times \T^2)}}{2} \int_{-1}^1 \int_{Z} | \varphi_\sigma (t -s, x - y   )  - \varphi_\sigma (t -s, x - y + a_{q+1} {\bf e}_2  ) | dy ds
\\ 
& \quad +  2 \mathcal{L}^2 (Z^c \cap B(x, 2 \sigma))  \| u \|_{L^\infty (\mathcal I_{q,2} \times \T^2)}   \int_{-1}^1 \| \varphi_\sigma(s,\cdot) \|_{L^\infty(\T^2)} ds 
\\
& \leq \frac{a_q^{1- \gamma} 4 \pi \| \nabla \varphi \|_{L^\infty} a_{q+1}}{ \sigma}   + \frac{256 a_{q+1} \sigma a_q^{1- \gamma} \| \varphi \|_{L^\infty}}{\sigma^2}.
\end{align*}
The same estimate holds for $t \in \mathcal{J}_{q,2}$. For $t \in \mathcal{J}_{q,1}$ the estimate is similar (in fact, even easier, since the velocity field is exactly periodic and therefore the corresponding set $Z^c \cap B(x , 2 \sigma)$ is empty).
Finally, we observe that for $t \in I= \mathcal{I}_{q,3} \cup \mathcal{J}_{q,3} \cup [1-T_{q+1}, 1+ T_{q+1}] $ the estimate follows from the bound 
$$
\|u \|_{L^\infty (I \times \T^2)} \leq 2 a_{q+1} a_q^{- \gamma} \leq \frac{256 \| \varphi \|_{C^1} a_{q+1} a_q^{1- \gamma}}{\sigma} \,. \qedhere
$$
\end{proof}


\begin{proof}[Proof of Theorem~\ref{t_main_convol}]
In the first three steps of the proof {we fix $q \in \N$ sufficiently large} and we describe (up to small errors that we explicitly quantify) the flow $\XX^{\sigma_q}$ on the time intervals $[0,1-T_q]$, $[1-T_q, 1+T_q]$, and $[1+ T_q, 2]$, respectively. In the last step we prove that the two subsequences $\vartheta_{\sigma_{2q}}$ and $\vartheta_{\sigma_{2q+1}}$ cannot converge (with respect to the weak$^*$ topology) to the same limit {as $q \to \infty$.}

%
%

\medskip

\noindent{\bf Step 1: $\XX^{\sigma_q}$ almost preserves the chessboards in the time interval $[0, 1-T_q]$. } {\sl For every $j+1 \leq q $ the flow $\XX^{\sigma_q}$ satisfies 
%
$$
\XX^{\sigma_q}_{1- T_j} ( x )  \in A_{j} \cap G_{j+1}[\sigma_q] \ \ \  \Longrightarrow \ \ \  \XX^{\sigma_q}_{1- T_{j+1}} ( x )  \in A_{j+1}[\sigma_q] 
$$
and
$$
\XX^{\sigma_q}_{1- T_j} ( x )  \in B_{j} \cap G_{j+1}[\sigma_q] \ \ \  \Longrightarrow \ \ \  \XX^{\sigma_q}_{1- T_{j+1}} ( x )  \in B_{j+1}[\sigma_q] \,.
$$
}

\medskip

We claim that for any $x \in G_{j+1}[\sigma_q]$ it holds
\begin{equation}\label{e:integrated}
\int_{1-T_j}^{1-T_{j+1}} u_{\sigma_q} (t, \XX^{\sigma_q}_{t,1-T_j} (x)) dt
= \int_{1-T_j}^{1- T_{j+1}} u(t, \XX_{t,1-T_j}(x)) dt \,.
\end{equation}
Indeed, we first observe that
$$
\int_{\mathcal{I}_{j,2}} u_{\sigma_q} (t, \XX^{\sigma_q}_{t,1-T_j} (x)) dt
= \begin{pmatrix} \pm \frac{a_j}{2} \\ 0 \end{pmatrix}
= \int_{\mathcal{I}_{j,2}} u(t, \XX_{t,1-T_j}(x)) dt \,,
$$
where we used that $u=0$ for $t \not \in \mathcal{I}_{j,2}[\sfrac{a_j^\gamma}{6}]$ and $\sigma_q = a_q^{1+\gamma} \leq \sfrac{a_q^\gamma}{6}$;
the $\pm$ sign depends on the strip to which the point $x$ belongs (recall~\eqref{d:vectorfield_3}).
Since $\sfrac{a_j}{2}$ is a multiple of $a_{j+1}$ we have that
$$
\XX^{\sigma_q}_{1-T_j+\overline{t}_j+2t_j,1-T_j} (x) = \XX_{1-T_j+\overline{t}_j+2t_j,1-T_j} (x) \in G_{j+1}[\sigma_q] \,.
$$
Similarly, we have
$$
\int_{\mathcal{I}_{j,3}} u_{\sigma_q} (t, \XX^{\sigma_q}_{t,1-T_j} (x)) dt
= \begin{pmatrix} 0 \\ \frac{a_{j+1} \pm a_{j+1}}{2} \end{pmatrix}
= \int_{\mathcal{I}_{j,3}} u(t, \XX_{t,1-T_j}(x)) dt \,,
$$
which in particular shows~\eqref{e:integrated}. 
%
%
Equality~\eqref{e:integrated} implies that $\XX^{\sigma_q}_{1-T_{j+1},1-T_j} (x) = \XX_{1-T_{j+1},1-T_j}(x)$ and in particular the two statements claimed in Step~1 follow from the corresponding properties \eqref{remark:action} for the flow $\XX$.

\medskip
 
\noindent{\bf Step 2: Shortness of the trajectories of $\XX^{\sigma_q}$ in the time interval $[1-T_q, 1+T_q]$.} {\sl For every~$x \in \T^2$ the flow $\XX^{\sigma_q}$ satisfies
$$
| \XX^{\sigma_q}_{1+T_q} (x) -  \XX^{\sigma_q}_{1-T_q} (x) | \leq  \sigma_q \,.
$$
}

Indeed, by Lemma~\ref{l_1} we can bound the displacement for $t \in [1-T_q, 1+ T_q]$ of the integral curves of the velocity field $u \star \varphi_\sigma$ by
$$
2T_q \| u \star \varphi_\sigma \|_{L^\infty((1-T_q, 1+ T_q ) \times \T^2)}
\leq
2 T_q \frac{\overline{C} C a_{q+1} a_{q}^{1- \gamma}}{\sigma_q}  
\leq 
2 \overline{C} C a_{q}^{1 + \delta  -2 \gamma }  
\leq 
a_q^{1 + \gamma} 
= 
\sigma_q\,,
$$
where the last inequality holds thanks to $\delta \geq 4 \gamma$ and assuming that $q$ is sufficiently large to have $a_q^{- \gamma} \geq 2  \overline{C} C$.

\medskip

\noindent{\bf Step 3: $\XX^{\sigma_q}$ swaps the chessboards in the time interval $[1+T_q, 2]$. } {\sl For every $j+1 \leq q $ the flow~$\XX^{\sigma_q}$ satisfies 
$$
\XX^{\sigma_q}_{1+ T_{j+1}} ( x )  \in A_{j+1} \cap G_{j+1}[\sigma_q] \ \ \  \Longrightarrow \ \ \  \XX^{\sigma_q}_{1+ T_{j}} ( x )  \in B_{j} [\sigma_q] 
$$
and 
$$
\XX^{\sigma_q}_{1+ T_{j+1}} ( x )  \in B_{j+1} \cap G_{j+1}[\sigma_q] \ \ \  \Longrightarrow \ \ \  \XX^{\sigma_q}_{1+ T_{j}} ( x )  \in A_{j} [\sigma_q] \, .
$$
}

This is shown as in Step~1, by recalling the presence of the swap velocity field and relying on \eqref{remark:swap_action}.

\medskip

\noindent{\bf Step 4: Lack of selection.} {\sl We conclude by showing that the two subsequences $\vartheta_{\sigma_{2q}}$ and $\vartheta_{\sigma_{2q+1}}$ cannot converge (with respect to the weak$^*$ topology) to the same limit as $q \to \infty$.}

\medskip

%
%

We define for every $q \in \N$ the set
$$ 
O_q = \bigcap_{k = 0}^{q-1} (\XX^{\sigma_q}_{1-T_k})^{-1} ( G_{k+1}[ \sigma_q ]) \cap   \bigcap_{k = 1}^{q} (\XX^{\sigma_q}_{1+T_{k}})^{-1} (G_{k}[\sigma_q]) \,.
 $$ 
We notice that $O_q$ has large measure. Indeed, since the flows $\XX^{\sigma_q}$ are measure preserving, thanks to \eqref{stima:G_q} and the choice of $a_0$ in~\eqref{c:d_0},  we can estimate
\begin{equation}\label{e:bigOq}
\mathcal{L}^2 (O_q^c) \leq 2 \sum_{k=0}^q \mathcal{L}^2 (( G_k[\sigma_q])^c)
\leq 2 \sum_{k=0}^q 24 a_k^{\epsilon \delta}
\leq 96 a_0^{\epsilon \delta} \leq  \frac{1}{10} \,.
\end{equation}

Since the unique solution $\vartheta_{\sigma_q}$ is characterized by~\eqref{eq:solutiontoPDE} we can test against $\mathbbm{1}_{A_0}(x)$ and compute
\begin{align*}
\int_{\T^2} \vartheta_{\sigma_q} &(t,x) \mathbbm{1}_{A_0}(x) \, dx
=
\int_{\T^2} \vartheta_{\initial} (x) \mathbbm{1}_{A_0} (\XX^{\sigma_q} (t,x)) \, dx \\
&=
\int_{O_q} \vartheta_{\initial} (x) \mathbbm{1}_{A_0} (\XX^{\sigma_q} (t,x)) \, dx +
\int_{O_q^c} \vartheta_{\initial} (x) \mathbbm{1}_{A_0} (\XX^{\sigma_q} (t,x)) \, dx \\
&=
\int_{O_q} \mathbbm{1}_{A_0} (x) \mathbbm{1}_{A_0} (\XX^{\sigma_q} (t,x)) \, dx -
\int_{O_q} \mathbbm{1}_{B_0} (x) \mathbbm{1}_{A_0} (\XX^{\sigma_q} (t,x)) \, dx +
\int_{O_q^c} \vartheta_{\initial} (x) \mathbbm{1}_{A_0} (\XX^{\sigma_q} (t,x)) \, dx \\
&= I_1 - I_2 +I_3 \,. 
\end{align*}
By~\eqref{e:bigOq}, we see that $|I_3| \leq \sfrac{1}{10}$.
%
%
From Steps~1--3 we see that, for $t \geq 1+ T_0$ and for $q \in \N$ even, there holds
\begin{align*}
x  \in A_{0} \cap O_q \quad &\Longrightarrow \quad \XX^{\sigma_q}_{t} ( x )  \in A_0 \,, \\
x \in B_{0} \cap O_q \quad &\Longrightarrow \quad  \XX^{\sigma_q}_{t} ( x )  \in B_0 \,,
\end{align*}
and, for $t \geq 1+ T_0$ and for $q \in \N$ odd, there holds
\begin{align*}
x \in A_{0} \cap O_q \quad &\Longrightarrow \quad \XX^{\sigma_q}_{t} ( x )  \in B_0 \,, \\
x \in B_{0} \cap O_q \quad &\Longrightarrow \quad  \XX^{\sigma_q}_{t} ( x )  \in A_0 \,.
\end{align*}
Using again~\eqref{c:d_0},  ~\eqref{stima:G_q} and  ~\eqref{e:bigOq}   we have that $I_1 \geq  \sfrac{1}{2} - \sfrac{2}{10}$ and $I_2=0$ for~$q$ even, while $I_1=0$ and $I_2 \geq \sfrac{1}{2} - \sfrac{2}{10}$ for~$q$ odd. Therefore,  for any $q\in\N$ sufficiently large,
$$
\int_{\T^2} \vartheta_{\sigma_{2q}} (t,x) \mathbbm{1}_{A_0}(x) \, dx \geq \frac{1}{2} - \frac{3}{10} > -\frac{1}{2} + \frac{3}{10}
\geq \int_{\T^2} \vartheta_{\sigma_{2q+1}} (t,x) \mathbbm{1}_{A_0}(x) \, dx
$$
which implies the thesis.
\end{proof}

\section{Convergence to a solution which conserves the $L^2$ norm in Theorem~\ref{t_mainadvection}}\label{section:conservativesolution}

Given $\alpha\in [0,1[$ as in the statement of Theorem~\ref{t_mainadvection}, we choose $\beta=0$ and recall that all relevant parameters have been correspondingly fixed in Section~\ref{ss:para}. In particular, we recall that
$$
a_0^{\sfrac{\epsilon \delta}{8} } \leq \frac{1}{20} \,, \quad
\epsilon \leq \frac{\delta^3}{ 50} \,, \quad 
\gamma = \frac{\delta}{8} \,, \quad
\kappa_q = a_q^{2 + 3 \epsilon} \,, \quad 
a_{q+1} =a_q^{1+ \delta}\,,
$$
and  that the length of the time intervals $\mathcal{I}_{q,j}$ is $a_q^{\gamma - \gamma \delta}$ (for $q \in m\N$  and $j=0$), $0$ (for $q \not\in m\N$ and $j=0$) and $a_q^\gamma$ otherwise. 

We consider the initial datum $\vartheta_{in}$ defined in~\eqref{initialdatum}  and the velocity field $u \in C^\alpha ((0,2) \times \T^2)$ constructed in Section~\ref{section:vectorfield} for times $0 \leq t \leq 1$ and extended for times~$1 \leq t \leq 2$ by formula \eqref{vectorfield_VV}. For the regularity of the velocity field see Remark \ref{remark:vectorfield}.

For any $q \in \N$, we consider the unique bounded solution $\vartheta_{\kappa_q}$ of the advection-diffusion equation~\eqref{e:intro:advdiff} with velocity field~$u$, initial datum $\vartheta_{\initial}$, and diffusivity $\kappa = \kappa_q$. 
We also define by truncation in time the smooth velocity fields
\begin{align*}
u_q (t,x) = u(t,x) \mathbbm{1}_{[1-T_q, 1+ T_q]^c} (t),
\end{align*}
{for any $q \in \N$}
and we let $\XX^q$ be the flow of $u_q$ and $\vartheta_q$ be the unique solution of the advection equation~\eqref{e:intro:advection} with velocity field $u_q$ and initial datum $\vartheta_{\initial}$.

We split the proof in four steps. In the first three steps, we fix $q \in \N$ and we show that for the flows $\XX^q$ and the stochastic flows $\XX^{\kappa_q}$ it holds
\begin{equation} \label{stability:flows}
| \XX^{\kappa_q}_t (x, \omega) - \XX^q_t (x)| \leq a_q^{1+ \frac{\epsilon}{2}} 
\qquad \text{for $(x, \omega) \in O_q \text{ and } t\in[0,2]$,}
\end{equation} 
for a certain set $O_q \subset \T^2 \times \Omega$ with $(\mathcal{L}^2 \otimes \mathbb{P}) (O_q) \to 1$ as $q\to\infty$. In the last step we exploit~\eqref{stability:flows} and the representation formulas for the solutions to show that
$$ 
\vartheta_{\kappa_q} - \vartheta_q \to 0 
$$
in the sense of distributions {as $q \to \infty$}. In particular, $\vartheta_q$ and $\vartheta_{\kappa_q}$ converge in the sense of distributions to the same limit. This will conclude the proof since $\vartheta_q$ strongly converges to a  solution  of the advection equation~\eqref{e:intro:advection} which conserves the $L^2$ norm.



\medskip

\noindent{\bf Step~1: Closeness of $\XX^{\kappa_q}_t$ and $\XX^q_t$ for $t \in [0,1-T_q]$.} {\sl 
For any $q$ let $ \overline{q}=\overline{q}(q)$ be the largest natural number such that 
\begin{align} \label{c:j_q}
a_{q}^{\sfrac{ \epsilon}{4} }  \exp(a_{\overline{q}}^{- \gamma - \delta - 2 \epsilon \delta})  \leq \frac{1}{6}
\end{align}
and notice that $\overline{q} \to \infty$ as $q \to \infty$.
We claim that there exist $D_{{q}} \subset \T^2$ and $\Omega_{q,1} \subset \Omega$ with $\mathcal{L}^2 (D_{{q}}) \geq 1 - 12 a_{\overline{q}}^{\epsilon \delta}$ and $\mathbb{P} (\Omega_{q,1}) \geq 1 - \exp\big({- a_q^{- \epsilon/2} /4}\big) $ such that for all $t \leq 1- T_q$
$$
  | \XX_{t}^{\kappa_q} (x, \omega) - \XX_t^q (x) | \leq \frac{a_{q}^{1+ \epsilon}}{3}
  \qquad \text{ for  $x \in D_{{q}}$ and $\omega \in \Omega_{q,1}$}.
$$
 }

\medskip

First, we define
\begin{equation}\label{omega_q}
\Omega_{q,1} : = \left \{ \omega \in \Omega : \sqrt{2 \kappa_q} \sup_{t \in [0,2]} |\WW_t (\omega) | \leq a_q^{1+
\frac{5 \epsilon}{4}}   \right \} \subset \Omega\,, 
\end{equation}
and we use~\eqref{stima:brownian} to get the probability estimate
\begin{equation}\label{omega_q2}
\mathbb{P} (\Omega_{q,1}) \geq 1 - \exp( - a_q^{- \epsilon /2} /4)\,,
\end{equation}
We employ~\eqref{l:gronwall} to get a closeness estimate in the smaller time interval $[0,1- T_{\overline{q}}]$, namely
\begin{align} \label{eq:gronwall_step1}
| \XX_{t}^{\kappa_q} (x, \omega) - \XX_t^q (x)| & \leq  \left ( \sqrt{2 \kappa_q} \sup_{t \in [0,2]} | \WW_t(\omega)| \right) \exp \left(\int_0^t \| \nabla u(s,\cdot) \|_{L^\infty} ds \right) \notag
\\
& \leq a_{q}^{1+ \frac{5 \epsilon}{4} }  \exp\left(a_{\overline{q}}^{- \gamma - \delta - 2 \epsilon \delta}\right) \leq  \frac{a_q^{1+ \epsilon}}{6},
\end{align} 
for any $\omega \in \Omega_{q,1}$,
where we used~\eqref{c:j_q} and the estimate $\| \nabla u(s,\cdot) \|_{L^\infty}  \leq a_{\overline{q}}^{- \gamma - \delta - 2 \epsilon \delta}$ for $s \leq 1 - T_{\overline{q}}$ (recall Remark~\ref{remark:vectorfield}).

We now show the closeness on the time interval $[1- T_{\overline{q}}, 1- T_q]$. To this extent, we recall the definition of $G_k$ in~\eqref{d:G_sets} and define
\begin{equation} \label{d:A_q}
D_q := \bigcap_{k=\overline{q}}^{q-1} D_{q,k}  := \bigcap_{k=\overline{q}}^{q-1} \{ x \in \T^2 : \XX^{q}_{1- T_k} (x) \in G_{k+1}[a_q^{1+ \epsilon}] \} \,.
\end{equation}
Thanks to \eqref{stima:G_q},  we observe that
\begin{equation}\label{estimate:A_q}
\mathcal{L}^2(D_q^c) \leq \sum_{k=\overline{q}}^{q-1} \mathcal{L}^2 (( G_{k+1}[a_q^{1+ \epsilon}])^c) \leq \sum_{k=\overline{q}}^{q-1} (20 a_{k+1}^{\epsilon \delta}  + 4 a_q^\epsilon) \leq 48 a_{\overline{q}}^{\epsilon \delta}\,.
\end{equation}
For any $x \in \T^2$ and $\omega \in \Omega_{q,1}$ we define
$$ \tau= \tau (x, \omega) := \min \left  \{ t: | \XX_{t}^{\kappa_q} (x, \omega) - \XX_t^q (x) | = \frac{a_{q}^{1+ \epsilon}}{3} \right  \} $$
which is well defined because the trajectories of both the flow and the stochastic flow are continuous in time. From the result in the first part of the step we see that $\tau \geq 1 - T_{\overline{q}}$. We need to show that $\tau \geq 1- T_q$. If this would not be the case, we would have
\begin{align}
| \XX_{\tau}^{\kappa_q} (x, \omega) - \XX_{\tau}^q (x) | 
 & \leq \left| \XX_{1-T_{\overline{q}}}^{\kappa_q} (x, \omega) - \XX_{1-T_{\overline{q}}}^q (x) \right| +
{\left|\int_{1- T_{\overline{q}} }^{\tau} \left ( u(\XX^{\kappa_q}_s (x, \omega)) - u_q(\XX^{q}_s (x )) \right ) ds \right|} \label{inutile:integ}
 \\
& \quad +  \sqrt{2 \kappa_q} |\WW_{\tau}(\omega) - \WW_{1-T_{\overline{q}}} (\omega)| \notag
\\
& \leq  \frac{a_q^{1 + \epsilon}}{6} + 2 a_q^{1+ \frac{ 5 \epsilon}{4}}  < \frac{a_q^{1+ \epsilon}}{3}, \notag
 \end{align}
where we used the definition of $D_q$ in~\eqref{d:A_q} and the definition of $\tau$ and  property \eqref{c:vectorfield_constant_locally}  to conclude that the integrand in \eqref{inutile:integ} is $0$; we also used \eqref{omega_q} in the second-to-last inequality. Therefore, we conclude that~$\tau \geq 1- T_q$.

\medskip

\noindent{\bf Step 2: Shortness of the trajectories of $\XX^{\kappa_q}$ in the time interval $[1-T_q, 1+T_q]$.} {\sl  We show that there exists a set $\Omega_{q,2} \subset \Omega$ with $\mathbb{P}(\Omega_{q,2}) \geq 1 - a_q^\epsilon $
such that 
$$
\left | \int_{1- T_q}^{1 + T_q} u(s, \XX^{\kappa_q}_s (x, \omega))  ds \right | \leq \frac{a_q^{1+ \epsilon}}{6}
\qquad \text{{for  $x \in \T^2$ and} $\omega \in \Omega_{q,2}$.}
$$
In particular, by Step~1 and using~\eqref{omega_q}, we have for $t \leq 1+T_q$
  \begin{align*}
 | \XX^{\kappa_q}_t (x, \omega) - \XX^{q}_t (x)| & \leq | \XX^{\kappa_q}_{1-T_q} (x, \omega) - \XX^{q}_{1-T_q} (x)| + \left | \int_{1- T_q}^{1 + T_q} u(s, \XX^{\kappa_q}_s (x, \omega))  ds \right | 
   + \sqrt{2 \kappa_q} | \WW_{t} (\omega) - \WW_{1-T_q}(\omega)| \\
   & \leq \frac{2 a_q^{1+ \epsilon}}{3}
   \qquad \text{for any $x \in D_q$ and $\omega \in \Omega_{q,1} \cap \Omega_{q,2}$.}
  \end{align*}  
  }
  
  \medskip
  
To this aim, we set $ \Omega_{q,2} =  \Omega_{q,2}^{1} \cap \Omega_{q,2}^{2}$,  where 
$$
  \Omega_{q,2}^1 = 
  \left  \{ \omega : \left | \int_{1-T_{q}}^{1- T_{q+1}} u(s, \XX_s^{\kappa_q}(x, \omega)) ds \right | \leq  \frac{a_{q}^{1+ \epsilon }}{18} 
  \mbox{ for any } x \in \T^2 \right  \}
$$
and
$$
     \Omega_{q,2}^2 = 
  \left  \{ \omega : \left | \int_{1 + T_{q+1}}^{1 + T_{q}} u(s, \XX_s^{\kappa_q}(x, \omega)) ds \right | \leq  \frac{a_{q}^{1+ \epsilon }}{18} \mbox{ for any } x \in \T^2 \right  \} \,.
$$
By the estimates in Remark~\ref{remark:vectorfield} and the choice of the parameters
$$
\left | \int_{1 - T_{q+1}}^{1 + T_{q+1}} u(s, \XX_s^{\kappa_q}(x, \omega)) ds \right | \leq  \|u \|_{L^\infty ((1-T_{q+1}, 1+ T_{q+1}) \times \T^2)}  \leq 2 a_{q+1}^{1- \gamma} \leq a_q^{1+ \frac{\delta}{2}} \leq \frac{a_q^{1+ \epsilon}}{18},
\quad \forall x \in \T^2 \,, \ \forall \omega \in \Omega
$$
therefore it just remains to prove the estimate on $\mathbb{P}(\Omega_{q,2})$. 
We only estimate $\mathbb{P}(\Omega_{q,2}^1)$ since the estimate for~$\Omega_{q,2}^2$ is identical. 

We recall that on the time interval $[1- T_q, 1-T_{q+1}]$ the velocity field is nonzero only for times in $\mathcal{I}_{q,2}$ and $\mathcal{I}_{q,3}$.
For the time interval $\mathcal{I}_{q,3}$, using \eqref{s:vectorfield_I_q5} and the choice of the parameters we have that 
$$  
\left | \int_{\mathcal{I}_{q,3}} u(s, \XX_s^{\kappa_q} (x, \omega)) ds \right | \leq \| u \|_{L^\infty (\mathcal{I}_{q,3} \times \T^2)} \leq 2 a_{q+1} a_q^{- \gamma} \leq \frac{a_{q}^{1 + \epsilon}}{36} \qquad \forall x \in \T^2 \,, \ \forall \omega \in \Omega \,.
$$
 
We now estimate the integral on the time interval $\mathcal{I}_{q,2}$ using the so-called It\^o-Tanaka trick. For shortness of notation we denote
$$
\tilde{u}(t, x ) = u( 1-T_{q} + \overline{t}_q+ t_q + t, x  )
\qquad\text{ and }\qquad
\widetilde{\XX}_t = \XX_{1-T_{q} + \overline{t}_q+  t_q +t}^{\kappa_q}
\qquad \text{for all $t \in [0, t_q]$.}
$$
We apply It\^o formula \eqref{eq:ito} to the stochastic flow $\widetilde{\XX}_t$ choosing $f: [0,t_q] \times \T^2 \to \R$ of the form $f(t, x_1, x_2) = \tilde{\eta}_{q,2} (t) g(x_2)$, where $\tilde{\eta}_{q,2}(t) = \eta_{q,2} ( 1- T_q + \overline{t}_q + t_q + t) $ and   $g$ is the solution to 
\begin{align*}
\begin{cases}
g '' (y)= \tilde{w}_{q+1,2} (y), 
\\
g(0)=g(1)=0.
\end{cases}
\end{align*}
In the last equation, $\tilde{w}_{q+1,2} : \T \to \R$ is defined by $w_{q+1,2} \star \psi_{q+1} (x_1,x_2)= ( \tilde{w}_{q+1,2}(x_2)  , 0)$, $\T$ is the one-dimensional torus, and $w_{q+1,2}$, $\psi_{q+1}$ and $\eta_{q,2}$ are defined in Section \ref{section:vectorfield}. 
The function $f$ enjoys the following estimates
\begin{align} \label{eq:stima_f}
\| f \|_{L^\infty([0,t_q] \times \T^2)} \leq  4 a_{q+1} a_{q}^{2- \gamma},  \ \ \ \ \  \| f \|_{L^\infty ([0,t_q] ; C^1(\T^2))} \leq 4 a_{q +1} a_q^{1- \gamma} \ \ \ \ \  \| f \|_{L^\infty (\T^2 ; C^1((0, t_q)))} \leq  a_{q+1} a_{q}^{2- 3 \gamma}.
\end{align}
Indeed, the time regularity directly follows from~\eqref{d:convolution_time}. In order show the spatial regularity we rely on the ``almost'' $a_{q+1}$-periodicity of ${w}_{q+1,2}$ (recall the construction in Section~\ref{section:vectorfield}). To this extent, we notice that $\tilde{w}_{q+1,2} (y + a_{q+1}) = - \tilde{w}_{q+1,2} (y)$ for any $y \in \bigcup_{i \in L} (i a_{q+1} , (i+1) a_{q+1})$, where we denote $L= \bigcup_{r=0}^2  \left \{ k \in \N : k \neq n \frac{a_{q}}{a_{q+1}} -r  \mbox{ for any } n \in \N \right \} $, while $\tilde{w}_{q+1,2} (y + a_q) = - \tilde{w}_{q+1,2} (y)$  for any $y \in \T$ and~$\int_{i a_q}^{(i+1) a_q} \tilde{w}_{q+1,2} (s) ds =0 $ for any $i \in \N$. Hence, setting $g'(y) = \int_0^y \tilde{w}_{q+1, 2}(z) dz$, we deduce 
$$| g'(y) | = \left | \int_0^y \tilde{w}_{q+1, 2}(z) dz \right | \leq  \int_{- 2 a_{q+1}}^{ 2 a_{q+1}} | \tilde{w}_{q+1, 2}(z) | dz  \leq  4 a_{q+1} a_q^{1- \gamma} $$
and  $g'(y ) = - g'(y + a_q)$  for any $y \in \T$. From the last property and the parity of $a_q^{-1}$ we also deduce that $g(y) = \int_0^y g'(s) ds$ satisfies $\|  g \|_{L^\infty} \leq 4 a_{q+1} a_q^{2- \gamma}$ and $g(0) = g(1)=0$.

Applying It\^o formula \eqref{eq:ito}  with the function $f: [0,t_q] \times \T^2 \to \R$ observing that $\Delta f (t, x_1, x_2) = \partial_{x_2}^2  f (t, x_1,  x_2) = \tilde{u}^1 (t,  x_1, x_2)$, where $\tilde u (t, x_1, x_2) = ( \tilde{u}^1 (t, x_1,  x_2) , 0)$,
we get
\begin{align} 
 \left |  \int_{0}^{t_q}  \tilde{u} (s, \widetilde{\XX}_s) ds \right  |   & =  \left |  \int_{0}^{t_q}  \tilde{u}^1 (s, \widetilde{\XX}_s) ds \right  |  \notag
 \\
&   \leq   \frac{1}{\kappa_q} \bigg | f( t_q, \widetilde{\XX}_{t_q}) - f(0, \widetilde{\XX}_0) \bigg | + \bigg | \frac{1}{\kappa_q} \int_0^{t_q} \partial_t f(s, \widetilde{\XX}_s)  ds \bigg |  \label{e:tricksum}
 \\
 & \quad + \bigg | \frac{1}{\kappa_q} \int_{0}^{t_q} \nabla f (s, \widetilde{\XX}_s) \cdot \tilde{u} (s, \widetilde{\XX}_s) ds  \bigg |   
  + \bigg | \sqrt{2 \kappa_q}  \int_{0}^{t_q} \nabla f(s, \widetilde{\XX}_s) \cdot d \WW_s \bigg ) \bigg | \notag
 \\
&  \leq 
\frac{1}{\kappa_q}  \left |   f(t_q, \widetilde{\XX}_{t_q}) - f(0 , \widetilde{\XX}_0) \right | 
 + \frac{1}{\kappa_q} \left | \int_0^{t_q} \partial_t f(s,  \widetilde{\XX}_s)  ds \right |  
 + \frac{1}{\kappa_q} \left |  \sqrt{2 \kappa_q} \int_{0}^{t_q}  \nabla f(s, \widetilde{\XX}_s) \cdot d \WW_s  \right | \notag
\end{align}
since $\nabla f \cdot \tilde u =0$. We now estimate all the terms in the sum: using~\eqref{eq:stima_f},~\eqref{s:vectorfield_I_q3} and ~$t_q = a_q^\gamma$ we get
\begin{equation} \label{stima:itotanaka1}
\frac{1}{\kappa_q} \left |   f(t_q, \widetilde{\XX}_{t_{q}}) - f(0, \widetilde{\XX}_0) \right |  \leq \frac{2 \| f \|_{L^\infty}}{\kappa_q} \leq  \frac{8 a_{q+1} a_{q}^{2 - \gamma}}{\kappa_q} \leq  \frac{a_{q}^{1 + \epsilon}}{108} 
\end{equation}
and
\begin{equation}\label{stima:itotanaka2}
\frac{1}{\kappa_q} \left | \int_0^{t_q} \partial_t f(s, \widetilde{\XX}_s)  ds \right | \leq  \frac{t_q \| f \|_{L^\infty (\T^2 ; C^1((0, t_q)))}}{\kappa_q} \leq \frac{a_{q+1} a_q^{2- 2 \gamma}}{\kappa_q} \leq  \frac{a_q^{1 + \epsilon}}{108} 
\end{equation}
for any $x \in \T^2$ and $\omega \in \Omega$,
since our choice of the parameters guarantees $2 \gamma + 5 \epsilon \leq \delta$ and $108 a_q^\epsilon \leq 1/4$.
 
For the last term in the sum~\eqref{e:tricksum} we use the It\^o isometry  
in \eqref{quadratic_variation} and using also~\eqref{eq:stima_f}  we conclude
 \begin{align*}
    \mathbb{E} \left [ \left | \int_0^{t_q}  \frac{\sqrt{2}}{\sqrt{\kappa_q}}  \nabla f (s, \widetilde{\XX}_s) d \WW_s \right |^2 \right ]^{1/2}
    = 
    \mathbb{E} \left [ \int_0^{t_q} \left  |\frac{\sqrt{2}}{\sqrt{\kappa_q}}  \nabla f(s, \widetilde{\XX}_s) \right  |^2 ds   \right ]^{1/2}
  \leq \frac{8 a_{q+1} a_q^{1- \gamma} a_{q}^{\gamma/2}}{\sqrt{\kappa_q}}.
 \end{align*}
Exploiting the Cauchy-Schwarz inequality and 
%
 the Markov inequality we get
\begin{align}
\mathbb{P} \left (\omega : \left | \int_0^{t_q}  \frac{\sqrt{2}}{\sqrt{\kappa_q}}  \nabla f (s, \widetilde{\XX}_s) \cdot d \WW_s \right | \geq \frac{a_{q}^{1 + \epsilon}}{108} \right ) \leq
\frac{8 \cdot 108 a_{q+1} a_{q}^{1 - \gamma/2}}{\sqrt{\kappa_q} a_{q}^{1 + \epsilon}}  \leq a_{q}^{ \epsilon}/2, \label{eq:seguire3}
\end{align}
since  the choice of parameters guarantees $\delta \geq \sfrac{ \gamma}{2} + 4 \epsilon$ and $16 \cdot 108 a_q^{\sfrac{\epsilon}{2}} \leq 1$.

Using \eqref{e:tricksum}, \eqref{stima:itotanaka1} and \eqref{stima:itotanaka2} we have
$$(\Omega_{q,2}^1)^c  \subset  \left \{ \omega : \left | \int_0^{t_q}  \frac{\sqrt{2}}{\sqrt{\kappa_q}}  \nabla f (s, \widetilde{\XX}_s) \cdot d \WW_s \right | \geq \frac{a_{q}^{1 + \epsilon}}{108} \right \}    $$
and thanks to \eqref{eq:seguire3} we get the estimate on $\Omega_{q,2}^1$.


\medskip

\noindent{\bf Step~3: Closeness of $\XX^{\kappa_q}_t$ and $\XX^q_t$ for $t \in [1+T_q, 2]$.} {\sl  We claim that for all $0 \leq t \leq 2$
$$
|\XX_{t}^{\kappa_q} (x, \omega) - \XX_t^q (x) | \leq  a_{q}^{1+ \frac{\epsilon}{2}} 
\qquad
\text{for $x \in D_q \cap \tilde{D}_q$ and $\omega \in \Omega_{q,1} \cap \Omega_{q,2}$\,,}
$$
for some $\tilde{D}_q$ with $\mathcal{L}^2(\tilde D_q^c) \leq  48 a_{\overline{q}}^{\epsilon \delta}$, where $\bar q = \bar q(q)$ is as in~\eqref{c:j_q}. 
}
%
%
%

\medskip
Indeed, for $t \leq 1 + T_q$, the claim follows from Step~1 and Step~2. For any $ 1+ T_q \leq t \leq 1+T_{\overline{q}}$,  arguing as in Step~1 (compare in particular with \eqref{eq:gronwall_step1}), we see that 
  \begin{align*}
 | \XX_{t}^{\kappa_q} & (x, \omega) - \XX_{t}^q (x) |  \\
\leq & \; | \XX_{1+T_{{q}}}^{\kappa_q} (x, \omega) - \XX_{1+ T_{{q}}}^q (x) | 
   +
\left|  \int_{1+ T_{{q}} }^{t} \Big[ u(\XX^{\kappa_q}_s (x, \omega)) - 
u_q(\XX^{q}_s (x )) \Big] ds \right| 
+|\WW_{t}(\omega) - \WW_{1+T_{q}} (\omega)|
\\
\leq & \; \frac{2 a_q^{1 + \epsilon}}{3} + 2 a_q^{1+ \frac{ 5 \epsilon}{4}}  < a_q^{1+ \epsilon}\,,
\end{align*}   
for any $x \in  D_q \cap \tilde{D}_q $ and  $\omega \in \Omega_{q,1} \cap \Omega_{q,2}$, where 
$$
\tilde{D}_q := \bigcap_{k=\overline{q} +1}^{q} \tilde{D}_{q,k}  := \bigcap_{k=\overline{q} +1}^{q} \{ x \in \T^2 : \XX^{q}_{1+ T_k} (x) \in G_{k}[a_q^{1+ \epsilon}] \}\,,
$$
which has the same  estimate as $D_q$ obtained  in \eqref{estimate:A_q}.
%
%
Finally by applying~\eqref{l:gronwall}  we conclude that for any $1 + T_{\overline{q}} \leq t \leq 2$, $x \in D_q \cap \tilde{D}_q$, $\omega \in \Omega_{q,1} \cap \Omega_{q,2}$
\begin{align*}
| \XX_{t}^{\kappa_q} (x, \omega) - \XX_{t}^q (x) | & \leq  \left ( |\XX_{1+T_{\overline{q}}}^{\kappa_q} - \XX_{1+T_{\overline{q}}}^{q} | + 2 \sqrt{2 \kappa_q} \sup_{s \in [0,t]} | \WW_s(\omega)| \right) \exp \left( \int_{1+ T_{\overline{q}}}^t \| \nabla u(s,\cdot) \|_{L^\infty} ds \right) 
\\
& \leq (a_q^{1+ \epsilon} + 2 a_q^{1 + \frac{5 \epsilon}{4}}) \exp (a_{\overline{q}}^{- \gamma - \delta - 2 \epsilon \delta}) \leq a_q^{1 + \frac{\epsilon}{2}}.
\end{align*}  

\medskip

\noindent{\bf Step~4: Convergence to a solution that conserves the $L^2$ norm. } {\sl 
%
%
We show that
$
\vartheta_{\kappa_q} - \vartheta_q \to 0
$
{ as $q \to \infty$} 
in the sense of distributions and that
$
 \| \vartheta_q - \vartheta \|_{L^1((0,2)\times\T^2)} \to 0
$
{as $q \to \infty$},
where $\vartheta $ is a solution of the advection-diffusion equation~\eqref{e:intro:advection} which conserves the $L^2$ norm.
This implies that $\vartheta_{\kappa_q} $ converges in the sense of distributions to the solution $\vartheta$.}

\medskip

By the representation formulas for $\vartheta_{\kappa_{q}}$ (recall Theorem~\ref{prop:representation}) and for $\vartheta_q$,
we deduce
\begin{align*}
\left  | \int_{\T^2} f( x) (\vartheta_{\kappa_{q}}(t,x ) - \vartheta_q(t,x))  dx \right  | & = \left | \mathbb{E} \int_{\T^2} (f(\XX_t^{\kappa_q}(x, \omega)) - f(\XX_t^{q}(x)) )  \vartheta_{\initial}(x) dx \right  |  
 \\
 & \leq \| \nabla f \|_{L^\infty}  \mathbb{E} \int_{\T^2} |\XX_t^{\kappa_q}(x, \omega)) -  \XX_t^{q}(x))  |dx \to 0,
\end{align*} 
where the convergence to zero for $q\to\infty$ holds uniformly in time by Steps~1--3, since $\mathcal{L}^2 (D_q \cap \tilde{D}_q) \to 1$ and $\mathbb{P} (\Omega_{q,1} \cap \Omega_{q,2}) \to 1$. This shows the convergence $\vartheta_{\kappa_q} - \vartheta_q \to 0$ in the sense of distributions. 

The fact that $\{ \vartheta_q \}_q$ is a Cauchy sequence in $L^1$ follows by the definition of the velocity field by reflection, namely $u(t,x) =- u(2-t, x)$, for $1<t<2$, which in turn implies $\vartheta_{q+1} (t,x) = \vartheta_{q}(t,x)$ for any  $x \in \T^2$ and~$t \in [1-T_q, 1+T_q]^c$. Since $u_q$ is smooth $\vartheta_q$ conserves the $L^2$ norm for any $q$. Therefore, as $q \to \infty$ the sequence $\vartheta_q$ converges to a solution $\vartheta$ of the advection equation~\eqref{e:intro:advection} which coincides with $\vartheta_q$ on $ [1-T_q, 1+T_q]^c \times \T^2$ and satisfies
$$ \| \vartheta (t, \cdot) \|_{L^2 (\T^2)} = \| \vartheta_{\initial} \|_{L^2 (\T^2)} \qquad \mbox{ for any } t \in [0,2] \setminus \{ 1 \} \, .  \qquad \text{\qedsymbol} $$ 

%
%
%
%
%

\section{Convergence to a solution which dissipates the $L^2$ norm in Theorem~\ref{t_main_OC} and Theorem~\ref{t_mainadvection}} \label{section:nonconservative}

Let us fix $\alpha + 2 \beta <1$ in the case of Theorem~\ref{t_main_OC} and $\alpha <1$ and $\beta =0$ in the case of Theorem~\ref{t_mainadvection}. 
%
We recall for the convenience of the reader the choice of parameters defined  in Section~\ref{ss:para}:
$$
\gamma  = \frac{p^\circ \beta (1 + 3\epsilon (1 + \delta)) (1 + \delta ) }{1 -2 \delta} + \frac{ \delta}{8}\,, \qquad
a_0^{\sfrac{\epsilon \delta}{8} } \leq \frac{1}{20} \,, \qquad
\epsilon \leq \frac{\delta^3}{ 50} \,, \qquad m -1 \geq \frac{16}{\delta^2}.
$$
The parameters are assumed to satisfy~\eqref{c:beta_eps_condition_all} depending on the choice of~$\alpha$ and~$\beta$ and  the diffusivity parameter for the convergence to a  solution which dissipates the $L^2$ norm has been set to
$$
 \tilde{\kappa}_q = a_q^{2 - \frac{\gamma}{1 + \delta} + 4 \epsilon} \,.
$$
We set $a_{q+1} =a_q^{1+ \delta}$ and the length of the time intervals $\mathcal{I}_{q,j}$ is $a_q^{\gamma - \gamma \delta}$ (for $q \in m\N$  and $j=0$), $0$ (for~$q \not\in m\N$ and $j=0$) and $a_q^\gamma$ otherwise.

We consider the initial datum $\vartheta_{\initial}$ defined in~\eqref{initialdatum}  and the velocity field $u$ constructed in Section~\ref{section:vectorfield} for times $0 \leq t \leq 1$ and extended for times~$1 \leq t \leq 2$ by formula \eqref{vectorfield_VV}. For the regularity of the velocity field in the case $\beta =0$ see Remark \ref{remark:vectorfield}, whereas in the general setting of Theorem \ref{t_main_OC}  see Section \ref{section:regularity}. 

 We let~$\vartheta_{\tilde \kappa_q}$ be the unique bounded solution of the advection-diffusion equation~\eqref{e:intro:advdiff} with velocity field~$u$, initial datum~$\vartheta_{\initial}$, and diffusivity $\tilde \kappa_q$. 
 
We split the proof in two steps. In the first step we show that the stochastic flow $\XX^{\tilde{\kappa}_q}$ satisfies
\begin{align*}
(x, \omega) \in \mathcal{O}_q \text{ and } x \in A_0[a_0^{1+ \epsilon}]  \ \ \Longrightarrow \ \   \XX^{\tilde \kappa_q}_{1-T_q} (x, \omega) \in  A_q
\\
(x, \omega) \in \mathcal{O}_q \text{ and } x \in B_0[a_0^{1+ \epsilon}]  \ \ \Longrightarrow \ \   \XX^{\tilde \kappa_q}_{1-T_q} (x, \omega) \in  B_q
\end{align*}
for a suitable set $\mathcal{O}_q \subset \T^2 \times \Omega$  of large measure uniformly for $q \in m \N$.
 In the second step we deduce that the solution $\vartheta_{\tilde \kappa_q}$ dissipates a large portion of its energy in the time interval $[1-T_q, 1-T_q + \overline{t}_q]$ for $q\in m\N$, more precisely we will show that
$$ 
2 \tilde \kappa_q \int_{1-T_q}^{1-T_q + \overline{t}_q} \int_{\T^2} | \nabla \vartheta_{\tilde \kappa_q} (s,x) |^2 dx ds \geq \frac{1}{2}.
$$
Since any weak limit of $\vartheta_{\tilde \kappa_q}$ is a solution of the advection equation, this is enough to conclude the proof.

\medskip

\noindent{\bf Step~1: $\XX^{\tilde \kappa_q}$ approximately preserves the chessboards for $t \in [0,1-T_q]$.  } {\sl 
 We claim that there exist $\Omega_q \subset \Omega$ and ${H}_q \subset \T^2 \times \Omega$ with $\mathbb{P}( \Omega_{q}) \geq  1 -  a_0^{\epsilon } $ for any $q \in m \N$ and $(\mathcal{L}^2 \otimes \mathbb{P}) (H_q) \geq 1 - 24 a_0^{\epsilon \delta} $ for any~$q \in \N$ and such that 
 \begin{subequations}
  \begin{align} 
  (x, \omega) \in H_q \cap  (A_0[a_0^{1+ \epsilon}]  \times \Omega_q)  \ \ & \Longrightarrow \ \  \XX^{\tilde \kappa_q}_{1-T_q} (x, \omega) \in A_q \label{implication:step1_diff}
  \\
   (x, \omega) \in H_q \cap  (B_0[a_0^{1+ \epsilon}]  \times \Omega_q)   \ \ & \Longrightarrow \ \  \XX^{\tilde \kappa_q}_{1-T_q} (x, \omega) \in B_q. \label{implication:step1_diff2}
  \end{align}
  \end{subequations}
}
Let us define 
\begin{align} \label{omega_W}
 \Omega_{q} = \bigcap_{k=-1}^{q-1} \Omega_{q, k} = \bigcap_{k=-1}^{q-1} \left \{ \omega \in \Omega : \sup_{t \in [1-T_{k}, 1-T_{k+1}]}  \sqrt{2 \tilde \kappa_q} | \WW_t - \WW_{1-T_{k}} | \leq a_{k+1}^{1+ \epsilon} \right \},
\end{align}
where $T_{-1} =0$. 
In order to prove that  $\mathbb{P}( \Omega_{q}) \geq  1 -  a_0^{\epsilon }$,
it is enough to show that  $\mathbb{P}(\Omega_{q, k}) \geq 1 - 2 e^{- a_{q}^{- \epsilon}}$ for any $k \leq q-1$.  To this aim we apply \eqref{stima:brownian} with $\tilde{T}= 1-T_k$, $T= 1- T_{k+1}$ and $c= a_{k+1}^{1+ \epsilon}$. For $k \leq q -m $ we use \eqref{d:parameters} and $\gamma \geq \sfrac{\delta}{8} \geq \sfrac{2}{\delta(m-1)}$ to estimate
$$ \dfrac{a_{k+1}^{2+ 2 \epsilon }}{  2 \tilde{\kappa}_{q} (T_{k} - T_{k+1}) } \geq a_k^{2 - \gamma+ 2 \delta  + \gamma \delta} a_q^{-2 + \gamma - \frac{\gamma \delta}{1+ \delta} - \epsilon} \geq a_q^{- \epsilon} a_k^{2 \delta + \gamma \delta} a_q^{ - \frac{\gamma \delta }{1+ \delta}} \geq a_q^{- \epsilon} a_k^{2 \delta + \gamma \delta - \gamma \delta (1 + (m-1) \delta)} \geq  a_q^{- \epsilon}.$$
For $ q-m < k \leq q-1$ we simply use that $T_k - T_{k+1} \leq 3 a_q^\gamma$ and~\eqref{d:AD_first}. 
 
We also  define
\begin{align} \label{d:H_q}
H_q := \bigcap_{k=0}^{q-1} H_{q,k}  := \bigcap_{k=0}^{q-1} \left \{ (x, \omega) \in \T^2 \times \Omega: \XX^{\tilde \kappa_q}_{1- T_k} (x, \omega) \in G_{k+1}[a_{k+1}^{1+ \epsilon}] \right \} 
\end{align}
and $H_{q, -1} = \emptyset$.
We only prove \eqref{implication:step1_diff},  as the proof of \eqref{implication:step1_diff2} is similar.  
Fix $(x, \omega) \in H_q  \cap (A_0[a_0^{1+ \epsilon}]  \times \Omega_q) $. We prove \eqref{implication:step1_diff} by showing by iteration on $-1 \leq k \leq q-1$ that 
\begin{align} \label{iteration:step1}
\XX^{\tilde \kappa_q}_{1-T_{k}} (x, \omega) \in A_{k} \ \ \Longrightarrow \ \  \XX^{\tilde \kappa_q}_{1-T_{k+1}} (x, \omega) \in A_{k+1},
\end{align}
where $A_{-1} = A_0$.
For $k=-1$ the property \eqref{iteration:step1} follows from the fact that $u \equiv 0$ in $[0,1-T_0]$, the definition of $\Omega_{q, -1} $ and the fact that $x \in A_0 [a_0^{1+ \epsilon}]$. We now show  \eqref{iteration:step1} for $k \geq 0$. Pick
 $\overline{x} \in \T^2$ such that~$\XX_{1-T_k}^q (\overline{x}) = \XX^{\tilde \kappa_q}_{1-T_k} (x, \omega)$ where $\XX_t^q$ is the flow \eqref{d:ODE}  of the smooth velocity field $u_q = u \mathbbm{1}_{[1-T_q, 1+T_q]^c}$. From the definition of $H_{q,k}$, the iterative assumption $\XX^q_{1-T_k}(\overline{x}) =\XX^{\tilde{\kappa}_q}_{1-T_k} (x, \omega) \in G_{k+1}[a_{k+1}^{1+ \epsilon}] \cap A_k$ and Remark~\ref{remark:locallyconstant:diffusion} we have that $\XX^q_{1-T_{k+1}}(\overline{x}) \in A_{k+1} $. Finally we observe that 
 \begin{align} \label{inutile:integrands}
 | \XX^q_{t} (\overline{x}) - \XX^{\tilde \kappa_q}_{t} (x, \omega)| 
 & \leq {\left | \int_{1-T_k}^t \left [  u(s, \XX^q_s(\overline{x}))  - u(s, \XX^{\tilde \kappa_q}_s(x, \omega))  \right ] ds  \right | }
  + \sqrt{2 \tilde \kappa_q} | \WW_{t}(\omega) - \WW_{1-T_k}(\omega)|  \leq a_{k+1}^{1+ \epsilon}
 \end{align}
 for any $1- T_k \leq t \leq 1-T_{k+1} $,
 where we used the definition of $\Omega_{q,k}$ and property \eqref{c:vectorfield_constant_locally} to conclude that the integrand in \eqref{inutile:integrands} vanishes.
 
The smallness estimate for $H_q^c$ is more delicate than other estimates in Section~\ref{section:conservativesolution} because the sets $H_{q,k}$ have no product structure in $\T^2 \times \Omega$. It relies on the measure preserving property \eqref{d:sLf:property}  of 
the stochastic flow applied to the set $G_{k+1} [a_{k+1}^{1+ \epsilon}]$ for any $0 \leq k \leq q-1$, the equality $\mathbbm{1}_{G_{k+1}[a_{k+1}^{1 + \epsilon}]} (\XX^{\tilde \kappa_q}_{1- T_k} ) = \mathbbm{1}_{H_{q,k}}$ in $\T^2 \times \Omega$ and the estimate \eqref{stima:G_q}, we get
$$  
(\mathcal{L}^2 \otimes \mathbb{P}) (H_q^c) =
(\mathcal{L}^2 \otimes \mathbb{P}) \left(\bigcup_{k=0}^{q-1} H_{q,k}^c \right) \leq \sum_{k=0}^{q-1} (\mathcal{L}^2 \otimes \mathbb{P}) (H_{q,k}^c) 
 \leq     \sum_{k=1 }^{q} (20 a_k^{\epsilon \delta} +4 a_k^{\epsilon} ) \leq  48 a_0^{\epsilon \delta}. 
 $$
Introducing the notation
$$
\mathcal{O}_q =  H_q \cap (\T^2 \times   \Omega_{q} )\,,
\qquad
\mathcal{O}_{q, A} = \mathcal{O}_q \cap (  A_0 [a_0^{1+ \epsilon}] \times \Omega )\,,
\qquad
\mathcal{O}_{q, B} = \mathcal{O}_q \cap ( B_0 [a_0^{1+ \epsilon}] \times \Omega )
$$
and relying on~\eqref{stima:G_q} and on the estimate on the measure of $\Omega_q$ we observe that 
$$
\mathcal{L}^2 \otimes \mathbb{P} (\mathcal{O}_{q, A} ) = 1 - \mathcal{L}^2 \otimes \mathbb{P} (\mathcal{O}_{q, A}^c )   \geq 1 - 48 a_0^{\epsilon \delta} - a_0^{\epsilon \delta} - \frac{1}{2} - 10 a_0^{\epsilon \delta} = \frac{1}{2} - 59 a_0^{\epsilon \delta}
$$ 
and observing also that $\mathcal{O}_{q, B} \cup \mathcal{O}_{q, A} = \mathcal{O}_q \cap (G_0 [a_0^{1+ \epsilon}] \times \Omega)$ we conclude that
$$
\mathcal{L}^2 \otimes \mathbb{P} ((\mathcal{O}_{q, B} \cup \mathcal{O}_{q, A})^c ) \leq  \mathcal{L}^2 \otimes \mathbb{P} (\mathcal{O}_q^c) + \mathcal{L}^2(G_0 [a_0^{1+ \epsilon}]^c)  \leq 49 a_0^{\epsilon \delta} + 20 a_0^{\epsilon \delta}\,.
$$

\medskip

\noindent{\bf Step.~2: Anomalous dissipation.} {\sl 
We claim that  for any $q \in m \N$
 $$
 \| \vartheta_{\tilde \kappa_q} (t, \cdot )  \|_{L^2(\T^2)}^2 \leq  a_0^{\epsilon \delta /2}   \qquad \mbox{ for any } t  \geq 1 - T_q + \overline{t}_q\, .
 $$
 In particular, since $ \| \vartheta_{\initial} \|_{L^2}^2 \geq \sfrac{3}{4} $ and $a_0^{\sfrac{\epsilon \delta}{2}} \leq \sfrac{1}{4}$
 from the energy balance~\eqref{e:intro:balance}   we have that 
 $$2 \tilde{\kappa}_q  \int_0^1 \int_{\T^2} | \nabla \vartheta_{\tilde \kappa_q} (s,x) |^2 dx ds \geq 2 \tilde{\kappa}_q   \int_{1-T_q}^{1- T_q + \overline{t}_q} \int_{\T^2} | \nabla \vartheta_{\tilde \kappa_q} (s,x) |^2 dx ds \geq  \frac{1}{2} \qquad \mbox{for any } q \in m\N \,  .$$
 }

\medskip

Let $f_q(x) = \vartheta_0(\lambda_q x)$ be the even chessboard function of side $a_q$ as in Definition~\ref{d_initial datum} 
and $f_q^{\tilde \kappa_q}$ be the solution of the heat equation with diffusivity parameter $  \tilde \kappa_q$, namely $\partial_t f_q^{\tilde \kappa_q} - \tilde \kappa_q \Delta f_q^{\tilde \kappa_q} =0$, and starting at time $ t= 1-T_q$ with initial datum $f_q$. 
The scaling properties of the heat equation and the $\lambda_q$-periodicity of the  initial datum $f_q$ (with average zero) imply the enhanced diffusion effect (see the end of Section \ref{ss:eur:dissipative})
$$
\left  \| f_q^{\tilde \kappa_q} (t , \cdot )  \right  \|_{L^2}^2 \leq e^{- { \lambda_q^2 \tilde \kappa_q (t - (1- T_q))}} \left \|  f_q   \right \|_{L^2}^2 \leq e^{- { \lambda_q^2 \tilde \kappa_q (t - (1- T_q))}}   \qquad \mbox{for any }  t \in [1- T_q, 1- T_q + \overline{t}_q]
$$
and also 
\begin{equation}
\label{eqn:lago} \left  \| f_q^{\tilde \kappa_q} (t , \cdot )  \right  \|_{L^2}^2 \leq a_q^{\epsilon} \leq \frac{a_0^{\epsilon \delta /2}}{4} \qquad \mbox{for } \   t - (1- T_q) = \overline{t}_q= a_q^{\gamma - \gamma \delta}
\end{equation}
where the second-to-last inequality holds thanks to \eqref{d:AD_second}. Since $u(t , \cdot ) \equiv 0$ for  $t \in [1- T_q, 1- T_q + \overline{t}_q]$ the function $\vartheta_{\tilde \kappa_q}$ solves as well the heat equation in such time interval. In the rest of this step we deal with the error due to the approximation of  the actual initial datum ${\vartheta}_{\tilde \kappa_q} (0, \cdot )$ with the chessboard function~$f_q$. 

Using the maximum principle, which in particular implies that $\vartheta_{\tilde \kappa_q} \leq 1 = f_q$ on $A_q$, $\int_{A_q} f_q = - \int_{B_q} f_q$ and the equality $$0 = \int_{\T^2} \vartheta_{\tilde \kappa_q} (1- T_q, x ) dx =  \int_{A_q} \vartheta_{\tilde \kappa_q} (1- T_q, x )  dx+ \int_{B_q} \vartheta_{\tilde \kappa_q} (1- T_q, x ) dx + \int_{(A_q \cup B_q)^c}\vartheta_{\tilde \kappa_q} (1- T_q, x ) dx  $$ we deduce that 
$$\| {\vartheta}_{\tilde \kappa_q} (1- T_q, \cdot ) - f_q \|_{L^1} \leq 2 \int_{A_q}  \left [ f_q (x) - \vartheta_{\tilde \kappa_q} (1- T_q, x ) \right ] dx + 3 \mathcal{L}^2 ((A_q \cup B_q)^c). $$
Using Step~1 and \eqref{feynman_kac} we estimate the right-hand side as
\begin{align*}
\int_{A_q}  {\vartheta}_{\tilde \kappa_q} ( 1- T_q,x) dx & =  \int_{\T^2} \mathbbm{1}_{A_{q}}(x) {\vartheta}_{\tilde \kappa_q}( 1- T_q ,x) dx  =  \mathbb{E} \int_{\T^2} \mathbbm{1}_{A_{q}}(\XX_{1- T_q}^{\tilde \kappa_q} (x, \cdot ))  d \vartheta_{\initial} (x) 
\\
&  \geq \iint_{\mathcal{O}_{q,A}}  \mathbbm{1}_{A_{q}}(\XX_{1- T_q}^{\tilde \kappa_q} (x, \omega )) dx d \mathbb{P}(\omega)  - (\mathcal{L}^2 \otimes \mathbb{P})((\mathcal{O}_{q,A} \cup \mathcal{O}_{q,B})^c)
   \\
   & = (\mathcal{L}^2 \otimes \mathbb{P}) (\mathcal{O}_{q,A})  - (\mathcal{L}^2 \otimes \mathbb{P})((\mathcal{O}_{q,A} \cup \mathcal{O}_{q,B})^c) \geq \frac{1}{2} - 128 a_0^{\epsilon \delta}.
\end{align*}
Using also that $\mathcal{L}^2 ((A_q \cup B_q)^c) \leq 20 a_q^{\epsilon \delta} \leq 20 a_0^{\epsilon \delta}$ and that $\int_{A_q} f_q = \mathcal{L}^2 (A_q) \leq \sfrac{1}{2}$, we deduce the estimate 
 $$ \frac 1 2  \| \vartheta_{\tilde \kappa_q} (1- T_q, \cdot ) - f_q \|_{L^2}^2 \leq \| {\vartheta}_{\tilde \kappa_q} (1- T_q, \cdot ) - f_q \|_{L^1} \leq 316 a_0^{\epsilon \delta} \leq a_0^{\sfrac{\epsilon \delta}{2}}/ 8.$$ This implies that
$$ \|  \vartheta_{\tilde \kappa_q} (t, \cdot ) - f_q^{\tilde \kappa_q} (t, \cdot)  \|_{L^2}^2 \leq \| \vartheta_{\tilde \kappa_q} (1- T_q, \cdot ) - f_q \|_{L^2}^2 \leq  \frac{a_0^{\sfrac{\epsilon \delta}{2}}}{4}, \qquad \mbox{ for any } t \in [1- T_q, 1- T_q + \overline{t}_q].$$
Hence, using also \eqref{eqn:lago}, we have
%
%
%
$$
\| \vartheta_{\tilde \kappa_q} (1- T_q + \overline{t}_q, \cdot ) \|_{L^2(\T^2)}^2 
\leq 2 \| \vartheta_{\tilde \kappa_q} (1- T_q + \overline{t}_q, \cdot ) - f_q^{\tilde \kappa_q} (1- T_q + \overline{t}_q, \cdot ) \|_{L^2}^2 +  2 \| f_q^{\tilde \kappa_q} (1- T_q + \overline{t}_q, \cdot )   \|_{L^2}^2
\leq {a_0^{\sfrac{\epsilon \delta}{2}}}.
$$
By the energy balance~\eqref{e:intro:balance} the same estimate $ \| \vartheta_{\tilde \kappa_q} (t, \cdot ) \|_{L^2(\T^2)}^2 \leq {a_0^{\sfrac{\epsilon \delta}{2}}} $ holds for any $t \geq 1- T_q + \overline{t}_q$. \qed

\section{Proof of the regularity in~Theorem~\ref{t_main_OC} }\label{section:regularity}

In this final section we prove the regularity of the velocity field $u$ and the uniform-in-diffusivity regularity estimate \eqref{bound_main_OC}  for the solutions $\vartheta_{\kappa}$ of the advection-diffusion equation~\eqref{e:intro:advdiff}, therefore concluding the proof of Theorem~\ref{t_main_OC}.

We begin with an elementary, but crucial, observation. Consider an horizontal or vertical shear flow~$\overline{u} \in C^{\infty}([0,1] \times \T^2)$ and consider the unique stochastic flow solving~\eqref{d:sdee}. Then we claim that
\begin{equation}\label{remark_regularity_SDE}
\sup_{\omega \in \Omega} \| \XX^{\kappa}_{0,t }( \cdot, \omega) \|_{W^{1, \infty}} \leq 3 +  \int_0^t  \| \nabla \overline u (s,\cdot) \|_{ L^\infty} ds \quad \mbox{ for any } t \in [0,1]. 
\end{equation}
Indeed, let us consider without loss of generality the case $\overline u(t, x,y) = (\overline u(t, y), 0)$, with a slight abuse of notation. The backward stochastic flow $\XX^{\kappa}_{0,t} = (\XX^{\kappa,1}_{0,t}, \XX^{\kappa, 2}_{0,t})$ can be explicitly computed as
$$
\begin{cases}
\XX^{\kappa,1}_{0,t} (x,y, \omega) =  x + \displaystyle\int_t^0 \overline{u} (s, \XX^{\kappa,2}_{0,s} (x, y, \omega) ) ds - \sqrt{2 \kappa} \WW_t^{\ 1} (\omega) \,
\\ 
\XX^{\kappa,2}_{0,t} (x,y, \omega) =  y - \sqrt{2 \kappa} \WW_t^{\ 2} (\omega) \,.
\end{cases}
$$
Hence, the derivative $\partial_y \XX^{\kappa}_{0,t} (x,y, \omega)$ can be estimated as follows
\begin{align*}
\left | \XX^{\kappa,1 }_{0,t} (x, y + h, \omega) -  \XX^{\kappa,1}_{0,t} (x, y , \omega) \right | & \leq \left | \int_0^t \overline{u} (s, y +h - \sqrt{2 \kappa} \WW_s^{\ 2}(\omega) ) - \overline{u} (s, y  - \sqrt{2 \kappa} \WW_s^{\ 2}(\omega) ) ds    \right  |
\\
&  \leq |h|  \int_0^t  \| \nabla \overline u (s,\cdot) \|_{ L^\infty} ds, 
\end{align*}
while $\XX^{\kappa, 2}_{0,t }(x, y + h, \omega) - \XX^{\kappa ,2 }_{0,t} (x,y, \omega) \equiv h$. Noticing that
$\XX^{\kappa}_{0,t }(x+ h, y , \omega) - \XX^{\kappa}_{0,t} (x,y, \omega) \equiv (h,0)$ we conclude that~\eqref{remark_regularity_SDE} holds.

\begin{proof}[Proof of the regularity in~Theorem~\ref{t_main_OC}]
%
%
We first prove that we $u \in L^p ((0,1);  C^\alpha(\T^2))$. Indeed, by \eqref{s:vectorfield_I_q5_new}
\begin{align*}
\| u \|_{L^p C^\alpha}^p & = \sum_{q=1}^\infty \int_{{\mathcal{I}}_q} \|u (s,\cdot)\|_{C^\alpha(\T^2)}^{p}  ds \leq  \sum_{q=1}^\infty \int_{{\mathcal{I}}_q} \|u (s,\cdot)\|_{L^\infty(\T^2)}^{p (1- \alpha)}   \|u (s,\cdot)\|_{W^{1,\infty}(\T^2)}^{p\alpha}   ds 
\\
& \lesssim \sum_{q=0}^\infty \lambda_q^{- \gamma}  \lambda_q^{p(1- \alpha)(\gamma -1)  }  \lambda_q^{p \alpha (\gamma -1)  } \lambda_{q+1}^{p \alpha (1 + \epsilon \delta)} 
\end{align*}
and the sum is finite if and only if 
$$ \frac{\gamma}{p} + 1 - \gamma - \alpha (1 + \epsilon \delta) (1 + \delta) >0$$
which indeed holds thanks to the choice \eqref{d:gamma}, the condition \eqref{c:alpha_beta_eps_kappa} and the fact that ${p}^\circ \geq 2$.

We now turn to the proof of the uniform-in-diffusivity bound in $L^{p^\circ}((0,1); C^\beta(\T^2))$ for the solutions~$\vartheta_{\kappa}$ of \eqref{e:intro:advdiff}. It is sufficient to prove that the (backward) stochastic flow is such that $\| \XX^{\kappa}_{0, \cdot } ( \cdot, \omega ) \|_{L^{p^\circ} C^\beta}$ is uniformly bounded independently on $\omega$ and $\kappa$. If this is the case, from the Feynman-Kac formula \eqref{feynman_kac_real}  we have
\begin{align}
| \vartheta_{\kappa} (t,x ) - \vartheta_{\kappa} (t,y)| & \leq \int_{\Omega} | \vartheta_{\initial} (\XX^{\kappa}_{0,t} (x, \omega)) - \vartheta_{\initial} (\XX^{\kappa}_{0,t} (y, \omega)) | d\mathbb{P}(\omega)  \notag
\\
& \leq \| \nabla \vartheta_{\initial} \|_{L^\infty} \int_{\Omega} |  \XX^{\kappa}_{0,t} (x, \omega) - \XX^{\kappa}_{0,t} (y, \omega) | d\mathbb{P}(\omega)  \notag
\\
 & \leq \| \nabla \vartheta_{\initial} \|_{L^\infty} \sup_{\omega \in \Omega} \| \XX^\kappa_{0,t} (\cdot, \omega ) \|_{C^\beta} |x- y|^\beta \label{proof:1}
\end{align} 
from which we conclude that $\| \vartheta_\kappa (t, \cdot) \|_{C^\beta(\T^2)} \lesssim \sup_{\omega \in \Omega} \| \XX^\kappa_{0,t} (\cdot, \omega ) \|_{C^\beta(\T^2)}$ for every $t \in [0,1]$. 

Let us show the regularity of the backward stochastic flow. Using the semigroup property of the flow, for any $t \in (1-T_q, 1- T_{q+1})$, we have for every $x\in \T^2 , \omega \in \Omega$
\begin{align*}
\XX^{\kappa}_{0,t} = \XX^{\kappa}_{0, 1-T_0} \circ \XX^{\kappa}_{1- T_0, 1-T_1} \circ \cdots \circ \XX^{\kappa}_{1-T_{q-1}, 1-T_q} \circ \XX^{\kappa}_{1-T_q, t}\,,
\end{align*}
therefore for every $\omega \in \Omega$
\begin{align} \label{stima:sde_itero}
\| \XX^{\kappa}_{0,t} (\cdot, \omega) \|_{W^{1, \infty}(\T^2)} \leq \prod_{j =0}^{q-1} \| \XX^{\kappa}_{1-T_{j},1-T_{j+1}} (\cdot, \omega) \|_{W^{1, \infty}} \sup_{t\in [{1-T_{q},1-T_{q+1}}]} \| \XX^{\kappa}_{1-T_{q},t} (\cdot, \omega) \|_{W^{1, \infty}}.
\end{align} 
We now estimate each term in the product. Using again the semigroup property we have
\begin{align*}
\| \XX^{\kappa}_{1-T_{j},1-T_{j+1}} (\cdot, \omega) \|_{W^{1, \infty}} \leq \| \XX^{\kappa}_{1-T_{j} + \overline{t}_j +  t_j,1-T_{j} + \overline{t}_j +2   t_j } (\cdot, \omega) \|_{W^{1, \infty}} 
\| \XX^{\kappa}_{1-T_{j} +  \overline{t}_j +2   t_j ,1-T_{j} +  \overline{t}_j +3 t_j } (\cdot, \omega) \|_{W^{1, \infty}}
\end{align*}
and using~\eqref{remark_regularity_SDE} and the estimates 
 in Remark \ref{remark:vectorfield} we have
$$
\| \XX^{\kappa}_{1-T_{j} + \overline{t}_j +  t_j,1-T_{j} + \overline{t}_j +2   t_j } (\cdot, \omega) \|_{W^{1, \infty}}  \leq 3 + \lambda_j^{-1} \lambda_{j+1}^{1 +\epsilon \delta}
$$
and
$$
\| \XX^{\kappa}_{1-T_{j} +  \overline{t}_j +2   t_j ,1-T_{j} +  \overline{t}_j +3 t_j } (\cdot, \omega) \|_{W^{1, \infty}} \leq 3 + \lambda_{j+1}^{\epsilon \delta}\,,
$$
which imply
$$
\| \XX^{\kappa}_{1-T_{j},1-T_{j+1}} (\cdot, \omega) \|_{W^{1, \infty}} \leq 16 \lambda_j^{\delta + 2 \epsilon \delta (1+ \delta)} \,.
$$
Analogously, 
\begin{align} \label{stima:sde_itero}
\sup_{t\in [{1-T_{q},1-T_{q+1}}]} \| \XX^{\kappa}_{1-T_{q},t} (\cdot, \omega) \|_{W^{1, \infty}} \leq 16 \lambda_q^{\delta + 2 \epsilon \delta (1+ \delta)} \,.
\end{align} 
Plugging this estimate in  \eqref{stima:sde_itero} we get
\begin{align*}
\| \XX^{\kappa}_{0,t} (\cdot, \omega) \|_{W^{1, \infty}(\T^2)} & \leq 16^{q+1} \prod_{j=0}^q \lambda_j^{ \delta (1+ 2 \epsilon   (1+ \delta))} = 16^{q+1}  \lambda_0^{\delta (1+ 2 \epsilon   (1+ \delta)) \sum_{j=0}^q  (1 + \delta)^{ j}}   
\\
& 
= 16^{q+1} \lambda_0^{(1+ 2 \epsilon   (1+ \delta)) [(1+ \delta)^{q+1}-1]} \leq 16^{q+1} \lambda_{q+1}^{1 + 2 \epsilon (1+ \delta)} \leq \lambda_{q+1}^{1 + 3 \epsilon (1+ \delta)},
\end{align*}
for any $0 \leq  t \leq 1 - T_{q+1}$,
where in the last inequality we used \eqref{c:d_0} to estimate $16^{q+1} \leq \lambda_{q+1}^{ \epsilon (1+ \delta)}$.
Finally,  we have by interpolation the bound uniformly in $\omega$ and $\kappa$ 
$$ 
\| \XX^{\kappa}_{0,t} (\cdot, \omega) \|_{C^\beta} \leq   \| \XX^{\kappa}_{0,t} (\cdot, \omega) \|_{W^{1, \infty}}^{\beta}  \| \XX^{\kappa}_{0,t} (\cdot, \omega) \|_{L^\infty}^{1- \beta} \leq \lambda_{q+1}^{\beta + 3\beta  \epsilon (1+ \delta)}
$$
 for any $ t \in (1-T_{q}, 1-T_{q+1}] $.
Therefore, using the previous observation \eqref{proof:1} we conclude that
\begin{align*}
\| \vartheta_\kappa \|_{L^{p^\circ } C^\beta}^{p^\circ} &= \sum_{q=0}^\infty \int_{{\mathcal{I}}_q} \| \vartheta_\kappa (s, \cdot ) \|_{ C^\beta}^{p^\circ}  ds \leq 4 \| \nabla \vartheta_{\initial} \|_{L^\infty}^{p^\circ}  \sum_{q=0}^{\infty} \lambda_q^{- \gamma +  \gamma \delta} \lambda_{q+1}^{p^{\circ}( \beta + 3\beta  \epsilon (1+ \delta))} 
\\
& =4  \| \nabla \vartheta_{\initial} \|_{L^\infty}^{p^\circ}  \sum_{q=0}^{\infty} \lambda_q^{- \gamma +  \gamma \delta + p^{\circ}( \beta + 3\beta  \epsilon (1+ \delta))(1 + \delta)} 
\end{align*}
and the sum is finite and independent of $\kappa$ since $- \gamma (1 -  \delta) + p^{\circ}( \beta + 3\beta  \epsilon (1+ \delta))(1 + \delta) <0$  thanks to~\eqref{d:gamma}.
\end{proof}

\bibliographystyle{alpha}
 \bibliography{biblio}

\begin{bibdiv}
\begin{biblist}

\bib{hamilt2}{article}{
      author={Alberti, Giovanni},
      author={Bianchini, Stefano},
      author={Crippa, Gianluca},
       title={Structure of level sets and {S}ard-type properties of {L}ipschitz
  maps},
        date={2013},
        ISSN={0391-173X},
     journal={Ann. Sc. Norm. Super. Pisa Cl. Sci. (5)},
      volume={12},
      number={4},
       pages={863\ndash 902},
      review={\MR{3184572}},
}

\bib{hamilt1}{article}{
      author={Alberti, Giovanni},
      author={Bianchini, Stefano},
      author={Crippa, Gianluca},
       title={A uniqueness result for the continuity equation in two
  dimensions},
        date={2014},
        ISSN={1435-9855},
     journal={J. Eur. Math. Soc. (JEMS)},
      volume={16},
      number={2},
       pages={201\ndash 234},
         url={https://doi.org/10.4171/JEMS/431},
      review={\MR{3161282}},
}

\bib{LerayHopf}{article}{
      author={Albritton, Dallas},
      author={Bru\'{e}, Elia},
      author={Colombo, Maria},
       title={Non-uniqueness of {L}eray solutions of the forced
  {N}avier-{S}tokes equations},
        date={2022},
        ISSN={0003-486X},
     journal={Ann. of Math. (2)},
      volume={196},
      number={1},
       pages={415\ndash 455},
         url={https://doi.org/10.4007/annals.2022.196.1.3},
      review={\MR{4429263}},
}

\bib{ACM-JAMS}{article}{
      author={Alberti, Giovanni},
      author={Crippa, Gianluca},
      author={Mazzucato, Anna~L.},
       title={Exponential self-similar mixing by incompressible flows},
        date={2019},
        ISSN={0894-0347},
     journal={J. Amer. Math. Soc.},
      volume={32},
      number={2},
       pages={445\ndash 490},
         url={https://doi.org/10.1090/jams/913},
      review={\MR{3904158}},
}

\bib{ACM-loss}{article}{
      author={Alberti, Giovanni},
      author={Crippa, Gianluca},
      author={Mazzucato, Anna~L.},
       title={Loss of regularity for the continuity equation with
  non-{L}ipschitz velocity field},
        date={2019},
        ISSN={2524-5317},
     journal={Ann. PDE},
      volume={5},
      number={1},
       pages={Paper No. 9, 19},
         url={https://doi.org/10.1007/s40818-019-0066-3},
      review={\MR{3933614}},
}

\bib{aize}{article}{
      author={Aizenman, Michael},
       title={On vector fields as generators of flows: a counterexample to
  {N}elson's conjecture},
        date={1978},
        ISSN={0003-486X},
     journal={Ann. of Math. (2)},
      volume={107},
      number={2},
       pages={287\ndash 296},
         url={https://doi.org/10.2307/1971145},
      review={\MR{482853}},
}

\bib{AV23}{misc}{
      author={Armstrong, Scott},
      author={Vicol, Vlad},
       title={Anomalous diffusion by fractal homogenization},
        date={2023},
}

\bib{bedrendiss}{article}{
      author={Bedrossian, Jacob},
      author={Blumenthal, Alex},
      author={Punshon-Smith, Sam},
       title={Almost-sure enhanced dissipation and uniform-in-diffusivity
  exponential mixing for advection-diffusion by stochastic {N}avier-{S}tokes},
        date={2021},
        ISSN={0178-8051},
     journal={Probab. Theory Related Fields},
      volume={179},
      number={3-4},
       pages={777\ndash 834},
         url={https://doi.org/10.1007/s00440-020-01010-8},
      review={\MR{4242626}},
}

\bib{PBGC}{article}{
      author={Bonicatto, Paolo},
      author={Ciampa, Gennaro},
      author={Crippa, Gianluca},
       title={On the advection-diffusion equation with rough coefficients: weak
  solutions and vanishing viscosity},
        date={2022},
        ISSN={0021-7824,1776-3371},
     journal={J. Math. Pures Appl. (9)},
      volume={167},
       pages={204\ndash 224},
         url={https://doi.org/10.1016/j.matpur.2022.09.005},
      review={\MR{4496901}},
}

\bib{BCZ}{article}{
      author={Bedrossian, Jacob},
      author={Coti~Zelati, Michele},
       title={Enhanced dissipation, hypoellipticity, and anomalous small noise
  inviscid limits in shear flows},
        date={2017},
        ISSN={0003-9527},
     journal={Arch. Ration. Mech. Anal.},
      volume={224},
      number={3},
       pages={1161\ndash 1204},
         url={https://doi.org/10.1007/s00205-017-1099-y},
      review={\MR{3621820}},
}

\bib{camilloelia}{article}{
      author={Bru\`e, Elia},
      author={De~Lellis, Camillo},
       title={Anomalous dissipation for the forced 3{D} {N}avier-{S}tokes
  equations},
        date={2023},
        ISSN={0010-3616,1432-0916},
     journal={Comm. Math. Phys.},
      volume={400},
      number={3},
       pages={1507\ndash 1533},
         url={https://doi.org/10.1007/s00220-022-04626-0},
      review={\MR{4595604}},
}

\bib{BDLIS}{article}{
      author={Buckmaster, Tristan},
      author={De~Lellis, Camillo},
      author={Isett, Philip},
      author={Sz\'{e}kelyhidi, L\'{a}szl\'{o}, Jr.},
       title={Anomalous dissipation for {$1/5$}-{H}\"{o}lder {E}uler flows},
        date={2015},
        ISSN={0003-486X},
     journal={Ann. of Math. (2)},
      volume={182},
      number={1},
       pages={127\ndash 172},
         url={https://doi.org/10.4007/annals.2015.182.1.3},
      review={\MR{3374958}},
}

\bib{OnsagerAdmissible}{article}{
      author={Buckmaster, Tristan},
      author={De~Lellis, Camillo},
      author={Sz\'{e}kelyhidi, L\'{a}szl\'{o}, Jr.},
      author={Vicol, Vlad},
       title={Onsager's conjecture for admissible weak solutions},
        date={2019},
        ISSN={0010-3640},
     journal={Comm. Pure Appl. Math.},
      volume={72},
      number={2},
       pages={229\ndash 274},
         url={https://doi.org/10.1002/cpa.21781},
      review={\MR{3896021}},
}

\bib{BuckmasterVicolAnnals}{article}{
      author={Buckmaster, Tristan},
      author={Vicol, Vlad},
       title={Nonuniqueness of weak solutions to the {N}avier-{S}tokes
  equation},
        date={2019},
        ISSN={0003-486X},
     journal={Ann. of Math. (2)},
      volume={189},
      number={1},
       pages={101\ndash 144},
         url={https://doi.org/10.4007/annals.2019.189.1.3},
      review={\MR{3898708}},
}

\bib{fried}{article}{
      author={Cheskidov, A.},
      author={Constantin, P.},
      author={Friedlander, S.},
      author={Shvydkoy, R.},
       title={Energy conservation and {O}nsager's conjecture for the {E}uler
  equations},
        date={2008},
        ISSN={0951-7715},
     journal={Nonlinearity},
      volume={21},
      number={6},
       pages={1233\ndash 1252},
         url={https://doi.org/10.1088/0951-7715/21/6/005},
      review={\MR{2422377}},
}

\bib{ciampa}{article}{
      author={Ciampa, Gennaro},
      author={Crippa, Gianluca},
      author={Spirito, Stefano},
       title={Smooth approximation is not a selection principle for the
  transport equation with rough vector field},
        date={2020},
        ISSN={0944-2669},
     journal={Calc. Var. Partial Differential Equations},
      volume={59},
      number={1},
       pages={Paper No. 13, 21},
         url={https://doi.org/10.1007/s00526-019-1659-0},
      review={\MR{4037474}},
}

\bib{ConstantinETiti}{article}{
      author={Constantin, Peter},
      author={E, Weinan},
      author={Titi, Edriss~S.},
       title={Onsager's conjecture on the energy conservation for solutions of
  {E}uler's equation},
        date={1994},
        ISSN={0010-3616},
     journal={Comm. Math. Phys.},
      volume={165},
      number={1},
       pages={207\ndash 209},
         url={http://projecteuclid.org/euclid.cmp/1104271041},
      review={\MR{1298949}},
}

\bib{constantinannals}{article}{
      author={Constantin, P.},
      author={Kiselev, A.},
      author={Ryzhik, L.},
      author={Zlato\v{s}, A.},
       title={Diffusion and mixing in fluid flow},
        date={2008},
        ISSN={0003-486X},
     journal={Ann. of Math. (2)},
      volume={168},
      number={2},
       pages={643\ndash 674},
         url={https://doi.org/10.4007/annals.2008.168.643},
      review={\MR{2434887}},
}

\bib{chesluo1}{article}{
      author={Cheskidov, Alexey},
      author={Luo, Xiaoyutao},
       title={Nonuniqueness of weak solutions for the transport equation at
  critical space regularity},
        date={2021},
        ISSN={2524-5317},
     journal={Ann. PDE},
      volume={7},
      number={1},
       pages={Paper No. 2, 45},
         url={https://doi.org/10.1007/s40818-020-00091-x},
      review={\MR{4199851}},
}

\bib{chesluo2}{misc}{
      author={Cheskidov, Alexey},
      author={Luo, Xiaoyutao},
       title={Extreme temporal intermittency in the linear sobolev transport:
  almost smooth nonunique solutions},
        date={2022},
}

\bib{Corrsin}{article}{
      author={Corrsin, Stanley},
       title={On the spectrum of isotropic temperature fluctuations in an
  isotropic turbulence},
        date={1951},
        ISSN={0021-8979},
     journal={J. Appl. Phys.},
      volume={22},
       pages={469\ndash 473},
      review={\MR{47458}},
}

\bib{procaccia2}{article}{
      author={Constantin, Peter},
      author={Procaccia, Itamar},
       title={The geometry of turbulent advection: sharp estimates for the
  dimensions of level sets},
        date={1994},
        ISSN={0951-7715},
     journal={Nonlinearity},
      volume={7},
      number={3},
       pages={1045\ndash 1054},
         url={http://stacks.iop.org/0951-7715/7/1045},
      review={\MR{1275539}},
}

\bib{CZ2020}{article}{
      author={Coti~Zelati, Michele},
       title={Stable mixing estimates in the infinite {P}\'{e}clet number
  limit},
        date={2020},
        ISSN={0022-1236},
     journal={J. Funct. Anal.},
      volume={279},
      number={4},
       pages={108562, 25},
         url={https://doi.org/10.1016/j.jfa.2020.108562},
      review={\MR{4095800}},
}

\bib{DCZ}{article}{
      author={Coti~Zelati, Michele},
      author={Drivas, Theodore~D.},
       title={A stochastic approach to enhanced diffusion},
        date={2021},
        ISSN={0391-173X},
     journal={Ann. Sc. Norm. Super. Pisa Cl. Sci. (5)},
      volume={22},
      number={2},
       pages={811\ndash 834},
      review={\MR{4288672}},
}

\bib{CZDE}{article}{
      author={Coti~Zelati, Michele},
      author={Delgadino, Matias~G.},
      author={Elgindi, Tarek~M.},
       title={On the relation between enhanced dissipation timescales and
  mixing rates},
        date={2020},
        ISSN={0010-3640},
     journal={Comm. Pure Appl. Math.},
      volume={73},
      number={6},
       pages={1205\ndash 1244},
         url={https://doi.org/10.1002/cpa.21831},
      review={\MR{4156602}},
}

\bib{Gautam}{article}{
      author={Drivas, Theodore~D.},
      author={Elgindi, Tarek~M.},
      author={Iyer, Gautam},
      author={Jeong, In-Jee},
       title={Anomalous dissipation in passive scalar transport},
        date={2022},
        ISSN={0003-9527},
     journal={Arch. Ration. Mech. Anal.},
      volume={243},
      number={3},
       pages={1151\ndash 1180},
         url={https://doi.org/10.1007/s00205-021-01736-2},
      review={\MR{4381138}},
}

\bib{Depauw}{article}{
      author={Depauw, Nicolas},
       title={Non unicit\'{e} des solutions born\'{e}es pour un champ de
  vecteurs {BV} en dehors d'un hyperplan},
        date={2003},
        ISSN={1631-073X},
     journal={C. R. Math. Acad. Sci. Paris},
      volume={337},
      number={4},
       pages={249\ndash 252},
         url={https://doi.org/10.1016/S1631-073X(03)00330-3},
      review={\MR{2009116}},
}

\bib{DPL89}{article}{
      author={DiPerna, R.~J.},
      author={Lions, P.-L.},
       title={Ordinary differential equations, transport theory and {S}obolev
  spaces},
        date={1989},
        ISSN={0020-9910},
     journal={Invent. Math.},
      volume={98},
      number={3},
       pages={511\ndash 547},
         url={https://doi.org/10.1007/BF01393835},
      review={\MR{1022305}},
}

\bib{Giri-DL}{article}{
      author={De~Lellis, Camillo},
      author={Giri, Vikram},
       title={Smoothing does not give a selection principle for transport
  equations with bounded autonomous fields},
        date={2022},
        ISSN={2195-4755},
     journal={Ann. Math. Qu\'{e}.},
      volume={46},
      number={1},
       pages={27\ndash 39},
         url={https://doi.org/10.1007/s40316-021-00160-y},
      review={\MR{4396067}},
}

\bib{DLL09}{article}{
      author={De~Lellis, Camillo},
      author={Sz\'{e}kelyhidi, L\'{a}szl\'{o}, Jr.},
       title={The {E}uler equations as a differential inclusion},
        date={2009},
        ISSN={0003-486X},
     journal={Ann. of Math. (2)},
      volume={170},
      number={3},
       pages={1417\ndash 1436},
         url={https://doi.org/10.4007/annals.2009.170.1417},
      review={\MR{2600877}},
}

\bib{DLL13}{article}{
      author={De~Lellis, Camillo},
      author={Sz\'{e}kelyhidi, L\'{a}szl\'{o}, Jr.},
       title={Dissipative continuous {E}uler flows},
        date={2013},
        ISSN={0020-9910},
     journal={Invent. Math.},
      volume={193},
      number={2},
       pages={377\ndash 407},
         url={https://doi.org/10.1007/s00222-012-0429-9},
      review={\MR{3090182}},
}

\bib{D05}{article}{
      author={Donzis, Diego~A.},
      author={Sreenivasan, Katepalli~R.},
      author={Yeung, P.~K.},
       title={Scalar dissipation rate and dissipative anomaly in isotropic
  turbulence},
        date={2005},
     journal={Journal of Fluid Mechanics},
      volume={532},
       pages={199 \ndash  216},
}

\bib{eyinksurvey}{article}{
      author={Eyink, Gregory~L.},
      author={Sreenivasan, Katepalli~R.},
       title={Onsager and the theory of hydrodynamic turbulence},
        date={2006},
        ISSN={0034-6861},
     journal={Rev. Modern Phys.},
      volume={78},
      number={1},
       pages={87\ndash 135},
         url={https://doi.org/10.1103/RevModPhys.78.87},
      review={\MR{2214822}},
}

\bib{Evans_SDEbook}{book}{
      author={Evans, Lawrence~C.},
       title={An introduction to stochastic differential equations},
   publisher={American Mathematical Society, Providence, RI},
        date={2013},
        ISBN={978-1-4704-1054-4},
         url={https://doi.org/10.1090/mbk/082},
      review={\MR{3154922}},
}

\bib{eyinkonsag}{article}{
      author={Eyink, Gregory~L.},
       title={Energy dissipation without viscosity in ideal hydrodynamics. {I}.
  {F}ourier analysis and local energy transfer},
        date={1994},
        ISSN={0167-2789},
     journal={Phys. D},
      volume={78},
      number={3-4},
       pages={222\ndash 240},
         url={https://doi.org/10.1016/0167-2789(94)90117-1},
      review={\MR{1302409}},
}

\bib{E96}{article}{
      author={Eyink},
       title={Intermittency and anomalous scaling of passive scalars in any
  space dimension.},
        date={1996},
     journal={Physical review. E, Statistical physics, plasmas, fluids, and
  related interdisciplinary topics},
      volume={54 2},
       pages={1497\ndash 1503},
         url={https://api.semanticscholar.org/CorpusID:1588394},
}

\bib{FGP}{article}{
      author={Flandoli, F.},
      author={Gubinelli, M.},
      author={Priola, E.},
       title={Well-posedness of the transport equation by stochastic
  perturbation},
        date={2010},
        ISSN={0020-9910},
     journal={Invent. Math.},
      volume={180},
      number={1},
       pages={1\ndash 53},
         url={https://doi.org/10.1007/s00222-009-0224-4},
      review={\MR{2593276}},
}

\bib{pulsed}{article}{
      author={Feng, Yuanyuan},
      author={Iyer, Gautam},
       title={Dissipation enhancement by mixing},
        date={2019},
        ISSN={0951-7715},
     journal={Nonlinearity},
      volume={32},
      number={5},
       pages={1810\ndash 1851},
         url={https://doi.org/10.1088/1361-6544/ab0e56},
      review={\MR{3942601}},
}

\bib{frisch}{book}{
      author={Frisch, Uriel},
       title={Turbulence},
   publisher={Cambridge University Press, Cambridge},
        date={1995},
        ISBN={0-521-45103-5},
        note={The legacy of A. N. Kolmogorov},
      review={\MR{1428905}},
}

\bib{HT23}{misc}{
      author={Huysmans, Lucas},
      author={Titi, Edriss~S.},
       title={Non-uniqueness and inadmissibility of the vanishing viscosity
  limit of the passive scalar transport equation},
        date={2023},
}

\bib{IsettOnsager}{article}{
      author={Isett, Philip},
       title={A proof of {O}nsager's conjecture},
        date={2018},
        ISSN={0003-486X},
     journal={Ann. of Math. (2)},
      volume={188},
      number={3},
       pages={871\ndash 963},
         url={https://doi.org/10.4007/annals.2018.188.3.4},
      review={\MR{3866888}},
}

\bib{IS18}{article}{
      author={Iyer, Kartik~P.},
      author={Schumacher, J\"org},
      author={Sreenivasan, Katepalli~R.},
      author={Yeung, P.~K.},
       title={Steep cliffs and saturated exponents in three-dimensional scalar
  turbulence},
        date={2018Dec},
     journal={Phys. Rev. Lett.},
      volume={121},
       pages={264501},
         url={https://link.aps.org/doi/10.1103/PhysRevLett.121.264501},
}

\bib{IS20}{article}{
      author={Iyer, Kartik~P.},
      author={Sreenivasan, Katepalli~R.},
      author={Yeung, P.~K.},
       title={Scaling exponents saturate in three-dimensional isotropic
  turbulence},
        date={2020May},
     journal={Phys. Rev. Fluids},
      volume={5},
       pages={054605},
         url={https://link.aps.org/doi/10.1103/PhysRevFluids.5.054605},
}

\bib{jiasverakselfsim}{article}{
      author={Jia, Hao},
      author={{\v S}ver{\'a}k, Vladim{\'i}r},
       title={Local-in-space estimates near initial time for weak solutions of
  the {N}avier-{S}tokes equations and forward self-similar solutions},
        date={2014},
        ISSN={0020-9910},
     journal={Invent. Math.},
      volume={196},
      number={1},
       pages={233\ndash 265},
         url={http://dx.doi.org/10.1007/s00222-013-0468-x},
      review={\MR{3179576}},
}

\bib{jiasverakillposed}{article}{
      author={Jia, Hao},
      author={{\v S}ver{\'a}k, Vladim{\'i}r},
       title={Are the incompressible 3d {N}avier-{S}tokes equations locally
  ill-posed in the natural energy space?},
        date={2015},
        ISSN={0022-1236},
     journal={J. Funct. Anal.},
      volume={268},
      number={12},
       pages={3734\ndash 3766},
         url={http://dx.doi.org/10.1016/j.jfa.2015.04.006},
      review={\MR{3341963}},
}

\bib{OB}{book}{
      author={{{\O}}ksendal, Bernt},
       title={Stochastic differential equations},
     edition={Sixth},
      series={Universitext},
   publisher={Springer-Verlag, Berlin},
        date={2003},
        ISBN={3-540-04758-1},
         url={https://doi.org/10.1007/978-3-642-14394-6},
        note={An introduction with applications},
      review={\MR{2001996}},
}

\bib{kolmo}{article}{
      author={Kolmogoroff, A.~N.},
       title={Dissipation of energy in the locally isotropic turbulence},
        date={1941},
     journal={C. R. (Doklady) Acad. Sci. URSS (N.S.)},
      volume={32},
       pages={16\ndash 18},
      review={\MR{0005851}},
}

\bib{KR05}{article}{
      author={Krylov, N.~V.},
      author={R\"{o}ckner, M.},
       title={Strong solutions of stochastic equations with singular time
  dependent drift},
        date={2005},
        ISSN={0178-8051},
     journal={Probab. Theory Related Fields},
      volume={131},
      number={2},
       pages={154\ndash 196},
         url={https://doi.org/10.1007/s00440-004-0361-z},
      review={\MR{2117951}},
}

\bib{K84}{incollection}{
      author={Kunita, H.},
       title={Stochastic differential equations and stochastic flows of
  diffeomorphisms},
        date={1984},
   booktitle={\'{E}cole d'\'{e}t\'{e} de probabilit\'{e}s de {S}aint-{F}lour,
  {XII}---1982},
      series={Lecture Notes in Math.},
      volume={1097},
   publisher={Springer, Berlin},
       pages={143\ndash 303},
         url={https://doi.org/10.1007/BFb0099433},
      review={\MR{876080}},
}

\bib{LL19}{book}{
      author={Le~Bris, Claude},
      author={Lions, Pierre-Louis},
       title={Parabolic equations with irregular data and related issues:
  Applications to stochastic differential equations},
   publisher={De Gruyter},
        date={2019},
        ISBN={9783110635508},
         url={https://doi.org/10.1515/9783110635508},
}

\bib{Modena1}{article}{
      author={Modena, Stefano},
      author={Sz\'{e}kelyhidi, L\'{a}szl\'{o}, Jr.},
       title={Non-uniqueness for the transport equation with {S}obolev vector
  fields},
        date={2018},
        ISSN={2524-5317},
     journal={Ann. PDE},
      volume={4},
      number={2},
       pages={Paper No. 18, 38},
         url={https://doi.org/10.1007/s40818-018-0056-x},
      review={\MR{3884855}},
}

\bib{Modena2}{article}{
      author={Modena, Stefano},
      author={Sz\'{e}kelyhidi, L\'{a}szl\'{o}, Jr.},
       title={Non-renormalized solutions to the continuity equation},
        date={2019},
        ISSN={0944-2669},
     journal={Calc. Var. Partial Differential Equations},
      volume={58},
      number={6},
       pages={Paper No. 208, 30},
         url={https://doi.org/10.1007/s00526-019-1651-8},
      review={\MR{4029736}},
}

\bib{Modena3}{article}{
      author={Modena, Stefano},
      author={Sattig, Gabriel},
       title={Convex integration solutions to the transport equation with full
  dimensional concentration},
        date={2020},
        ISSN={0294-1449},
     journal={Ann. Inst. H. Poincar\'{e} Anal. Non Lin\'{e}aire},
      volume={37},
      number={5},
       pages={1075\ndash 1108},
         url={https://doi.org/10.1016/j.anihpc.2020.03.002},
      review={\MR{4138227}},
}

\bib{Obukhov}{article}{
      author={Obukhov, A.~M.},
       title={The structure of the temperature field in a turbulent flow},
        date={1949},
     journal={Izvestiya Akad. Nauk SSSR. Ser. Geograf. Geofiz.},
      volume={13},
       pages={58\ndash 69},
      review={\MR{0034164}},
}

\bib{onsager}{article}{
      author={Onsager, L.},
       title={Statistical hydrodynamics},
        date={1949},
        ISSN={0029-6341},
     journal={Nuovo Cimento (9)},
      volume={6},
      number={Supplemento, 2 (Convegno Internazionale di Meccanica
  Statistica)},
       pages={279\ndash 287},
      review={\MR{36116}},
}

\bib{P94}{article}{
      author={Pierrehumbert, R.T.},
       title={Tracer microstructure in the large-eddy dominated regime},
        date={1994},
        ISSN={0960-0779},
     journal={Chaos, Solitons and Fractals},
      volume={4},
      number={6},
       pages={1091\ndash 1110},
  url={https://www.sciencedirect.com/science/article/pii/0960077994901392},
        note={Special Issue: Chaos Applied to Fluid Mixing},
}

\bib{sree2}{article}{
      author={Sreenivasan, Katepalli~R.},
       title={Turbulent mixing: a perspective},
        date={2019},
        ISSN={0027-8424},
     journal={Proc. Natl. Acad. Sci. USA},
      volume={116},
      number={37},
       pages={18175\ndash 18183},
         url={https://doi.org/10.1073/pnas.1800463115},
      review={\MR{4012539}},
}

\bib{sree1}{article}{
      author={Sreenivasan, Katepalli~R.},
       title={The passive scalar spectrum and the {O}bukhov-{C}orrsin
  constant},
        date={1996},
        ISSN={1070-6631},
     journal={Phys. Fluids},
      volume={8},
      number={1},
       pages={189\ndash 196},
         url={https://doi.org/10.1063/1.868826},
      review={\MR{1367065}},
}

\bib{S00}{article}{
      author={Shraiman, Boris~I.},
      author={Siggia, Eric~D.},
       title={Scalar turbulence},
        date={2000},
     journal={Nature},
      volume={405},
       pages={639\ndash 646},
}

\bib{emil_notes}{misc}{
      author={Wiedemann, Emil},
       title={Conserved quantities and regularity in fluid dynamics},
        date={2020},
}

\bib{yaglom}{article}{
      author={Yaglom, A.~M.},
       title={On the local structure of the temperature field in a turbulent
  fluid},
        date={1949},
     journal={Doklady Akad. Nauk SSSR (N.S.)},
      volume={69},
       pages={743\ndash 746},
      review={\MR{0034165}},
}

\bib{YZ}{article}{
      author={Yao, Yao},
      author={Zlato\v{s}, Andrej},
       title={Mixing and un-mixing by incompressible flows},
        date={2017},
        ISSN={1435-9855},
     journal={J. Eur. Math. Soc. (JEMS)},
      volume={19},
      number={7},
       pages={1911\ndash 1948},
         url={https://doi.org/10.4171/JEMS/709},
      review={\MR{3656475}},
}

\end{biblist}
\end{bibdiv}
 
\end{document}